\documentclass{birkjour}
\usepackage{amsmath, amssymb}
\usepackage{amsthm}
\newtheorem{theorem}{Theorem}[section]
\newtheorem{lemma}[theorem]{Lemma}
\newtheorem{corollary}[theorem]{Corollary}
\newtheorem{proposition}[theorem]{Proposition}
\theoremstyle{definition} 
\newtheorem{definition}[theorem]{Definition}
\theoremstyle{remark}
\newtheorem{remark}[theorem]{Remark}
\newtheorem{example}{Example}
\numberwithin{equation}{section}

\begin{document}

%---------------------------------------------------------------------------
%
\title{$k$-slant distributions}

\author{Dan Radu La\c tcu}

\address{Central University Library of Timi\c soara\\ 
300223 Timi\c soara, Romania} 
\email{latcu07@yahoo.com}


 
\date{}
%%% ----------------------------------------------------------------------

\begin{abstract}
Inspired by the concepts of slant distribution and slant submanifold, with their variants of hemi-slant, semi-slant, \mbox{bi-slant,} or almost bi-slant, we introduce the more general concepts of \mbox{$k$-slant} distribution and $k$-slant submanifold in the settings of an almost Hermitian, an almost product Riemannian, an almost contact metric, and an almost paracontact metric manifold and study some of their properties. We prove that, for any proper $k$-slant distribution in the tangent bundle of a Riemannian manifold, there exists another one in its orthogonal complement, and we establish basic relations (metric properties, formulae relating the involved tensor fields, conformal properties) between them. Furthermore, allowing the slant angles to depend on the points of the manifold, we generalize these concepts and those of pointwise slant distribution and pointwise slant submanifold to the concepts of $k$-pointwise slant distribution and $k$-pointwise slant submanifold in the above-mentioned settings. For any $k$-pointwise slant distribution, we prove the existence of a corresponding one in its orthogonal complement and reveal basic relations between them. We also provide sufficient conditions for $k$-pointwise slant distributions to become $k$-slant distributions and establish other related results. By the end, for the fulfilment of some specific requirements, we introduce a special class of \mbox{$k$-point}\-wise slant distributions, that of pointwise $k$-slant distributions, and the corresponding class of submanifolds, pointwise $k$-slant submanifolds, which is slightly more general than the class of generic submanifolds in sense of Ronsse, getting new results. 

\end{abstract}
%%% 
\maketitle

%%%%%%%%%%%%%%%%%%%%%%%%%%
\section{Introduction}\label{intro}

The theory of submanifolds isometrically immersed into smooth manifolds carrying different geometric structures, such as almost complex, almost product, almost contact, almost paracontact, has been continuously developed in the last half-century and has become of intensive study in the last twenty years. A fact that motivated such an interest was the introduction over time of special types of submanifolds. In the sequel, we will point out only some of the most significant moments of this development as far as they are related to the present study. 

In 1978, Bejancu \cite{bejancu CR} introduced the notion of CR (or semi-invariant) sub\-man\-i\-fold for almost Hermitian manifolds, integrating the concepts of totally real (or anti-invariant) and holomorphic (or invariant) submanifold in a single one. Later, this notion was extended by Bejan \cite{bejan} to that of almost semi-invariant submanifold. 

In 1990, Chen \cite{chen} introduced for an isometric immersion into an almost Hermitian manifold the notion of general slant immersion, requiring for the structural endomorphism to make a constant angle of any value in $[0,\frac{\pi}{2}]$ with the tangent space of the submanifold, thus generalizing the notions of holo\-mor\-phic and totally real submanifold. This immersion and the corresponding submanifold were simply named slant immersion and slant submanifold, respectively, when they are not holomorphic, i.e., the constant angle between the structural endomorphism and the tangent space of the submanifold is different from zero. In this case, the angle was called the slant angle of the slant immersion. Since then, the study of slant sub\-man\-i\-folds with respect to different structures substantially evolved (e.g., \cite{alegre, blaga, cabrerizo, chen2, lotta, papaghiuc, se, sahin, tastan}).

In 1990, Chen \cite{chen2, chen} and Ronsse \cite{ronsse} considered the or\-thog\-o\-nal decomposition of the tangent spaces of a submanifold of an almost Hermitian manifold into the direct sum of the eigenspaces corresponding to the square of the tangential component of the structural tensor field. Imposing some conditions to ensure the existence of regular distributions in the tangent bundle, the submanifold was called by Ronsse a generic submanifold or, in more restrictive conditions, a skew CR submanifold. Properties of the above decomposition have been studied by the two authors. Avoiding supplementary considerations, we will further outline the main ideas they used for defining these notions. 

Considering $M$ an immersed submanifold of an almost Hermitian man\-i\-fold $(\overline{M},\varphi,g)$, we have the orthogonal decomposition of the tangent space of $\overline{M}$ at a point $x\in M$ into the tangent space $T_xM$ of $M$ at $x$ and its orthogonal complement $(T_xM)^{\perp}$ in $T_x\overline{M}$, 
$$T_x\overline{M}=T_xM\oplus (T_xM)^{\perp},$$ 
and we denote 
$$\varphi_x X_x=f_x X_x+ w_x X_x$$
for any tangent vector $X_x\in T_xM$, where $f_x X_x\in T_xM$ and $w_x X_x\in (T_xM)^{\perp}$.  

Since $(\overline{M},\varphi,g)$ is an almost Hermitian manifold, we have
$$g(f_x X_x,Y_x)+g(X_x,f_x Y_x)=0$$
for any $X_x,Y_x$ tangent vectors to $M$ at $x$. 

Since $f$ is skew-symmetric, and, hence, $f^2$ is symmetric, all the eigenvalues $\lambda_i(x)$ of $f_x^2$ with respect to $x$ are real and lie in $[-1,0]$. If $\lambda_i(x)\neq0$, the cor\-re\-spond\-ing eigenspace 
$D_x^i$ is of even dimension and is invariant under $f_x$. 
Each tangent space $T_xM$ of $M$ at $x$ admits the following orthogonal decomposition into the eigenspaces $D_x^i$ of $f_x^2$: 
$$T_xM={D_x^1}\oplus\ldots\oplus{D_x^{k(x)}}.$$

Furthermore, denoting $\lambda_i(x)=-\alpha_i(x)^2$, with $\alpha_i(x)\in [0,1]$, we get $\alpha_i(x)=\cos\theta_i(x)$, where $\theta_i(x)$ is the angle between $\varphi_x X_x$ and $T_xM$ for any nonzero $X_x\in D_x^i$, $i=\overline{1,k(x)}$. 

Denoting by ${D_x^{\alpha}}$ the eigenspace corresponding to the eigenvalue \linebreak $\lambda(x)= -\alpha(x)^2$, where $\alpha(x)\in [0,1]$, Ronsse considered in the K\"ahlerian case the following definition. 

\begin{definition}\hspace{-2pt}\cite{ronsse}
A submanifold $M$ of a K\"ahler manifold $(\overline{M},\varphi,g)$ is called a \textit{generic submanifold} if there exists an integer $k$ and some real functions $\alpha_i:\nolinebreak M\rightarrow(0,1)$, $i=\overline{1,k}$, such that: 
\begin{enumerate}
\item each $-\alpha_i^2(x)$, for $i=\overline{1,k}$, is a distinct eigenvalue of $f^2_x$, and \newline 
$T_xM={D_x^0}\oplus{D_x^1}\oplus{D_x^{\alpha_1}}\oplus\ldots\oplus{D_x^{\alpha_k}}$ for $x\in M$; 
\item the dimensions of ${D_x^0}$, ${D_x^1}$, ${D_x^{\alpha_i}}$,  $i=\overline{1,k}$, are independent of $x\in M$. 
\end{enumerate}
In addition, if each $\alpha_i$ is constant on $M$, then $M$ is called a \textit{skew CR submanifold}. 

For the sake of generality, ${D_x^0}$ and ${D_x^1}$ are allowed to be the null space $\{0\}$, but $D_x^{\alpha_i}$ is not null for $i=\overline{1,k}$. 

If $k=0$, then $M$ is called a \textit{CR submanifold}, and, if $k=0$ and ${D_x^0}$, ${D_x^1}$ are non-null (i.e., at least 1-di\-men\-sion\-al), then $M$ is called a \textit{proper CR submanifold}. 
\end{definition}

\begin{remark}\label{p204}
Due to the existence of at least one function $\alpha_i: M\rightarrow(0,1)$ in the definition of a generic submanifold, a CR submanifold is neither a skew CR nor a generic submanifold. 
\end{remark}

The notion of slant submanifold was generalized by Etayo \cite{etayo} to that of quasi-slant submanifold. He called a submanifold $N$ of the almost Hermitian manifold $(M,\varphi,g)$ a quasi-slant submanifold if, for any point $p\in N$ and any $X \in T_p N\verb=\=\{0\}$, the angle $\theta (p)$ between $\varphi_p X$ and the tangent space $T_p N$ depends only on the point $p\in N$ and not on the nonzero tangent vector $X$. The slant angle became so a slant function. Later, the name of the submanifold changed to that of pointwise slant submanifold \cite{chen-garay}. 
After that, different variants of pointwise slant submanifolds (semi-slant, hemi-slant, bi-slant) in different settings have been investigated (e.g., \cite{gulbahar, kazemi, pahan, park}).

\medskip 
In the present paper, we introduce the $k$-slant and $k$-pointwise slant concepts for smooth regular distributions and submanifolds, where $k\in \mathbb{N}^*$, which together enclose the above-mentioned notions in more general frameworks. At the beginning, we define and study the notion of $k$-slant distribution in the tangent bundle of a Riemannian manifold $(\overline{M},g)$ endowed with a $g$-compatible $(1,1)$-tensor field. Correspondingly, we introduce the notion of \mbox{$k$-slant} submanifold, which includes that of CR and skew CR submanifold. For a unitary treatment of the almost contact metric and almost paracontact metric structures, we will make use of an almost ($\epsilon$)-contact metric structure, where $\epsilon\in \{-1,1\}$, which, for the values $-1$ or $1$ of $\epsilon$, corresponds to one or the other of the two structures. We will proceed similarly to the unitary treatment of the almost Hermitian and almost product Riemannian structures. We also describe the concept of skew CR submanifold for each of the settings considered in the paper, showing the relation between this concept and that of $k$-slant submanifold. 
Further, we introduce and study the notion of \mbox{$k$-point}\-wise slant distribution and, correspondingly, that of $k$-pointwise slant submanifold. In addition, we describe the notion of generic submanifold, in sense of Ronsse \cite{ronsse}, in the almost Hermitian, almost product Riemannian, almost contact metric, and almost paracontact metric setting, showing that the $k$-pointwise slant framework is more general than the generic one in each of the considered settings, illustrating this through examples. 
 
In short, we initiate the study of algebraic and geometric properties of \mbox{$k$-slant} and $k$-pointwise slant distributions in the settings of almost contact metric, almost paracontact metric, almost Hermitian, and almost product Riemannian geometry. 

The paper is mainly structured in two parts. After a short introduction, it follows a section of general considerations regarding definitions, properties, and notations to be used. The main components of the first part of the paper are sections \ref{alm_cont} and \ref{alm_hermitian}, which treat with $k$-slant distributions and \mbox{$k$-slant} submanifolds in an almost contact metric, an almost paracontact metric, an almost Hermitian, or an almost product Riemannian man\-i\-fold.  
Emphasizing the slant properties of the distributions and the correspondence between these, we prove that to any $k$-slant distribution it corresponds another \mbox{$k$-slant} distribution in its orthogonal complement and establish some relations between their components, such as that regarding the dimensions of the slant distributions or that regarding the angles. New properties and formulas related to the structural endomorphism or to certain pairs of vector fields are ob\-tained. In particular, we identify some conformal properties. 

The second part is devoted to the introduction and study of $k$-pointwise slant distributions and $k$-pointwise slant submanifolds. It reveals properties and results corresponding to those in the first part. In short, its first segment, section \ref{pointwise_gen_consid}, consists of general considerations related to the pointwise slant framework. This is followed by sections \ref{pointwise_alm_cont} and \ref{pointwise_alm_hermitian}, which deal with the study of $k$-pointwise slant distributions in an almost contact metric, an almost paracontact metric, an almost Hermitian, and an almost product Riemannian manifold. In section \ref{slant_via_pointwise_slant}, we search for sufficient conditions for a $k$-pointwise slant distribution to be a $k$-slant distribution. On the way, we get a lot of results in connection to this. By the end, related to the investigated problem, a special class of $k$-pointwise slant distributions, that of pointwise \mbox{$k$-slant} distributions, and the corresponding class of submanifolds, that of pointwise \mbox{$k$-slant} submanifolds (which is slightly more general than the class of generic submanifolds), are introduced, and corresponding properties are revealed. 

For every different setting, both in the first and in the second part, we explain how the results got for a type of distributions are transferred to the same type of submanifolds. 

The absence of any supplementary conditions leads to a sufficiently large gen\-er\-al\-i\-ty of the obtained results.

%%%%%%%%%%%%%%%%%%%%%%%%%%%
\section{General considerations}\label{gen_consid}

We will adopt throughout the paper the following definitions and notations. 

For any manifold $M$, we will denote by $TM$ the set of all smooth vector fields on $M$. Moreover, all the manifolds and vector fields considered as well as the Riemannian metric are assumed to be smooth. By regular distribution we mean a smooth regular distribution, that is, a smooth distribution for which any local basis has the same dimension. 

The localization of a vector field $Z$, of a distribution $D$, of a Riemannian metric $g$, or of an arbitrary tensor field $\psi$ in a point $x$ will be denoted by $Z_x$, $D_x$, $g_x$, and $\psi_x$, respectively. 

On a Riemannian manifold $(\overline{M},g)$, for $\epsilon \in \{-1,1\}$, we consider a \linebreak$(1,1)$-tensor field $\varphi$ such that
\begin{equation}\label{1}
g(\varphi X,Y)=\epsilon g(X,\varphi Y) 
\text{ for any }X,Y\in T\overline{M}.
\end{equation}

For any distribution $D$ on $\overline{M}$, we denote by $D^{\perp}$ the orthogonal complement of $D$ in $T\overline{M}$; hence, we have the or\-thog\-o\-nal de\-com\-po\-si\-tion
$$T\overline{M}=D\oplus D^{\perp}.$$
For any $Z\in T\overline{M}$, let $fZ$ and $wZ$ denote the components of $\varphi Z$ in $D$ and in $D^{\perp}$, respectively, calling $f$ the \textit{component of $\varphi$ into $D$}. Also, we denote by $|Z|$ the real nonnegative function defined on $\overline{M}$ by $x\mapsto \sqrt{g_x(Z_x,Z_x)}$ and the norm of the tangent vector $Z_x$ by $\|Z_x\|$. The notation $\langle Z\rangle$ will represent the $C^\infty (\overline{M})$-module generated by $Z$. 
We will call dimension of the regular distribution $D$, and we will denote it by $\dim (D)$, the dimension of the vector space $D_x$, where $x$ is an arbitrary point in $\overline{M}$. 

For any $1$-form or $(1,1)$-tensor field $\psi$ on $\overline{M}$, we denote by $\ker \psi$ the distribution consisting of all smooth vector fields $X\in T\overline{M}$ for which $\psi X=\nolinebreak 0$. For any $x\in \overline{M}$ and any tangent vector $v$ of $\overline{M}$ in $x$, by $\psi v$ we will actually mean $\psi_x v$. For any immersed submanifold $M$ of $\overline{M}$, the points of $M$ will be identified with the corresponding points of $\overline{M}$ through the immersion. For any $X, Y\in TM$, by $\psi Y$ and $g(X,Y)$ we mean the ''localization'' of the tensor field $\psi$ on $M$, $\psi_M=\{\psi_x\}_{ x\in M}$, applied to $Y$ (i.e., $\{\psi_x(Y_x)\}_{x\in M}$) and the family $\{g_x(X_x, Y_x)\}_{x\in M}$, or, equivalently, the functions defined on ${M}$ by $x\mapsto \psi_x(Y_x)$ and $x\mapsto {g_x(X_x,Y_x)}$. Also, the same meaning will be assigned to $\psi Y$ and $g(X,Y)$ if $X=\{X_x\}_{x\in M}$ and $Y=\{Y_x\}_{x\in M}$, with $X_x,Y_x \in T_x\overline{M}$ for $x\in M$, are smooth families (with respect to $x\in M$) of tangent vectors of $\overline{M}$. 
For such a $Y$, in particular for $Y\in TM$, the notation $|Y|$ will represent the real nonnegative function defined by $x\mapsto \sqrt{g_x(Y_x,Y_x)}$ for $x\in {M}$. Consequently, the notation $|\psi Y|$ will inherit the same meaning.

Let $D$ be a distribution on $\overline{M}$. We immediately notice: 

\begin{lemma}\label{lema1}
For any $X,Y \in D$ and $U,V\in D ^{\bot}$, we have:
\begin{align}
g(X,fY) &=\epsilon g(fX,Y), \nonumber\\ 
g(X,fU) &=\epsilon g(wX,U), \nonumber\\ 
g(U,wV) &=\epsilon g(wU,V). \nonumber
\end{align}
\end{lemma}

\begin{lemma}\label{lema2}
For any $X,Y \in D$ and $U,V\in D^{\bot}$, we have: 
\begin{align}
g(f^2X,Y)&=\epsilon g(fX,fY)=g(X,f^2Y), \nonumber\\ 
g(fwX,Y)&=\epsilon g(wX,wY)=g(X,fwY), \nonumber\\ 
g(wfU,V)&=\epsilon g(fU,fV)=g(U,wfV), \nonumber\\ 
g(w^2U,V)&=\epsilon g(wU,wV)=g(U,w^2V), \nonumber\\ 
g(wfX,U)&=\epsilon g(fX,fU)=g(X,f^2U), \nonumber\\ 
g(w^2X,U)&=\epsilon g(wX,wU)=g(X,fwU). \nonumber 
\end{align}

\end{lemma}

Relative to any immersed submanifold $M$ of $\overline{M}$, we consider the following or\-thog\-o\-nal decomposition: 
$$T\overline{M}=TM\oplus (TM)^{\perp}.$$
Using the same notation as above, for any $Z\in TM$, we will denote by $fZ$ and $wZ$ the components of $\varphi Z$ in $TM$ and in $(TM)^{\perp}$, respectively, if there is no other distribution $D$ in the context. In this case, we will call $f$ and $w$ the \textit{tangential} and the \textit{orthogonal} \textit{component} of $\varphi$ with respect to $M$, respectively. 

For any distribution $D$ on ${M}$, we denote by $D^{\perp}$ the orthogonal complement of $D$ in $T\overline{M}$; hence, we have the following or\-thog\-o\-nal de\-com\-po\-si\-tion: 
$$T\overline{M}=D\oplus D^{\perp}.$$
For any $Z\in T{M}$ and, more general, for any smooth family $Z=\{Z_x\}_{x\in M}$ of tangent vectors of $\overline{M}$ in $x\in M$, respectively ($Z_x \in T_x\overline{M}$ for $x\in M$), we will denote by $fZ$ and $wZ$ the components of $\varphi Z$ in $D$ and in $D^{\perp}$, respectively, calling $f$ the \textit{component of $\varphi$ into $D$}. 
Also, for $Z\in T\overline{M}$, we will denote by $Z_M$ the smooth family $\{Z_x\}_{x\in M}$ and by $fZ$ and $wZ$ the components of $\varphi Z_M$ in $D$ and in $D^{\perp}$, respectively. 

\begin{remark}\label{p194}
Lemmas \ref{lema1} and \ref{lema2} are also true if $D$ is a distribution on $M$, where $M$ is an immersed submanifold of $\overline M$. 
\end{remark}

Throughout this section, we will consider $M$ to be $\overline{M}$ or an immersed submanifold of $\overline{M}$ if not specified otherwise. 

\begin{definition}\label{5} Let $D$ be a non-null distribution (i.e., $D\neq \{0\}$) on $M$. \\ 
\hspace*{7pt} (i) We will say that a vector field $Z$ on $M$ or $\overline{M}$ and the distribution $D$ \textit{make an angle} $\theta\in [0,\frac{\pi}{2}]$ and will denote this by $\widehat{(Z, D)}=\theta$ if there is $x\in\nolinebreak M$ with $Z_x\neq 0$, and, for any such $x$, the angle between $Z_x$ and the vector space $D_x$ is equal to $\theta$. \\ 
\hspace*{7pt} (ii)\,\cite{chen, papaghiuc, cabrerizo} The distribution $D$ is called a \textit{slant distribution} if, for any $x \in {M}$ and $v \in D_x\verb=\=\{0\}$, we have $\varphi_x v\neq 0$, and the angle between $\varphi_x v$ and the vector space $D_x$ is nonzero and does not depend on $x$ or $v$. 
Denoting this angle by $\theta$ and calling it \textit{slant angle}, we will also call the distribution a \textit{$\theta$-slant distribution}. \\ 
\hspace*{7pt} (iii) The distribution $D$ is called \textit{invariant} if, for any $x \in {M}$ and $v \in\nolinebreak D_x$, we have $\varphi_x v\in D_x$. 
\end{definition}

\begin{remark}\label{p74}
If $D$ is a $\theta$-slant distribution, then $\varphi X$ and the distribution $D$ make an angle $\theta$ for any vector field $X\in D\verb=\=\{0\}$. 
\end{remark}

\begin{remark}\label{p59}
Obviously, the direct sum $D_1\oplus D_2$ of two orthogonal invariant distributions, $D_1,\,D_2$, on ${M}$ is an invariant distribution. 
\end{remark}

\begin{definition}\label{p193}
We will say that \textit{the orthogonality of vector fields on $M$} (or \textit{from $TM$}) \textit{is invariant under $\varphi$} if, for any two orthogonal vector fields $X,Y\in \nolinebreak T{M}$, we have $g_x (\varphi_x X_x, \varphi_x Y_x)=0$ for any $x\in M$, which will be denoted by $\varphi X\perp \varphi Y$. 

For $D$ a distribution on $M$, we will say that \textit{the orthogonality of vector fields from $D$ is invariant under $\varphi$} if, for any two orthogonal vector fields $X,Y\in D$, we have $\varphi X\perp \varphi Y$. 
\end{definition}

\begin{remark}
If $\varphi$ acts isometrically on a distribution $D$, then  the orthogonality of vector fields from $D$ is invariant under $\varphi$. 
\end{remark}

\begin{proposition}\label{p31}
Let $D_1,\,D_2$ be two orthogonal slant distributions on ${M}$ such that $D_1,\,D_2$ have the same slant angle $\theta$. Denoting, for any $Z\in T{M}$, by $fZ$ the component of $\varphi Z$ in $D_1\oplus D_2$, assume that: \\ 
\hspace*{7pt} i)\; the orthogonality of vector fields from $D_1\oplus D_2$ is invariant under $\varphi$;\\ 
\hspace*{7pt} ii) $f(D_i)\subseteq D_i$, $i=1,2$.\\ 
Then, the two slant distributions, $D_1,\,D_2$, can be joined into a single slant distribution with slant angle $\theta $.
\end{proposition}

\begin{proof} 
It is enough to check that, for arbitrary $x\in {M}$, $v_1\in (D_1)_x\verb=\=\{0\}$, and $v_2\in (D_2)_x\verb=\=\{0\}$, the tangent vector $\varphi_x (v_1+v_2)$, which is nonzero,  makes the angle $\theta $ with $(D_1\oplus D_2)_x$, i.e., 
$\|f_x (v_1+v_2)\|^2=
\cos^2\theta \|\varphi_x (v_1+v_2)\|^2$. First, we notice that $g(f_x v_1,f_x v_2)=0$.
Hence, 
\begin{align} 
\|f_x (v_1+v_2)\|^2=\|f_x v_1\|^2+\|f_x v_2\|^2 &=
\cos^2\theta\|\varphi_x v_1\|^2+\cos^2\theta\|\varphi_x v_2\|^2 \nonumber\\ 
&= \cos^2\theta\|\varphi_x (v_1+v_2)\|^2. \nonumber \qedhere 
\end{align}
\end{proof} 

\medskip 
Taking into account Proposition \ref{p31} and Remark \ref{p59}, we get 

\begin{corollary}\label{p197}
Let $L_1$, $L_2$, \ldots , $L_m$ be mutually orthogonal distributions on ${M}$, invariant with respect to $\tilde f$ (the component of $\varphi$ into $\oplus_{i=1}^m L_i$) which are slant (at least one) or invariant distributions (with respect to $\varphi$) such that the orthogonality of vector fields from $\oplus_{i=1}^m L_i$ is invariant under $\varphi$. Then, the direct sum $\oplus_{i=1}^m L_i$ can be represented as an orthogonal sum of slant distributions with distinct slant angles and at most one invariant distribution. 
\end{corollary}

We are now ready to provide the definition of a $k$-slant distribution. 

\begin{definition}\label{8} 
Let $k\in \mathbb{N}^*$. We will call the distribution $D$ on ${M}$ a \textit{\mbox{$k$-slant} distribution} if there exists an orthogonal decomposition of $D$ into regular distributions, 
$$D=\oplus_{i=0}^k{D_i}$$
with $D_i$ non-null for $i=\overline{1,k}$ and $D_0$ possible null (i.e., $D_0= \{0\}$), and there exist distinct values $\theta_i\in (0,\frac{\pi}{2}]$, $i=\overline{1,k}$, such that: \\ 
\hspace*{7pt} (i) \ $D_i$ is a $\theta_i$-slant distribution, $i=\overline{1,k}$;\\ 
\hspace*{7pt} (ii)\, $\varphi X\in D_0$ for any $X\in D_0$ 
(i.e., $\widehat{(\varphi X, D)}=0=:\theta_0$ for $X\in D_0$ with $\varphi X\neq 0$, and $f(D_0)\subseteq D_0$);\\ 
\hspace*{7pt} (iii) $f(D_i)\subseteq D_i$, $i=\overline{1,k}$. 

\medskip
We will say that $D$ is a \textit{multi-slant distribution} for $k\geq 2$. 

If we want to specify the values of the slant angles, we will say that $D$ is a \textit{$(\theta_1,\theta_2,\ldots,\theta_k)$-slant distribution}. 

We will call $D_0$ the \textit{invariant component} and $\oplus_{i=1}^kD_i$ the \textit{proper \mbox{$k$-slant} component} of $D$. 

The distribution $D=\oplus_{i=0}^kD_i$ will be called a \textit{proper $k$-slant distribution} if $D_0=\{0\}$. 
\end{definition}

\begin{remark}\label{6}
In view of (iii), we  notice that (i) is equivalent to \\ 
\hspace*{7pt} (i') For any $i\in \{1,\ldots,k\}$, $x\in {M}$, and $v\in (D_i)_x\verb=\=\{0\}$, we have $\varphi v \neq 0$ and 
$ \widehat{(\varphi v, D_x)}=\theta_i$. 
\end{remark}	

\begin{remark}\label{p196}
In view of Corollary \ref{p197} and Remark \ref{p59}, any orthogonal sum $D=\oplus_{i=0}^k{D_i}$ of invariant and slant (at least one) distributions on $M$ which are invariant with respect to $f$ such that the orthogonality of vector fields from $D$ is invariant under $\varphi$ can be represented as a $k'$-slant distribution, where $1\leq k'\leq k$. If there is no invariant component in $D$, then $D$ can be represented as a proper $k'$-slant distribution. 
\end{remark}

\begin{proposition}\label{p167}
Let $k\in \mathbb{N}^*$ and $D$ be a non-null distribution on ${M}$ decomposable into an orthogonal sum of regular distributions, $D=\oplus_{i=0}^k{D_i}$ with $D_i\neq \{0\}$ for $i=\overline{1,k}$ and $D_0$ invariant (possible null). Let $pr_i$ denote the projection operator from $TM$ onto $D_i$ for $i=\overline{1,k}$. 
If $\varphi$ restricted to $\oplus_{i=1}^k{D_i}$ is an isometry, and $f(D_i)\subseteq D_i$ for $i=\overline{1,k}$, and there exist $k$ distinct values $\theta_i\in (0,\frac{\pi}{2}]$, $i=\overline{1,k}$, such that  
\begin{equation}\label{99}
f^2X=\epsilon \sum_{i=1}^k\cos^2\theta_i\cdot pr_iX \, \text{ for any } X\in \oplus_{i=1}^k{D_i}\,,
\end{equation}
then $D$ is a $k$-slant distribution with slant angles $\theta_i$ corresponding to $D_i$, $i=\overline{1,k}$. 
\end{proposition}

\begin{proof}
 $f$ satisfies (\ref{99}); hence, for any $i\in\{1,\ldots,k\}$ and $X_i\in D_i$, we get 
\begin{equation}\nonumber
f^2X_i=\epsilon \cos^2\theta_i\cdot X_i\,; 
\end{equation}
 thus, in view of Lemma \ref{lema1}, Remark \ref{p194}, and the fact that $\varphi|_{\oplus_{i=1}^k{D_i}}$ is an isometry, 
\begin{equation}\nonumber 
|fX_i|^2=\epsilon g(f^2 X_i,X_i)= \cos^2\theta_i \cdot |X_i|^2= \cos^2\theta_i \cdot |\varphi X_i|^2, 
\end{equation}
from which it results that $D_i$ is a slant distribution with slant angle $\theta_i$. 
Additionally, $D_0$ is invariant, hence the conclusion. 
\end{proof}

Let $D$ be an orthogonal sum of distributions on $M$, $D=\oplus_{i=0}^kD_i$ for some $k\in \mathbb{N}^*$, with $D_0$ invariant (with respect to $\varphi$) and the $D_i$'s ($i=\nolinebreak\overline{1,k}$) non-null slant distributions with different slant angles. If $D$ is a \mbox{$k$-slant} distribution on $M$, from Definition \ref{8} (iii), we get 
\begin{equation}\label{18}
\varphi(D_i)\perp D_j \ \ \text{for} \ \ i\neq j  \ \ \text{from} \ \ \{1,\ldots,k\}.
\end{equation}

Conversely, from (\ref{18}), we have $f(D_i)\perp D_j$ for any $i\neq j$ in $\{1,\dots,k\}$. 
For $X\in D_0$ and $Y\in D_i$ with $i\geq 1$, we have $\varphi X\perp Y$, and, in view of (\ref{1}), we get $X\perp \varphi Y$. We obtain $\varphi (D_i)\perp D_0$; hence, $f(D_i)\perp D_0$. Since $f(D_i)\subseteq \oplus_{j=0}^k{D_j}$, we get $f(D_i)\subseteq D_i$ for any $i\in \{1,\ldots,k\}$.
Therefore, 

\begin{remark}\label{p25}
Condition (iii) from Definition \ref{8} of a $k$-slant distribution can be replaced by \\ 
\hspace*{7pt} (iii') $\varphi(D_i)\perp D_j$ for any $i\neq j$ from $\{1,\ldots,k\}$.
\end{remark}

\begin{remark}\label{p198}
If, additionally, the orthogonality of vector fields from the proper $k$-slant distribution $\oplus_{i=1}^kD_i$ is invariant under $\varphi$, we get $g(\varphi X,\varphi Y)=\nolinebreak 0$ 
for any $X$ and $Y$ vector fields belonging to distinct distributions among $D_1,\ldots,D_k$; hence, 
\begin{equation}\label{15}
\varphi(D_1),\ldots,\varphi(D_k) \,\text{ are orthogonal}. \nonumber 
\end{equation}

Since
$g(\varphi X,\varphi Y)=g(fX,fY)+g(wX, wY)$, 
from Definition \ref{8} (iii), we get 
\begin{equation}\label{17}
w(D_i)\perp w(D_j) \ \ \text{for} \ \ i\neq j \ \ \text{from} \ \ \{1,\ldots,k\}.
\end{equation}
\end{remark}

In view of (\ref{17}), we get 

\begin{proposition}\label{p64}
If the orthogonality of vector fields from the proper \mbox{$k$-slant} distribution $\oplus_{i=1}^kD_i$ is invariant under $\varphi$, we have 
$$w(\oplus_{i=1}^kD_i)=\oplus_{i=1}^kw(D_i).$$
\end{proposition}

\begin{remark}\label{16}
In analogy with the existent terminology for slant submanifolds, particular types of $k$-slant distributions will be named as follows. 

Let $D=\oplus_{i=0}^kD_i$ be a $k$-slant distribution. For $k=1$ and $D_0=\{0\}$, $D$ is a \textit{slant distribution}; it is an \textit{anti-invariant distribution} if $\theta_1=\frac{\pi}{2}$. For $k=1$ and $D_0\neq \{0\}$, $D$ is a \textit{semi-invariant distribution} if $\theta_1=\frac{\pi}{2}$ or a \textit{semi-slant distribution} if $\theta_1<\frac{\pi}{2}$. For $k=2$ and $D_0=\{0\}$, $D$ is a \textit{bi-slant distribution}; it is a \textit{hemi-slant distribution} if one of the slant angles is equal to $\frac{\pi}{2}$. 
\end{remark}

Let $M$ be an immersed submanifold of $\overline{M}$ and $k\in \mathbb{N}^*$. 
Considering the notion of $k$-slant distribution, we introduce the notion of $k$-slant submanifold. 

\begin{definition}\label{p30}
We will call $M$ a \textit{$k$-slant submanifold} of $\overline{M}$\, if $TM$ is a \mbox{$k$-slant} distribution. 

We will say that $M$ is a \textit{$(\theta_1,\theta_2,\ldots,\theta_k)$-slant sub\-man\-i\-fold} if we want to specify the values $\theta_i$ of the slant angles, or a \textit{multi-slant submanifold} if $k\geq 2$. 

Denoting $TM=\oplus_{i=0}^kD_i$, where $D_0$ is the invariant component, we will call $\oplus_{i=1}^kD_i$ the \textit{proper $k$-slant distribution associated} to $M$. 

We will call $M$ a \textit{proper $k$-slant submanifold} if\, $TM$ is a proper \mbox{$k$-slant} distribution. 

$M$ is called an \textit{invariant submanifold} if\, $TM$ is an invariant distribution. 
\end{definition}

The explicit formulation of the above definition is  

\begin{definition}\label{20}
We will say that $M$ is a \textit{$k$-slant submanifold} of $\overline{M}$ if there exists an orthogonal decomposition of $TM$ into regular distributions, 
$$TM=\oplus_{i=0}^k{D_i}$$
with $D_i\neq \{0\}$ for $i= \overline{1,k}$ and $D_0$ possible null, and there exist distinct values $\theta_i\in (0,\frac{\pi}{2}]$, $i=\overline{1,k}$, such that: \\ 
\hspace*{7pt} (i) \;$\varphi v \neq 0$, and $ \widehat{(\varphi v, (D_i)_x)}=\theta_i$ for any $x\in {M}$ and $v\in (D_i)_x\verb=\=\{0\}$, $i=\overline{1,k}$;\\ 
\hspace*{7pt} (ii) \,$\varphi v\in (D_0)_x$ for any $x\in M$ and $v\in (D_0)_x$;\\ 
\hspace*{7pt} (iii) $fv\in (D_i)_x$ for any $x\in M$ and $v\in (D_i)_x$, $i=\overline{1,k}$. 
\end{definition}

\begin{remark}\label{p88}
As justified above, we have: \\ 
(a) Condition (i) of Definition \ref{20} can be replaced by \\ 
\hspace*{7pt} (i') $\varphi v \neq 0$, and 
$ \widehat{(\varphi v, T_xM)}=\theta_i$ for any $x\in {M}$ and $v\in (D_i)_x\verb=\=\{0\}$, $i=\overline{1,k}$; \\ 
(b) Condition (iii) of Definition \ref{20} can be replaced by \\ 
\hspace*{7pt} (iii') $\varphi(D_i)\perp D_j$ for any $i\neq j$ from $\{1,\ldots,k\}$.
\end{remark}

\begin{remark}\label{p89}
We notice that all the results that would be valid for any \mbox{$k$-slant} distribution on an arbitrary submanifold of $\overline{M}$ will, in particular, be valid for any $k$-slant submanifold $M$ of $\overline{M}$. 
\end{remark}

\begin{proposition}\label{p175}
Let $M$ be an immersed submanifold of $\overline{M}$ such that $TM$ is decomposable into an orthogonal sum of regular distributions, $TM=\nolinebreak\oplus_{i=0}^k{D_i}$ with $D_0$ invariant (possible null) and $D_i\neq \{0\}$ for $i=\nolinebreak\overline{1,k}$. Let $pr_i$ denote the projection operator from $TM$ onto $D_i$ for $i=\overline{1,k}$. If $\varphi$ restricted to $\oplus_{i=1}^k{D_i}$ is an isometry, and $f(D_i)\subseteq D_i$ for $i=\overline{1,k}$, and there exist $k$ distinct values $\theta_i\in (0,\frac{\pi}{2}]$, $i=\overline{1,k}$, such that  
\begin{equation}\nonumber
f^2X=\epsilon \sum_{i=1}^k\cos^2\theta_i\cdot pr_iX\, \text{ for any }\, X\in \oplus_{i=1}^k{D_i},
\end{equation}
then $M$ is a $k$-slant submanifold of $\overline{M}$ with slant angles $\theta_i$ corresponding to $D_i$, $i=\overline{1,k}$. 

\end{proposition}

\begin{remark}\label{part_subm} 
Let $M$ be a $k$-slant submanifold of $\overline{M}$ and $TM=\oplus_{i=0}^k{D_i}$. 
The already known particular cases are the following. 

If $k=1$ and $D_0=\{0\}$, $M$ is a \textit{slant submanifold}; it is an \textit{anti-invariant submanifold} for $\theta_1=\frac{\pi}{2}$. If $k=1$ and $D_0\neq \{0\}$, $M$ is a \textit{semi-invariant submanifold} for $\theta_1=\frac{\pi}{2}$ or a \textit{semi-slant submanifold} for $\theta_1<\frac{\pi}{2}$. If $k=2$, $M$ is an \textit{almost bi-slant submanifold}; if, additionally, $D_0=\{0\}$, $M$ is a \textit{bi-slant submanifold}, and it is a \textit{hemi-slant submanifold} if one of the slant angles is equal to $\frac{\pi}{2}$. 
\end{remark}

\begin{definition}\label{p72}
Let $M$ be $\overline{M}$ or an immersed submanifold of $\overline{M}$, and let $X=\{X_x\}_{x\in M}$ and $Y=\{Y_x\}_{x\in M}$, with $X_x,Y_x \in T_x\overline{M}$ for $x\in M$, be two nonzero smooth families (with respect to $x\in M$) of tangent vectors of $\overline{M}$ (in particular, $X$ and $Y$ can be two nonzero vector fields on $M$). 
We will say that $X$ and $Y$ are \textit{angular compatible} if there is $x\in {M}$ such that $X_x$ and $Y_x$ are both nonzero. 
Denoting ${M}_{\widehat{X,Y}}:=\{x\in {M}\,|\,X_x, Y_x \neq 0\}$, we introduce the \textit{angular function} of $X$ and $Y$ (in short, the \textit{angle} between $X$ and $Y$) as $\widehat{(X,Y)}:\nolinebreak {M}_{\widehat{X,Y}}\rightarrow\nolinebreak[4] [0,{\pi}]$ defined by
$\widehat{(X,Y)}(x) :=\widehat{(X_x,Y_x)}$ for $x\in {M}_{\widehat{X,Y}}$\,, where $\widehat{(X_x,Y_x)} =\arccos \displaystyle\frac {g_x (X_x,Y_x)}{\|X_x\|\cdot \|Y_x\|}$\,. 

Correspondingly, we will denote by $\cos \widehat{(X,Y)}$ and $\sin \widehat{(X,Y)}$ the real functions defined on ${M}_{\widehat{X,Y}}$\, by $x\mapsto \cos \widehat{(X_x,Y_x)}$ and $x\mapsto \sin \widehat{(X_x,Y_x)}$, re\-spec\-tive\-ly. 
\end{definition}

%%%%%%%%%%%%%%%%%%%%%%%%%%
\section{$k$-slant distributions and $k$-slant submanifolds in almost contact and almost paracontact metric geometries}\label{alm_cont}

On a Riemannian manifold $(\overline{M},g)$, we consider a unitary vector field $\xi$ and its dual $1$-form $\eta$ (defined by $\eta(X)=g(X,\xi)$ for any $X\in T\overline{M}$); this satisfies $\eta(\xi)=1$. For a fixed $\epsilon\in \{-1,1\}$, let $\varphi$ be a $(1,1)$-tensor field on $(\overline{M},g)$ $\epsilon$-compatible with $g$, i.e., 
\begin{equation}\label{35}
g(\varphi X,Y)=\epsilon g(X,\varphi Y)\, 
\text{ for any $X,Y\in T\overline{M}$},\nonumber  
\end{equation}
such that 
\begin{equation}\label{2}
\varphi^2=\epsilon (I-\eta\otimes \xi). \nonumber
\end{equation}
We immediately get: 
\begin{equation}\label{3}
g(\varphi X,\varphi Y)=g(X,Y)-\eta(X)\eta(Y)\,
\text{ for any }X,Y\in T\overline{M};
\end{equation}
\begin{equation}
\varphi \xi=0,\ \text{and}\ \eta(\varphi X)=0\, \text{ for any }X\in T{\overline{M}};
\nonumber 
\end{equation}
\begin{equation}\label{7}
\varphi^2X=\epsilon X,\ \varphi X\in \langle\xi\rangle^{\perp},\ \text{and}\ |\varphi X|=|X|\,
\text{ for any }X\in \langle\xi\rangle^{\perp}.
\end{equation}

Therefore, $\ker \varphi=\langle\xi\rangle$; hence, $\dim(\ker \varphi)=1$.

\begin{remark}\label{p90}
For $\epsilon =-1$ (in which case $\overline{M}$ has to be odd dimensional), $(\varphi,\xi,\eta,g)$ defines an \textit{almost contact metric structure} and $(\overline{M},\varphi,\xi,\eta,g)$ becomes an \textit{almost contact metric manifold}, while for $\epsilon =1$, $(\varphi,\xi,\eta,g)$ defines an \textit{almost paracontact metric (Riemannian) structure} and $(\overline{M},\varphi,\xi,\eta,g)$ becomes an \textit{almost paracontact metric (Riemannian) manifold}. 
\end{remark}

The notions present in the next definition were actually introduced in \cite{latcu} under the names of $\epsilon$-\textit{almost contact metric structure} and $\epsilon$-\textit{almost contact metric manifold}. To avoid a possible misunderstanding, relating to the use of these names with another meaning in the semi-Riemannian case, we will rename them as follows. 

\begin{definition}\label{p152}
For $\epsilon\in \{-1,1\}$, we will call $(\varphi,\xi,\eta,g)$ an \textit{almost} ($\epsilon$)-\textit{contact metric structure} and $(\overline{M},\varphi,\xi,\eta,g)$ an \textit{almost} ($\epsilon$)-\textit{contact metric manifold}. 
\end{definition}

\begin{remark}\label{p33}
In view of (\ref{3}), we notice that in an almost ($\epsilon$)-contact metric manifold $(\overline{M},\varphi,\xi,\eta,g)$, that is, in an almost contact  metric manifold or an almost paracontact  metric manifold, $\varphi$ restricted to $\langle\xi\rangle^{\perp}$ is an isometry; hence, it preserves the orthogonality of vector fields from $\langle\xi\rangle^{\perp}$. 
\end{remark}

Throughout this section, we consider that any sub\-man\-i\-fold $M$ of $\overline{M}$ we deal with satisfies $\xi\in TM$. 

Let $M$ be an immersed submanifold of $\overline{M}$. Since $\langle\xi\rangle=\ker \varphi$, $\langle\xi\rangle$ does not participate to any slant distribution on $M$ but can be considered as a part of an invariant component of $TM$. Thus, any slant distribution on $M$ is included in $\langle\xi\rangle^{\perp_{TM}}$ (the orthogonal complement of $\langle\xi\rangle$ in $TM$), and, in view of (\ref{7}), the definition of a $k$-slant submanifold will become the following. 

\begin{definition}\label{4}
Let $k\in \mathbb{N}^*$. We say that $M$ is a \textit{$k$-slant submanifold} of $(\overline{M},\varphi,\xi,\eta,g)$ if there exists an orthogonal decomposition of $TM$ into regular distributions, 
$$TM=\oplus_{i=0}^k{D_i}\oplus \langle \xi\rangle$$
with $D_i\neq \{0\}$ for $i= \overline{1,k}$ and $D_0$ possible null, and there exist distinct values $\theta_i\in (0,\frac{\pi}{2}]$, $i=\overline{1,k}$, such that: \\ 
\hspace*{7pt} (i) \,$\widehat{(\varphi X, D_i)}=\theta_i$ for any $X\in D_i\verb=\=\{0\}$, $i=\overline{1,k}$;\\ 
\hspace*{7pt} (ii) \,$\varphi X\in D_0$ for any $X\in D_0$ (i.e., $\widehat{(\varphi X, TM)}=0=:\theta_0$ for $X\in\nolinebreak D_0\verb=\=\{0\}$, and $f(D_0)\subseteq D_0$);\\ 
\hspace*{7pt} (iii) $f(D_i)\subseteq D_i$ for $i=\overline{1,k}$. 
\end{definition}

\begin{remark}\ \\ 
(a) In view of (iii), condition (i) can be replaced by \\ 
\hspace*{7pt} (i') $\widehat{(\varphi X, TM)}=\theta_i$ for any $X\in D_i\verb=\=\{0\}$, $i=\overline{1,k}$;\\ 
(b) Condition (iii) can be replaced by \\ 
\hspace*{7pt} (iii') $\varphi(D_i)\perp D_j$ for any $i\neq j$ from $\{1,\ldots,k\}$.
\end{remark}

\begin{remark}
We notice that $\oplus_{i=1}^kD_i$ is a proper $(\theta_1,\theta_2,\ldots,\theta_k)$-slant distribution and represents the proper $k$-slant distribution associated to $M$. 
\end{remark}

The last definition can be reformulated as follows. 

\begin{definition}\label{p205}
Let $k\in \mathbb{N}^*$. We say that $M$ is a \textit{$k$-slant submanifold} of $\overline{M}$ if, in the orthogonal decomposition $TM=D\oplus \langle \xi\rangle$, $D$ is a $k$-slant distribution. 
\end{definition}

\begin{remark}\label{p206}
The correspondence \textit{$k$-slant submanifold $\leftrightarrow$ $k$-slant distribution} will take place, in the almost ($\epsilon$)-contact metric setting and with the notations of the last definition, between the submanifold $M$ and the distribution $D$, i.e., between the submanifold $M$ and $\langle\xi\rangle^{\perp_{TM}}$ organized as a $k$-slant distribution. 
\end{remark}

\begin{example}\label{ex1}
Let $\overline M=\mathbb{R}^{4k+3}$ be the Euclidean space for some $k\geq 2$, with the canonical coordinates 
$(x_{1},\ldots, x_{4k+3})$, and let $\{e_{1}=\frac{\partial }{\partial x_{1}},\ldots,e_{4k+3}=\frac{\partial }{\partial x_{4k+3}}\}$ be the natural basis in the tangent bundle. Let $\epsilon\in \{-1,1\}$, and let us define
a vector field $\xi $, a $1$-form $\eta $, and a $(1,1)$-tensor field $\varphi$ by: 
$$\xi =e_{4k+3}, \quad \eta =dx_{4k+3},$$
$$
\varphi e_{1} = e_{2}, \quad \varphi e_{2}=\epsilon e_{1},$$
$$\varphi e_{4j-1} =\frac{j^2-1}{j^2+1}{\ }e_{4j}+\epsilon \frac{2j}{j^2+1}{\ }e_{4j+2},$$
$$\varphi e_{4j}=\epsilon \frac{j^2-1}{j^2+1}{\ }e_{4j-1}+\epsilon \frac{2j}{j^2+1}{\ }e_{4j+1},$$
$$\varphi e_{4j+1} =\frac{2j}{j^2+1}{\ }e_{4j}-\epsilon \frac{j^2-1}{j^2+1}{\ }e_{4j+2},$$
$$\varphi e_{4j+2}=\frac{2j}{j^2+1}{\ }e_{4j-1}-\frac{j^2-1}{j^2+1}{\ }e_{4j+1},$$
$$\varphi e_{4k+3} =0$$
for $j= \overline{1,k}$.
Let the metric tensor field $g$ be given by $g(e_{i},e_{j})=\delta _{ij}$, $i, j=\nolinebreak\overline{1,4k+3}$. Then, $(\overline{M}, \varphi,\xi,\eta, g )$ is an almost ($\epsilon$)-contact metric manifold. It's to be noticed that, for $\epsilon=-1$, it is an almost contact metric manifold, and, for $\epsilon=1$, it is an almost paracontact metric manifold.

We define the following submanifold of $\overline M$: 
$$
M:=\{(x_{1},\ldots, x_{4k+3})\in \mathbb{R}^{4k+3} \ | \ x_{4j+1}=x_{4j+2}=0,\ j=\overline{1,k}\}.  $$

Considering 
$D_0=\langle e_{1},e_{2}\rangle,\, D_{j}=\langle e_{4j-1},e_{4j}\rangle,\, j=\overline{1,k}$, we notice that $M$ is a $k$-slant submanifold with $TM=\oplus_{i=0}^k D_i\oplus \langle\xi\rangle$. The corresponding \mbox{$k$-slant} distribution is $\oplus_{i=0}^k D_i$, where $D_0$ is an invariant distribution, and $D_j$, $j=\overline{1,k}$, are slant distributions with corresponding slant angles 
$$\theta _{j}=\arccos \left(\frac{j^2-1}{j^2+1}\right),\ j=\overline{1,k}\,.$$
$\oplus_{i=1}^kD_i$ is the proper $k$-slant distribution associated to $M$.
\end{example}

\medskip
Further, until the end of this section, we will consider $M$ to be $\overline{M}$ or an immersed submanifold of $\overline{M}$ if not specified otherwise. 

In the following, we will consider that $\epsilon=-1$, that is, $(\overline{M},\varphi,\xi,\eta,g)$ is an almost contact metric manifold. 
Let $k\in \mathbb{N}^*$ and $D=\oplus_{i=0}^k{D_i}$ be a \mbox{$k$-slant} distribution on $M$ with $D_0$ the invariant component such that $\xi \perp D$, and let $G$ be the orthogonal complement of $D\oplus \langle \xi\rangle$ in $T\overline{M}$, i.e., $G=(D\oplus \langle \xi\rangle)^\perp$. Let $\theta_1,\theta_2,\dots ,\theta_k$ denote the slant angles of $D$, and let $\theta_0=0$. We notice that, for any $Z\in T\overline{M}$, the components of $\varphi Z_M$ in $D$ and in $D^\perp$ coincide with the components of $\varphi Z_M$ in $D\oplus \langle \xi\rangle$ and in $G$, respectively.
From (\ref{7}) and Definition \ref{8} (ii), we get
\begin{equation}\label{9}
\varphi(D_0)=D_0
\end{equation}
and, therefore, $w(D_0)=\{0\}$, and $f(D_0)=D_0$\,.
We have $\eta(X)=g(X,\xi)=0$ for any $X\in D \oplus G$, which implies
\begin{equation}\label{12}
\ker \eta_M=D \oplus G, \nonumber
\end{equation}
where $\eta_M$ is the ''localization'' $\{\eta_x\}_{x\in M}$ of $\eta$ on $M$, 
and, from $\eta(\varphi(T\overline{M}))=\nolinebreak\{0\}$, we get 
$$\varphi(D \oplus G)\subseteq D \oplus G.$$
In view of (\ref{7}), it follows that 
\begin{equation}\label{13}
\varphi(D \oplus G)=D \oplus G, 
\end{equation}

and we get: 

\begin{remark}\label{p184}
\begin{align*}
(i) \quad \varphi^2(D_i) &=D_i \,\text{ for } i=\overline{1,k},\ \text{and}\ \, \varphi^2(G)=G;\\ 
(ii) \ \ f(\varphi X) &= -X, \text{ and }\, w(\varphi X)=0 \,\text{ for any } X\in D;\\ 
(iii) \;\,f(\varphi U) &=0, \text{ and }\, w(\varphi U)=- U \,\text{ for any } U\in G. \nonumber
\end{align*}
\end{remark}

For any $i\in \{1,\ldots,k\}$ and $X_i\in D_i\verb=\=\{0\}$, from Definition \ref{8} (i), we have $\varphi X_i\neq 0$ and 
\begin{equation}\label{34}
|fX_i|=\cos \theta_i\cdot |\varphi X_i|, 
\end{equation}
which, for $X_i,Y_i \in D_i$, implies
\begin{equation}\label{36}
g(f^2X_i,Y_i)= \cos^2 \theta_i \cdot g(\varphi^2 X_i, Y_i).
\end{equation}

Taking into account that $f(D_i)\subseteq D_i$, we notice that, for any $Z \in T{M}$ and $X_i\in D_i$, we have 
\begin{equation}\label{10}
g(f^2X_i,Z)= \cos^2 \theta_i \cdot g(\varphi^2 X_i,Z); 
\nonumber 
\end{equation}
so, we get
\begin{equation}\label{37}
f^2X_i=- \cos^2\theta_i\cdot X_i\, \text{ for any } X_i\in D_i\,.
\end{equation}

From (\ref{7}) and Definition \ref{8} (ii), we deduce that $f^2X=-X$ for $X \in\nolinebreak D_0$. Denoting by $pr_i$ the projection operators onto $D_i$, $i=\overline{0,k}$, we get 

\begin{proposition}\label{p23}
\begin{equation}\label{38}
f^2X=-\sum_{i=0}^k\cos^2\theta_i\cdot pr_iX\, \text{ for any } X\in D.
\end{equation}
\end{proposition}

\begin{corollary}\label{p0}
\begin{equation}\label{39}
f(D_i)=D_i\, \text{ for any } i \text{ with } \theta_i\neq \frac{\pi}{2}\,. \nonumber
\end{equation}
\end{corollary}

\begin{remark}\label{169}
Proposition \ref{p23} and Corollary \ref{p0} are, in particular, valid if $M$ is a $k$-slant submanifold of $(\overline{M},\varphi,\xi,\eta,g)$, considering the distribution $D=\oplus_{i=0}^k{D_i}$ for $TM=\oplus_{i=0}^k{D_i}\oplus \langle \xi\rangle$. 
\end{remark}

Taking into account Propositions \ref{p167}, \ref{p23} and Remark \ref{p33}, we get 

\begin{theorem}\label{p168}
Let $\mathfrak D$ be a non-null distribution on $M$ such that $\mathfrak D \perp\nolinebreak \xi$ and $\mathfrak D$ is decomposable into an orthogonal sum of regular distributions, $\mathfrak D=\nolinebreak\oplus_{i=0}^k{\mathfrak D_i}$ with $\mathfrak D_i\neq \{0\}$ for $i=\overline{1,k}$ and $\mathfrak D_0$ invariant (possible null). Let $pr_i$ denote the projection operator onto $\mathfrak D_i$ for $i=\overline{0,k}$, $f$ the component of $\varphi$ into $\mathfrak D$ (i.e., $f=pr_{\mathfrak D}\circ \varphi$), and $\theta_0=0$. If $f(\mathfrak D_i)\subseteq \mathfrak D_i$ for $i=\overline{1,k}$, then the following assertions are equivalent: \\ 
\hspace*{7pt} (a) There exist $k$ distinct values $\theta_i\in (0,\frac{\pi}{2}]$, $i=\overline{1,k}$, such that  
\begin{equation}\nonumber
f^2X= -\sum_{i=0}^k\cos^2\theta_i\cdot pr_iX\, \text{ for any } X\in \mathfrak D;
\end{equation}
\hspace*{7pt} (b) $\mathfrak D$ is a $k$-slant distribution with slant angles $\theta_i$ corresponding to $\mathfrak D_i$, $i=\nolinebreak\overline{1,k}$. 
\end{theorem}

\begin{remark}\label{p170}
Theorem \ref{p168} provides a necessary and sufficient condition for a submanifold $M$ of an almost contact metric manifold $\overline M$ to be a \mbox{$k$-slant} submanifold, considering $\mathfrak D$ to be the distribution on $M$  given by $\mathfrak D=\nolinebreak\oplus_{i=0}^k{\mathfrak D_i}$ if $TM=\oplus_{i=0}^k{\mathfrak D_i}\oplus \langle \xi\rangle$.
\end{remark}

\begin{remark}\label{p201}
We describe below the notion of \textit{skew CR submanifold of an almost contact metric manifold}, relating it to the concept of $k$-slant submanifold.  

Let $M$ be an immersed submanifold of an almost contact metric manifold $(\overline{M},\varphi,\xi,\eta,g)$, and, for any $Z\in T{M}$, let $fZ$ be the tangential component of $\varphi Z$ (the component of $\varphi Z$ in $TM$). 

In view of Lemmas \ref{lema1}, \ref{lema2} and Remark \ref{p194}, $f$ is skew-symmetric; hence, $f^2$ is symmetric. Denoting by $\lambda_i(x)$, $i=\overline{1,m(x)}$, the distinct eigenvalues of $f_x^2$ acting on the tangent space $T_xM$ for $x\in M$, these eigen\-val\-ues are all real and nonpositive. 
In view of (\ref{7}) and $\varphi \xi=0$, we have $|fX|\leq|\varphi X|\leq |X|$ for any $X\in T{M}$; hence, all the $\lambda_i(x)$'s are contained in $[-1,0]$. For any $x\in M$, let $\mathfrak D_x^i$ denote the eigenspace corresponding to $\lambda_i(x)$, $i=\overline{1,m(x)}$; then, the tangent space $T_xM$ of $M$ at $x$ has the following orthogonal decomposition into the eigenspaces of $f_x^2$: 
$$T_xM={\mathfrak D_x^1}\oplus\ldots\oplus{\mathfrak D_x^{m(x)}}.$$
Every eigenspace $\mathfrak D_x^i$ is invariant under $f_x$, and, for $\lambda_i(x)\neq0$, the cor\-re\-spond\-ing $\mathfrak D_x^i$ is of even dimension. 

Now, let us consider that $M$ is a skew CR submanifold of $\overline{M}$, i.e.: 
\begin{enumerate}
\item $m(x)$ does not depend on $x\in M$ (we will denote $m(x)=m$); 
\item the dimension of ${\mathfrak D_x^i}$, $i=\overline{1,m}$, is independent of $x\in M$; 
\item each $\lambda_i(\cdot)$ is constant on $M$ (we will denote $\lambda_i(x)=\lambda_i)$. 
\end{enumerate}

So, for any tangent space $T_xM$, we have the same number $m$ of distinct eigenvalues, $\lambda_1$, \dots , $\lambda_m$, of $f_x^2$, these being independent of $x\in M$. 
Denoting by $\mathfrak D_i$ the distribution corresponding to the family $\{\mathfrak D_x^i:\nolinebreak x\in \nolinebreak M\}$ for $i=\overline{1,m}$, we get for $TM$ the orthogonal decomposition: 
$$TM={\mathfrak D_1}\oplus\ldots\oplus{\mathfrak D_m}\,.$$
Moreover, every $\mathfrak D_i$ is invariant under $f$, and, for $\lambda_i\neq0$, the cor\-re\-spond\-ing distribution $\mathfrak D_i$ is of even dimension.

Since $f\xi=\varphi\xi=0$, one of the $\mathfrak D_i$'s contains $\langle \xi \rangle$ and corresponds to the zero eigenvalue of $f^2$; let $\mathfrak D_m$ be that distribution. Let us decompose $\mathfrak D_m$ into $\langle\xi\rangle$ and the orthogonal complement of $\langle\xi\rangle$ in $\mathfrak D_m$, denoted by $\mathfrak D'_m$, $\mathfrak D_m=\langle\xi\rangle\oplus\mathfrak D'_m$. We notice that $\mathfrak D'_m$, if non-null, is a slant distribution with slant angle $\frac{\pi}{2}$.

For every $i\in \{1,\ldots,m-1\}$, let $\alpha_i\in (0,1]$ denote the positive value for which $\lambda_i=-\alpha_i^2$. Then, for any $X\in  \mathfrak{D}_i\verb=\=\{0\}$, we get 
$$|fX|^2=-\lambda_ig(X,X)= \alpha_i^2 |X|^2= \alpha_i^2 |\varphi X|^2$$
and $|fX|=\alpha_i |X|=\alpha_i |\varphi X|$; hence, $\alpha_i=\cos\zeta_i$, where $\zeta_i$ is the value of the angle between $\varphi X$ and $TM$, the same for any nonzero $X\in \mathfrak D_i$. 

The distributions $\mathfrak{D}_i$, $i=\overline{1,m-1}$, are slant distributions with distinct slant angles $\zeta_i$ except at most one of them, which is invariant (with respect to $\varphi$) and corresponds to $\alpha_i=1$ if such one exists. 

It follows that, since $\oplus_{i=1}^{m-1}\mathfrak{D}_i\oplus\mathfrak{D}'_m$ does not reduce to an invariant distribution with respect to $\varphi$, $M$ is a $k$-slant submanifold of $\overline{M}$, where $k$ is one of the values: $m,\, m-1,\, m-2$. 
\end{remark}

We conclude: 

\begin{proposition}\label{p123}
Any skew CR submanifold of an almost contact metric man\-i\-fold is a $k$-slant submanifold. 
\end{proposition}

\begin{proposition}\label{p210}
Any $k$-slant submanifold of an almost contact metric man\-i\-fold which is not an anti-invariant or a CR submanifold is a skew CR submanifold. 
\end{proposition}

%%%%%%%%%%%%%%%%%%%%%%%%%%%%%%%%%%%
\subsection{The dual $k$-slant distribution in almost contact metric \newline geometry}\label{subsec_epsilon=-1}

We continue to investigate, for $k\in \mathbb{N}^*$, the properties of the $k$-slant distribution $D=\oplus_{i=0}^k{D_i}$ on $M$, where $M$ is $\overline{M}$ or an immersed submanifold of the almost contact metric manifold $(\overline{M},\varphi,\xi,\eta,g)$, with $D_0$ the invariant component and $\xi \perp D$, underlining the properties of $G$, the orthogonal complement of $D\oplus \langle \xi\rangle$ in $T\overline{M}$. 

For $U\in G$, we have $g(fU,\xi)=0$, which implies
$f(G)\perp \langle \xi\rangle$, 
and, if $X\in D_0$, we have $g(fU,X)=0$, which gives
$f(G)\perp D_0$\,. 
In conclusion,
\begin{equation}\label{42}
f(G)\subseteq \oplus_{i=1}^k{D_i}\,.
\end{equation}

For $U\in G$,\, $U\perp w(\oplus_{i=1}^k{D_i})$, in view of Lemma \ref{lema1}, we have \linebreak $fU\perp \oplus_{i=1}^k{D_i}$, and, using (\ref{42}), we get $fU=0$.
Taking into account (\ref{17}), we have the following orthogonal decomposition: 

\begin{theorem}\label{p1}
\begin{equation}\label{43}
G=\oplus_{i=1}^kw({D_i})\oplus H, \text{ where }f(H)=\{0\}.
\nonumber
\end{equation}
\end{theorem}

\medskip
For $V\in H$, $X\in D_i$, $i= \overline{1,k}$, we have
$g(wX,wV)=g(\varphi X, \varphi V)=0$; hence, $wV\perp w(D_i)$. It follows that $w(H)\perp \oplus_{i=1}^kw({D_i})$; thus, $w(H)\subseteq H$.

In view of $f(H)=\{0\}$ and (\ref{7}), for any $V\in H$, we have 
$$\varphi V=wV\in H, \ \text{ hence }\ w^2V=w\varphi V=\varphi^2 V=- V\in H,$$
which implies $w^2(H)=H$. Taking into account that $w(H)\subseteq H$, we get 

\begin{corollary}\label{p62}
\begin{equation}\label{45}
\varphi (H)=w(H)=H. \nonumber
\end{equation}
\end{corollary}

\medskip 
From Definition \ref{8} (ii)-(iii), (\ref{3}), Lemma \ref{lema2}, and (\ref{38}), we obtain 

\begin{proposition}\label{p2}
For any $X,Y\in D$, we have: 
\begin{align}%\label{48}
g(\varphi X,\varphi Y)&= \sum_{i=0}^k g(pr_iX, pr_iY), \nonumber\\ 
%\label{49}
g(fX,fY)&=\sum_{i=0}^k \cos^2 \theta_i \cdot g(pr_iX, pr_iY), \nonumber\\ 
%\label{50}
g(wX,wY)&=\sum_{i=1}^k \sin^2 \theta_i \cdot g(pr_iX, pr_iY). \nonumber
\end{align}

\end{proposition}

\bigskip 
Hence, $-g(fwX,Y)= \sum_{i=1}^k \sin^2 \theta_i \cdot g(pr_iX, Y)$. 

\begin{corollary}\label{p69}
\begin{equation}\label{51}
fwX=- \sum_{i=1}^k \sin^2 \theta_i \cdot pr_iX\  
\text{ for any }\ X\in D.
\end{equation}
\end{corollary}

\medskip 
In view of (\ref{42}) and (\ref{51}), we get: 

\begin{proposition}\label{p3}
\begin{align*}
f(w(D_i)) &=D_i \,\text{ for } i=\overline{1,k}\,;\\ 
f(G) &=\oplus_{i=1}^k{D_i}\,. 
\end{align*}
Moreover, for any $i\in \{1,\ldots,k\}$, $w|_{D_i}$ and $f|_{w(D_i)}$ are injective; therefore, $w(D_i)$ and $D_i$ localized in any point of $M$ are isomorphic; hence, $w(D_i)$ is also regular as $D_i$ is, and they both have the same dimension.
\end{proposition}

\begin{remark}\label{p70}
If $\theta_j=\frac{\pi}{2}$, then, for any $X_j\in D_j$, we have
\begin{equation}\label{54}
fwX_j=-X_j
\end{equation}
and $wfwX_j=- wX_j$, which implies
$wfU_j=-U_j$ for any \ $U_j\in w(D_j)$; 
thus, $f|_{w(D_j)}: w(D_j) \rightarrow D_j$ and $w|_{D_j}: D_j \rightarrow w(D_j)$ are anti-inverse to each other.
\end{remark}

\begin{proposition}\label{p4}
For any $X\in D\oplus \langle \xi\rangle$ and $U\in G$, we have:
\begin{align}%\label{56}
f^2X+fwX &=-X+ \eta(X) \xi, \nonumber\\ 
%\label{57}
wfX+w^2X &=0, \nonumber \\ 
%\label{58}
f^2U+fwU &=0, \nonumber
\\ 
%\label{59}
wfU+w^2U &=-U. \nonumber
\end{align}
\end{proposition}

\medskip 
In view of Proposition \ref{p4} and Theorem \ref{p1}, we get 

\begin{corollary}\label{p60}
For any $U_0,V_0\in H$, we have: 
\begin{align}
w^2U_0&=-U_0\,,\nonumber\\  
g(wU_0,wV_0)&=g(U_0,V_0),\nonumber\\ 
|wU_0|&=|U_0|.\nonumber
\end{align}
\end{corollary}

\medskip
For $X_i \in D_i$, $i=\overline{1,k}$, we have $w^2X_i=-wfX_i\in w(D_i)$; hence, 
$w^2(D_i)\subseteq w(D_i)$.

If $\theta_i\neq \frac{\pi}{2}$ and $Y_i\in D_i$, there exists $X_i\in D_i$ with $fX_i=-Y_i$; hence, $wY_i=w^2X_i$, so $w(D_i)\subseteq w^2(D_i)$.

If $\theta_j= \frac{\pi}{2}$ and $X_j\in D_j$, we have $fX_j=0$ and $\varphi wX_j=\varphi^2X_j=-X_j$; hence,
$w^2X_j+fwX_j=-X_j$, so $w^2X_j=0$.

We deduce: 

\begin{proposition}\label{p5}
\[
w^2(D_i)=
\begin{cases}
w(D_i) & \text{for} \ \ \theta_i\neq \frac{\pi}{2}\,,\\ 
\ \{0\} & \textit{for} \ \ \theta_i=\frac{\pi}{2}\,.
\end{cases}
\]
\end{proposition}

\medskip 
For $\theta_j=\frac{\pi}{2}$ and $U_j=wX_j\in w(D_j)$, we get $wU_j=0$ and $wfU_j=-U_j$.

In general, for $i\in \{1,\ldots,k\}$ and $U_i\in w(D_i)$, let $X_i\in D_i$ with $U_i=\nolinebreak wX_i$. We have
$w(fU_i)=w(fwX_i)=-\sin^2\theta_i \cdot wX_i$\,, hence 
\begin{equation}\label{62}
w(fU_i)=-\sin^2\theta_i \cdot U_i 
\end{equation}
and \ $w^2U_i=w\varphi U_i-w fU_i=\varphi^2 U_i-w fU_i$. We get 
$w^2U_i=-\cos^2\theta_i \cdot U_i$\,. 

\begin{proposition}\label{p6}
For any $U\in w(D)$, $U=\sum_{i=1}^kU_i$ with $U_i\in w(D_i)$, we have:
\begin{equation}\label{64}
wfU= -\sum_{i=1}^k\sin^2\theta_i\cdot U_i\,, 
\end{equation}
\begin{equation}\label{65}
w^2U= -\sum_{i=1}^k\cos^2\theta_i\cdot U_i\,.
\end{equation}
\end{proposition}

\medskip 
In view of Proposition \ref{p6} and Lemma \ref{lema1}, we get  

\begin{proposition}\label{p67}
For any $U,V\in w(D)$, $U=\sum_{i=1}^kU_i$, $V=\sum_{i=1}^kV_i$ with $U_i,V_i\in w(D_i)$, $i=\overline{1,k}$, we have: 
\begin{align}
%\label{82}
g(fU,fV)&=\sum_{i=1}^k \sin^2 \theta_i \cdot g(U_i, V_i), \nonumber\\ 
%\label{83}
g(wU,wV)&=\sum_{i=1}^k \cos^2 \theta_i \cdot g(U_i, V_i), \nonumber\\ 
%\label{84}
g(\varphi U,\varphi V)&= \sum_{i=1}^k g(U_i, V_i). \nonumber
\end{align}

\end{proposition}

\medskip 
From (\ref{38}), (\ref{65}), and Lemma \ref{lema2}, we obtain 

\begin{proposition}\label{p7}
For any $X=\sum_{i=0}^kX_i$ and $U=\sum_{i=1}^kU_i$ with $X_0\in D_0$, $X_i\in D_i$, and $U_i\in w(D_i)$, $i=\overline{1,k}$, we have: 
$$|fX|^2={\sum_{i=0}^k\cos^2\theta_i\cdot |X_i|^2}\ \text{ and }\ |wU|^2={\sum_{i=1}^k\cos^2\theta_i\cdot  |U_i|^2}.$$

In particular, we get:\, $|fX_0|= |X_0|$ for $X_0\in D_0$, 
$$|fX_i|=\cos\theta_i\cdot |X_i| \text{ and } |wU_i|=\cos\theta_i\cdot  |U_i| \text{ for } X_i\in D_i,\, U_i\in w(D_i),\, i=\overline{1,k}.
$$
\end{proposition}

\begin{corollary}\label{p15}
For any $i\in \{1,\ldots,k\}$, $X_i\in D_i$, $U_i\in w(D_i)$, we have: 
$$|wX_i|=\sin\theta_i \cdot |X_i| \ \text{ and }\ |fU_i|=\sin\theta_i \cdot |U_i|.$$

\end{corollary}

For more general vector fields, we obtain

\begin{proposition}\label{p16}
For $X=\sum_{i=1}^kX_i$ and $U=\sum_{i=1}^kU_i$ with $X_i\in D_i$, 
$U_i\in\nolinebreak w(D_i)$, $i=\overline{1,k}$, we have: 
$$|wX|^2={\sum_{i=1}^k\sin^2\theta_i\cdot |X_i|^2}\ \text{ and }\ |fU|^2={\sum_{i=1}^k\sin^2\theta_i\cdot |U_i|^2}.$$
\end{proposition}

\medskip 
Taking into account Propositions \ref{p2}, \ref{p7}, Corollary \ref{p60}, and relationships (\ref{65}), (\ref{3}), and (\ref{7}), we get 

\begin{proposition}\label{p13}
For any $i\in \{1,\ldots ,k\}$ with $\theta_i\neq \frac{\pi}{2}$ and any $X_i, Y_i\in \nolinebreak D_i\verb=\=\{0\}$, $U_i, V_i\in w(D_i)\verb=\=\{0\}$, $X_0, Y_0\in D_0\verb=\=\{0\}$, 
$U_0, V_0\in H\verb=\=\{0\}$, $\overline{X}, \overline{Y}\in (D\oplus G)\verb=\=\{0\}$ such that ${M}_{\widehat{X_i,Y_i}}\,, {M}_{\widehat{X_0,Y_0}}\,, {M}_{\widehat{U_i,V_i}}\,, {M}_{\widehat{U_0,V_0}}\,, {M}_{\widehat{\overline{X},\overline{Y}}}$ are nonempty, we have: \\ 
\hspace*{7pt} (i) \ \;$\cos(\widehat{fX_0,fY_0})=\cos(\widehat{\varphi X_0,\varphi Y_0})=\cos(\widehat{X_0,Y_0})$;\\ 
\hspace*{7pt} (ii) \ $\cos(\widehat{fX_i,fY_i})=\cos(\widehat{\varphi X_i,\varphi Y_i})=\cos(\widehat{X_i,Y_i})$;\\ 
\hspace*{7pt} (iii) \,$g(wU_i,wV_i)=\cos^2\theta_i \cdot g(U_i,V_i)$;\\ 
\hspace*{7pt} (iv) \;$\cos(\widehat{wU_0,wV_0})=\cos(\widehat{U_0,V_0})=\cos(\widehat{\varphi U_0,\varphi V_0})$;\\ 
\hspace*{7pt} (v) \ \;$\cos(\widehat{wU_i,wV_i})=\cos(\widehat{U_i,V_i})=\cos(\widehat{\varphi U_i,\varphi V_i})$;\\ 
\hspace*{7pt} (vi) \ $\cos(\widehat{\varphi \overline{X},\varphi \overline{Y}})=\cos(\widehat{\overline{X},\overline{Y}})$.
\end{proposition}

\medskip

If $\theta_i\neq \frac{\pi}{2}$ and $U_i\in w(D_i)\verb=\= \{0\}$, then, for any $x\in {M}$ with $(U_i)_x\neq 0$, from Lemma \ref{lema2} and (\ref{65}), we have $w(U_i)_x \neq 0$, $\varphi (U_i)_x \neq 0$, and 
$$\cos(\widehat{(wU_i)_x,(\varphi U_i)_x})=\frac{g((wU_i)_x,(\varphi U_i)_x)}{\|(wU_i)_x\| \cdot \|(\varphi U_i)_x\|}=\cos\theta_i\,; \text{ thus, } (\widehat{\varphi U_i, G})=\theta_i\,.$$

For $\theta_j=\frac{\pi}{2}$ and $U_j=wX_j,\ X_j\in D_j\verb=\= \{0\}$, in view of (\ref{54}) and Proposition \ref{p5}, we have $fU_j=-X_j\neq 0$ and $wU_j=0$; thus, ${(\widehat{\varphi U_j, G})}=\frac{\pi}{2}$\,. We can state

\begin{theorem}\label{p9}
The distribution $G=\oplus_{i=1}^kw(D_i)\oplus H$ is a $k$-slant distribution with $H$ the invariant component and $\oplus_{i=1}^kw(D_i)$ the proper $k$-slant com\-po\-nent, the slant distribution $w(D_i)$ having the same slant angle as $D_i$ for $i= \overline{1,k}$.
\end{theorem}

\begin{definition}\label{p10}
We will call $\oplus_{i=1}^kw(D_i)$ \textit{the dual $k$-slant distribution} of $\oplus_{i=1}^kD_i$.
\end{definition}

\begin{remark}\label{p11}
In the same way we defined the dual of the proper \mbox{$k$-slant} component $\oplus_{i=1}^kD_i$ of the distribution $D$ by means of $w$, we can construct the dual of the proper $k$-slant component $\oplus_{i=1}^kw(D_i)$ of the distribution $G$ by means of $f$, $f(\oplus_{i=1}^kw(D_i))=\oplus_{i=1}^kfw(D_i)$. 
\end{remark}

\begin{corollary}\label{p12}
The dual of the proper $k$-slant distribution $\oplus_{i=1}^kw(D_i)$, which is $\oplus_{i=1}^kf(w(D_i))$, is precisely the $k$-slant distribution $\oplus_{i=1}^kD_i$\,.
\end{corollary}

\medskip 
In view of Proposition \ref{p3} and Corollary \ref{p0}, denoting $w(D_i)$ by $G_i$, we obtain: 

\begin{proposition}\label{p66}
\begin{align*}
w(f(G_i)) &=G_i\,\text{ for }i=\overline{1,k}\,;\\ 
f^2(G_i) &=\left\{\begin{array}{ll}
D_i&\text{if }\theta_i\neq\frac{\pi}{2}\,, \\ 
\{0\}&\text{if }\theta_i= \frac{\pi}{2}\,. 
\end{array}
\right. 
\end{align*}
\end{proposition}

In view of (\ref{51}), we immediately get 

\begin{lemma}\label{p92}
For $X,Y\in \oplus_{i=1}^kD_i$, $U,V\in \oplus_{i=1}^kw(D_i)$, $x\in {M}$, we have: \\ 
\hspace*{7pt} (i)\, \ $X_x\neq 0$ if and only if $(wX)_x\neq 0$;\\ 
\hspace*{7pt} (ii) \ $U_x\neq 0$ if and only if $(fU)_x\neq 0$;\\ 
\hspace*{7pt} (iii) ${M}_{\widehat{X,Y}}={M}_{\widehat{wX,wY}}$ and 
${M}_{\widehat{U,V}}={M}_{\widehat{fU,fV}}$\,. 
\end{lemma}

\medskip 
The relation between an angle of two vector fields of a slant distribution and the angle of the corresponding vector fields in the dual distribution will be established in the next proposition.

Taking into account (\ref{51}), (\ref{62}), Corollary \ref{p15}, and Lemmas \ref{lema2}, \ref{p92}, we deduce: 

\begin{proposition}\label{p14}
For any $i\in\{1,\ldots,k\}$, $X_i, Y_i\in D_i\verb=\=\{0\}$, and \newline$U_i, V_i\in w(D_i)\verb=\=\{0\}$ with ${M}_{\widehat{X_i,Y_i}}$ and ${M}_{\widehat{U_i,V_i}}$ nonempty, we have: \\ 
\hspace*{7pt} (i) \ \;\,$g(wX_i,wY_i)=\sin^2\theta_i \cdot g(X_i,Y_i)$;\\ 
\hspace*{7pt} (ii) \;\,$g(fU_i,fV_i)=\sin^2\theta_i \cdot g(U_i,V_i)$;\\ 
\hspace*{7pt} (iii)\, $\cos(\widehat{wX_i,wY_i})=\cos(\widehat{X_i,Y_i})$;\\ 
\hspace*{7pt} (iv)\; $\cos(\widehat{fU_i,fV_i})=\cos(\widehat{U_i,V_i})$.
\end{proposition}

In view of Propositions \ref{p13} and \ref{p14}, we notice that  

\begin{theorem}\label{p21}
 $f$ and $w$ restricted to $D_i$ or $w(D_i)$, $i=\overline{1,k}$ \textup(excepting $f|_{D_j}$ and $w|_{w(D_j)}$ with $\theta_j=\frac{\pi}{2}$, in which case $f|_{D_j}$ and $w|_{w(D_j)}$ are vanishing\textup), 
 $f|_{D_0}$, $w|_H$, and $\varphi|_{D\oplus G}$ are conformal maps \textup(all of them preserve the angles\textup). 
\end{theorem}

From the orthogonal decompositions of $D$ and $G$ and from the above considerations, for any pair of angular compatible vector fields in a \mbox{$k$-slant} distribution, we find a corresponding pair which forms the same angle in the dual $k$-slant distribution.

\begin{theorem}\label{p17} 
For $X,Y\in (\oplus_{i=1}^kD_i)\verb=\=\{0\}$ and $U,V\in (\oplus_{i=1}^kw(D_i))\verb=\=\{0\}$ with ${M}_{\widehat{X,Y}}$ and ${M}_{\widehat{U,V}}$ nonempty, denoting $X=\sum_{i=1}^kX_i$, $Y=\sum_{i=1}^kY_i$, $U=\nolinebreak\sum_{i=1}^kU_i$, and $V=\sum_{i=1}^kV_i$, where $X_i,Y_i\in D_i$ and $U_i,V_i\in w(D_i)$, $i=\overline{1,k}$, we have: \\ 
\hspace*{7pt} (i) \ \;\,$g(wX,wY)=\sum_{i=1}^k\sin^2\theta_i \cdot g(X_i,Y_i)$;\\ 
\hspace*{7pt} (ii) \;\,$g(fU,fV)=\sum_{i=1}^k\sin^2\theta_i \cdot g(U_i,V_i)$;\\ 
\hspace*{7pt} (iii) \,$\cos(\widehat{wX,wY})=\cos\sphericalangle (\sum_{i=1}^k\sin\theta_i\cdot X_i,\ \sum_{i=1}^k\sin\theta_i\cdot Y_i)$;\\ 
\hspace*{7pt} (iv) \;$\cos(\widehat{fU,fV})=\cos\sphericalangle (\sum_{i=1}^k\sin\theta_i\cdot U_i,\ \sum_{i=1}^k\sin\theta_i\cdot V_i)$.
\end{theorem}

\begin{corollary}\label{p18}
For $X,Y\in (\oplus_{i=1}^kD_i)\verb=\=\{0\}$ and $U,V\in (\oplus_{i=1}^kw(D_i))\verb=\=\{0\}$ with ${M}_{\widehat{X,Y}}$ and ${M}_{\widehat{U,V}}$ nonempty, denoting $X= \sum_{i=1}^kX_i$, $Y=\sum_{i=1}^kY_i$, $U=\nolinebreak\sum_{i=1}^kU_i$, and $V=\sum_{i=1}^kV_i$, where $X_i,Y_i\in D_i$ and $U_i,V_i\in w(D_i)$, $i=\overline{1,k}$, we have: \\ 
\hspace*{7pt} (i) \ \;\,$g(X,Y)=\sum_{i=1}^k\frac{1}{\sin^2\theta_i}g(wX_i,wY_i)$;\\ 
\hspace*{7pt} (ii) \;\,$g(U,V)=\sum_{i=1}^k\frac{1}{\sin^2\theta_i}g(fU_i,fV_i)$;\\ 
\hspace*{7pt} (iii) \,$\cos(\widehat{X,Y})=\cos\sphericalangle (\sum_{i=1}^k\frac{1}{\sin\theta_i}\cdot wX_i,\ \sum_{i=1}^k\frac{1}{\sin\theta_i}\cdot wY_i)$;\\ 
\hspace*{7pt} (iv) \;$\cos(\widehat{U,V})=\cos\sphericalangle (\sum_{i=1}^k\frac{1}{\sin\theta_i}\cdot fU_i,\ \sum_{i=1}^k\frac{1}{\sin\theta_i}\cdot fV_i)$.
\end{corollary}

\begin{corollary}\label{p19}
For any $i\in \{1,\ldots ,k\}$, $X_i,Y_i\in D_i\verb=\=\{0\}$, $U_i,V_i\in w(D_i)\verb=\=\{0\}$ with ${M}_{\widehat{X_i,Y_i}}$ and ${M}_{\widehat{U_i,V_i}}$ nonempty, we have: \\ 
\hspace*{7pt} (i) \ \;\,$g(fwX_i,fwY_i)=\sin^4\theta_i \cdot g(X_i,Y_i)$;\\ 
\hspace*{7pt} (ii) \;\,$g(wfU_i,wfV_i)=\sin^4\theta_i \cdot g(U_i,V_i)$;\\ 
\hspace*{7pt} (iii) \,$\cos(\widehat{fwX_i,fwY_i})=\cos(\widehat{X_i,Y_i})$;\\ 
\hspace*{7pt} (iv) \;$\cos(\widehat{wfU_i,wfV_i})=\cos(\widehat{U_i,V_i})$.
\end{corollary}

\begin{corollary}\label{p20} 
For $X,Y\in \oplus_{i=1}^kD_i\verb=\=\{0\}$ and $U,V\in \oplus_{i=1}^kw(D_i)\verb=\=\{0\}$ with ${M}_{\widehat{X,Y}}$ and ${M}_{\widehat{U,V}}$ nonempty, denoting $X= \sum_{i=1}^kX_i$, $Y=\sum_{i=1}^kY_i$,\linebreak $U=\nolinebreak\sum_{i=1}^kU_i$, and $V=\sum_{i=1}^kV_i$, where $X_i,Y_i\in D_i$ and $U_i,V_i\in w(D_i)$, $i=\overline{1,k}$, we have: \\ 
\hspace*{7pt} (i) \ \;\,$g(fwX,fwY)=\sum_{i=1}^k\sin^4\theta_i \cdot g(X_i,Y_i)$;\\ 
\hspace*{7pt} (ii) \;\,$g(wfU,wfV)=\sum_{i=1}^k\sin^4\theta_i \cdot g(U_i,V_i)$;\\ 
\hspace*{7pt} (iii) \,$\cos(\widehat{fwX,fwY})=\cos\sphericalangle (\sum_{i=1}^k\sin^2\theta_i\cdot X_i,\ \sum_{i=1}^k\sin^2\theta_i\cdot Y_i)$;\\ 
\hspace*{7pt} (iv) \;$\cos(\widehat{wfU,wfV})=\cos\sphericalangle (\sum_{i=1}^k\sin^2\theta_i\cdot U_i,\ \sum_{i=1}^k\sin^2\theta_i\cdot V_i)$.
\end{corollary}

\begin{remark}\label{p91}
All the results got are, in particular, valid in a $k$-slant submanifold framework, that is, for $M$ a $k$-slant submanifold of $(\overline{M},\varphi,\xi,\eta,g)$, considering, for $TM=\oplus_{i=0}^k{D_i}\oplus \langle \xi\rangle$, the distribution $D=\oplus_{i=0}^k{D_i}$.
\end{remark}

%%%%%%%%%%%%%%%%%%%%%%%%%
\subsection{The dual $k$-slant distribution in almost paracontact metric \newline geometry}\label{subsec_epsilon=1}

Let $(\overline{M},\varphi,\xi,\eta,g)$ be an almost paracontact metric manifold, $k\in \mathbb{N}^*$, and $M$ be $\overline{M}$ or an immersed submanifold of $\overline{M}$. 
Let $D=\oplus_{i=0}^k{D_i}$ be a \mbox{$k$-slant} distribution on $M$ with $D_0$ the invariant component such that $\xi \perp D$, and let $G$ be the orthogonal complement of $D\oplus \langle \xi\rangle$ in $T\overline{M}$. 

\begin{remark}\label{p154}
All the results obtained in the almost contact metric case remain valid, with similar justifications, in the almost paracontact metric case (that is, for an almost (1)-contact metric manifold) with corresponding sign modifications where necessary. More precisely, relationships \eqref{9}-\eqref{36}, Lemma \ref{p92}, Propositions \ref{p2}, \ref{p3}, \ref{p5}, \ref{p67}-\ref{p13}, \ref{p66}, \ref{p14}, \ref{p16}, Theorems \ref{p1}, \ref{p9}, \ref{p21}, \ref{p17}, Definition \ref{p10}, Remark \ref{p11}, and Corollaries \ref{p0}, \ref{p62}, \ref{p12}, \ref{p15}, \ref{p18}--\ref{p20} remain further valid as they were stated. 
\end{remark}

Changes will appear in the following statements: Theorem \ref{p168}, Propositions \ref{p23}, \ref{p4}, \ref{p6}, Corollaries \ref{p69}, \ref{p60}, and Remark \ref{p70}, which become: 

\begin{proposition}\label{p93}
\begin{equation}
f^2X=\sum_{i=0}^k\cos^2\theta_i\cdot pr_iX\, \text{ for any } X\in D.
\nonumber
\end{equation}
\end{proposition}

\begin{corollary}\label{p94}
\begin{equation}
fwX=\sum_{i=1}^k \sin^2 \theta_i \cdot pr_iX\, \text{ for any } X\in D.
\nonumber 
\end{equation}
\end{corollary}

\begin{remark}\label{p95}
For $\theta_j=\frac{\pi}{2}$ and $X_j\in D_j$, we have
$fwX_j= X_j$ and hence $wfU_j=\nolinebreak U_j$ for any $U_j\in w(D_j)$,
so $f|_{w(D_j)}: w(D_j) \rightarrow D_j$ and \linebreak $w|_{D_j}: D_j \rightarrow w(D_j)$ are inverse to each other in the case $\epsilon=1$. 
\end{remark}

\begin{proposition}\label{p96}
For any $X\in D\oplus \langle \xi\rangle$ and\, $U\in G$, we have:
\begin{align}
f^2X+fwX &= X-\eta(X)\xi, \nonumber\\ 
wfX+w^2X &=0, \nonumber \\ 
f^2U+fwU &=0, \nonumber \\ 
wfU+w^2U &= U. \nonumber
\end{align}
\end{proposition}

\begin{corollary}\label{p97}
For any\, $U_0,V_0\in H$, we have: 
\begin{align}
w^2U_0&=U_0\,,\nonumber\\ 
g(wU_0,wV_0)&=g(U_0,V_0)\nonumber,\\ 
|wU_0|&=|U_0|.\nonumber
\end{align}
\end{corollary}

\begin{proposition}\label{p98}
For any\, $U\in w(D)$, $U=\sum_{i=1}^kU_i$ with $U_i\in w(D_i)$, we have:
\begin{equation}
wfU=\sum_{i=1}^k\sin^2\theta_i\cdot U_i\,, 
\nonumber 
\end{equation}
\begin{equation}
w^2U=\sum_{i=1}^k\cos^2\theta_i\cdot U_i\,. 
\nonumber 
\end{equation}
\end{proposition}

\begin{remark}\label{p99}
All the results related to $k$-slant distributions on an arbitrary submanifold of the almost paracontact metric manifold $(\overline{M},\varphi,\xi,\eta,g)$ can be transferred to any $k$-slant submanifold $M$ of $\overline{M}$ by taking $D=\oplus_{i=0}^k{D_i}$ if $TM=\oplus_{i=0}^k{D_i}\oplus \langle \xi\rangle$. Thus, the obtained results are also valid when considered in a $k$-slant submanifold framework. 
\end{remark}

Taking into account Propositions \ref{p167}, \ref{p93} and Remark \ref{p33}, we obtain 

\begin{theorem}\label{p171}
Let $\mathfrak D$ be a non-null distribution on $M$ such that $\mathfrak D\perp\nolinebreak \xi$ and $\mathfrak D$ is decomposable into an orthogonal sum of regular distributions, $\mathfrak D=\nolinebreak\oplus_{i=0}^k{\mathfrak D_i}$ with $\mathfrak D_i\neq \{0\}$ for $i=\overline{1,k}$ and $\mathfrak D_0$ invariant (possible null). Let $pr_i$ denote the projection operator onto $\mathfrak D_i$ for $i=\overline{0,k}$, $f$ the component of $\varphi$ into $\mathfrak D$, and $\theta_0=0$. If $f(\mathfrak D_i)\subseteq \mathfrak D_i$ for $i=\overline{1,k}$, then the following assertions are equivalent: \\ 
\hspace*{7pt} (a) There exist $k$ distinct values $\theta_i\in (0,\frac{\pi}{2}]$, $i=\overline{1,k}$, such that  
\begin{equation}\nonumber
f^2X= \sum_{i=0}^k\cos^2\theta_i\cdot pr_iX \ \text{for any} \  X\in \mathfrak D;
\end{equation}
\hspace*{7pt} (b) $\mathfrak D$ is a $k$-slant distribution with slant angles $\theta_i$ corresponding to $\mathfrak D_i$, $i=\nolinebreak\overline{1,k}$. 
\end{theorem}

\begin{remark}\label{p172}
Theorem \ref{p171} provides a necessary and sufficient condition for a submanifold $M$ of an almost paracontact metric manifold $\overline M$ to be a \mbox{$k$-slant} submanifold, $\mathfrak D$ being the distribution on $M$ given by $\mathfrak D=\oplus_{i=0}^k{\mathfrak D_i}$ if $TM=\oplus_{i=0}^k{\mathfrak D_i}\oplus \langle \xi\rangle$.
\end{remark}

\begin{example}\label{ex2}
In Example \ref{ex1}, in the setting given by an almost ($\epsilon$)-contact metric manifold $(\overline{M}, \varphi,\xi,\eta, g )$, we consider the distributions $G_j:=\langle e_{4j+1}, e_{4j+2}\rangle$, $j=\overline{1,k}$, in $(TM)^{\perp}$. Then, $\oplus_{j=1}^kG_j$ is the dual $k$-slant distribution of $\oplus_{j=1}^kD_j$. We have $f(G_j)=D_j$ for $j=\overline{1,k}$, so $\oplus_{j=1}^kD_j$ is the dual $k$-slant distribution of $\oplus_{j=1}^kG_j$.
\end{example}

\begin{remark}\label{p202}
We describe below the notion of \textit{skew CR submanifold of an almost paracontact metric manifold}, relating it to the concept of $k$-slant submanifold. 

Let $M$ be an immersed submanifold of an almost paracontact metric manifold $(\overline{M},\varphi,\xi,\eta,g)$, and let $fZ$ be the tangential component of $\varphi Z$ for any $Z\in T{M}$. 

In view of Lemmas \ref{lema1}, \ref{lema2}  and Remark \ref{p194}, $f$ is symmetric, so $f^2$ is symmetric. Denoting, for any $x\in M$, by $\lambda_i(x)$, $i=\overline{1,m(x)}$, the distinct eigenvalues of $f_x^2$ acting on the tangent space $T_xM$, these eigen\-val\-ues are all real and nonnegative. 
In view of (\ref{7}) and $\varphi \xi=0$, we have $|fX|\leq\nolinebreak|\varphi X|\leq\nolinebreak |X|$ for any $X\in T{M}$; hence, all the $\lambda_i(x)$'s are contained in $[0,1]$. 
For any $x\in M$, let $\mathfrak D_x^i$ denote the eigen\-space of $f_x^2$ cor\-re\-spond\-ing to $\lambda_i(x)$, $i=\nolinebreak\overline{1,m(x)}$;
the tangent space $T_xM$ of $M$ at $x$ has the following orthogonal decomposition: 
$$T_xM={\mathfrak D_x^1}\oplus\ldots\oplus{\mathfrak D_x^{m(x)}}.$$
We notice that every eigenspace $\mathfrak D_x^i$ is invariant under $f_x$. 

Now, let us consider that $M$ is a skew CR submanifold of $\overline{M}$, that is: 
\begin{enumerate}
\item $m(x)$ does not depend on $x\in M$ (we will denote $m(x)=m$); 
\item the dimension of ${\mathfrak D_x^i}$, $i=\overline{1,m}$, is independent of $x\in M$; 
\item each $\lambda_i(\cdot)$ is constant on $M$ (we will denote $\lambda_i(x)=\lambda_i)$. 
\end{enumerate}

Thus, for any $x\in M$, there is the same number $m$ of distinct eigenvalues of $f_x^2$, these, denoted by $\lambda_1$, \dots , $\lambda_m$, being independent of $x$. Denoting by $\mathfrak D_i$ the distribution corresponding to the family $\{\mathfrak D_x^i: x\in M\}$ for $i=\overline{1,m}$, we deduce that $TM$ accepts the orthogonal decomposition: 
$$TM={\mathfrak D_1}\oplus\ldots\oplus{\mathfrak D_m}\,.$$

Moreover, the distribution $\mathfrak D_i$ is invariant under $f$ for every $i$.

Since $f\xi=\varphi\xi=0$, one of the $\mathfrak D_i$'s contains $\langle \xi \rangle$; let $\mathfrak D_m$ be that distribution, so $\lambda_m=0$. 
Let $\mathfrak D'_m$ denote the orthogonal complement of $\langle\xi\rangle$ in $\mathfrak D_m$. Then, we have $\mathfrak D_m=\langle\xi\rangle\oplus\mathfrak D'_m$. We notice that, if $\mathfrak D'_m$ is non-null, then it is a slant distribution with slant angle $\frac{\pi}{2}$.

For any $i\in\{1,\ldots,m-1\}$ and $X\in  \mathfrak{D}_i\verb=\=\{0\}$, denoting $\alpha_i=\sqrt{\lambda_i}\in\nolinebreak (0,1]$, we have 
$$|fX|^2=\lambda_ig(X,X)= \alpha_i^2 |X|^2= \alpha_i^2 |\varphi X|^2$$
\noindent and $|fX|=\alpha_i |X|=\alpha_i |\varphi X|$; hence, $\alpha_i=\cos\zeta_i$\,, where $\zeta_i$ is the angle between $\varphi X$ and $TM$, the same for any nonzero $X\in \mathfrak D_i$. 

The distributions $\mathfrak{D}_i$, $i=\overline{1,m-1}$, are slant distributions with distinct slant angles $\zeta_i$ except that distribution which corresponds to $\alpha_i=1$ and is invariant (with respect to $\varphi$) if such one exists. 

Hence, since $\oplus_{i=1}^{m-1}\mathfrak{D}_i\oplus\mathfrak{D}'_m$ does not reduce to an invariant distribution under $\varphi$, $M$ is a $k$-slant submanifold of $\overline{M}$, where $k$ is one of the values: $m-2,\, m-1,\, m$. 

\end{remark}

We conclude: 

\begin{proposition}\label{p124}
Any skew CR submanifold of an almost paracontact metric manifold is a $k$-slant submanifold. 
\end{proposition}

\begin{proposition}\label{p211}
Any $k$-slant submanifold of an almost paracontact metric man\-i\-fold which is not an anti-invariant or a CR submanifold is a skew CR submanifold. 
\end{proposition}

%%%%%%%%%%%%%%%%%%%%%%%%%%%%%%%%%
\section{$k$-slant distributions in almost Hermitian and almost product Riemannian settings}\label{alm_hermitian}

In the sequel, we  will provide a unitary approach for the almost Hermitian and almost product Riemannian settings. 

Let $\overline{M}$ be a smooth manifold, $g$ a Riemannian metric on $\overline{M}$, $\epsilon\in \{-1,1\}$, and $\varphi$ a $(1,1)$-tensor field on $\overline{M}$ satisfying 
\begin{equation}\label{h2}
\varphi^2=\epsilon I \text{ \ and \ }
g(\varphi X,Y)=\epsilon g(X,\varphi Y) 
\text{ for any $X,Y\in T\overline{M}$}.
\end{equation}
We immediately get
\begin{equation}\label{h3} 
g(\varphi X,\varphi Y)=g(X,Y)\text{ for any }X,Y\in T\overline{M}. 
\end{equation}

\begin{remark}
$(\overline{M},\varphi,g)$ is an \textit{almost Hermitian manifold} for $\epsilon=-1$, and it is an \textit{almost product Riemannian manifold} for $\epsilon=1$. 
\end{remark}

\begin{remark}\label{p32}
In view of (\ref{h3}), we notice that $\varphi$ is an isometry; hence, the orthogonality of vector fields is invariant under $\varphi$. 
\end{remark}

\medskip 
For $M$ an immersed submanifold of $\overline{M}$ and $k\in \mathbb{N}^*$, 
we have the following equivalent formulation of the definition of a $k$-slant submanifold. 

\begin{definition}\label{p26}
We say that $M$ is a \textit{$k$-slant submanifold} of $\overline{M}$ if there exists an orthogonal decomposition of $TM$ into regular distributions, 
$$TM=\oplus_{i=0}^k{D_i}$$
with $D_i\neq \{0\},\ i= \overline{1,k}$, and $D_0$ possible null, and there exist distinct values $\theta_i\in (0,\frac{\pi}{2}]$, $i=\overline{1,k}$, such that: \\ 
\hspace*{7pt} (i)\; $\widehat{(\varphi X, D_i)}=\theta_i$ for any $X\in D_i\verb=\=\{0\}$, $i=\overline{1,k}$;\\ 
\hspace*{7pt} (ii)\; $\varphi X\in D_0$ for any $X\in D_0$ (i.e.,  $f(D_0)\subseteq D_0$, and $\widehat{(\varphi X, TM)}=\nolinebreak 0=: \nolinebreak\theta_0$ for $X\in D_0\verb=\=\{0\}$);\\ 
\hspace*{7pt} (iii) $f(D_i)\subseteq D_i$, $i=\overline{1,k}$. 
\end{definition}

\begin{remark}\label{32}
In view of (iii) and Remark \ref{p25}, we have: \\ 
(a) Condition (i) is equivalent to \\ 
\hspace*{7pt} (i') $\widehat{(\varphi X, TM)}=\theta_i$ for any $X\in D_i\verb=\=\{0\}$, $i=\overline{1,k}$;\\ 
(b) Condition (iii) can be replaced by \\ 
\hspace*{7pt} (iii') $\varphi(D_i)\perp D_j$ for any $i\neq j$ from $\{1,\ldots,k\}$.
\end{remark}

\medskip 
In the sequel, until the end of the section, we will consider $M$ to be $\overline{M}$ or an immersed submanifold of $\overline{M}$ if not specified otherwise. 

Let $k\in \mathbb{N}^*$, $D=\oplus_{i=0}^k{D_i}$ be a $k$-slant distribution on ${M}$ with $D_0$ the invariant component, and $G$ be the orthogonal complement of $D$ in $T\overline{M}$. Let $\theta_0=0$ and $\theta_i$ denote the slant angle of $D_i$ for $i=\overline{1,k}$. For any $Z\in T\overline{M}$, as established, we denote by $fZ$ and $wZ$ the components of $\varphi Z_M$ in $D$ and in $G$, respectively. Also, let $pr_i$ denote the projection operator onto $D_i$, $i=\overline{0,k}$.

From (\ref{h2}) and Definition \ref{8} (ii), we obtain: 

\begin{remark}\label{p153}
\begin{equation}\label{77}
\varphi(D_0)=D_0,\ \ 
w(D_0)=\{0\}, \ \ f(D_0)=D_0; 
\nonumber 
\end{equation}
\begin{equation}\label{78}
\varphi^2(D_i)=D_i \, \text{ for } i=\overline{1,k}\,,\ \  
\varphi^2(G)=G;
\nonumber
\end{equation}
\begin{equation}\label{79}
f(\varphi X)=\epsilon X,\ \ w(\varphi X)=0 \,\text{ for any } X\in D; \nonumber
\end{equation}
\begin{equation}\label{80}
f(\varphi V)=0,\ \ w(\varphi V)= \epsilon V \,\text{ for any } V\in G. \nonumber
\end{equation}
\end{remark}

For $X_i\in D_i\verb=\=\{0\},\ i= \overline{1,k}$, in view of Definition \ref{8} (i), we have
$|fX_i|=\nolinebreak\cos \theta_i \cdot |\varphi X_i|$, from which
$g(f^2X_i,Y_i)= \cos^2 \theta_i \cdot g(\varphi^2 X_i, Y_i)$ for $X_i,Y_i \in\nolinebreak{D_i}$. Since $f(D_i)\subseteq D_i$, we get $
g(f^2X_i,Y)= \cos^2 \theta_i \cdot g(\varphi^2 X_i,Y)
$ for any $Y\in T{M}$, so $f^2X_i= \epsilon\cos^2\theta_i\cdot X_i \text{ for } X_i\in D_i$.
Since $f^2X=\epsilon X$ for $X\in D_0$, we get 

\begin{proposition}\label{p36}
\begin{equation}\label{44}
f^2X=\epsilon\sum_{i=0}^k\cos^2\theta_i\cdot pr_iX\, \text{ for any } X\in D.
\nonumber
\end{equation}
\end{proposition}

\begin{corollary}\label{p37}
\begin{equation}\label{46}
f(D_i)=D_i \,\text{ for any } i \text{ with } \theta_i\neq \frac{\pi}{2}\,. \nonumber
\end{equation}
\end{corollary}

Taking into account Propositions \ref{p167}, \ref{p36} and Remark \ref{p32}, we obtain 

\begin{theorem}\label{p173}
Let $\mathfrak D$ be a non-null distribution on $M$ decomposable into an orthogonal sum of regular distributions, $\mathfrak D=\oplus_{i=0}^k{\mathfrak D_i}$ with $\mathfrak D_i\neq \{0\}$ for $i=\nolinebreak\overline{1,k}$ and $\mathfrak D_0$ invariant (possible null). Let $pr_i$ denote the projection operator onto $\mathfrak D_i$ for $i=\overline{0,k}$, $f$ the component of $\varphi$ into $\mathfrak D$, and $\theta_0=0$. If $f(\mathfrak D_i)\subseteq \mathfrak D_i$ for $i=\overline{1,k}$, then the following assertions are equivalent: \\ 
\hspace*{7pt} (a) There exist $k$ distinct values $\theta_i\in (0,\frac{\pi}{2}]$, $i=\overline{1,k}$, such that  
\begin{equation}\nonumber
f^2X= \epsilon \sum_{i=0}^k\cos^2\theta_i\cdot pr_iX \, \text{ for any } X\in \mathfrak D;
\end{equation}
\hspace*{7pt} (b) $\mathfrak D$ is a $k$-slant distribution with slant angles $\theta_i$ corresponding to $\mathfrak D_i$, $i=\nolinebreak\overline{1,k}$. 
\end{theorem}

\begin{remark}\label{p174}
Theorem \ref{p173} provides a necessary and sufficient condition for a submanifold $M$ of $\overline M$ to be a $k$-slant submanifold, considering $\mathfrak D=TM$ if $TM=\oplus_{i=0}^k{\mathfrak D_i}$.
\end{remark}

\medskip 
Let us now return to the $k$-slant distribution $D=\oplus_{i=0}^k{D_i}$ on ${M}$ with its orthogonal complement $G$ in $T\overline{M}$. 
For $U\in G$ and $X\in D_0$, we have $g(fU,X)=0$, which implies
$f(G)\perp D_0$, so
\begin{equation}\label{81}
f(G)\subseteq \oplus_{i=1}^k{D_i}\,.
\end{equation}

For $U\in G$ with $U\perp w(\oplus_{i=1}^k{D_i})$, in view of Lemma \ref{lema1}, we have
$fU\perp \nolinebreak\oplus_{i=1}^k{D_i}$, and, using (\ref{81}), we get $fU=0$.
Taking into account (\ref{17}), we have the following orthogonal decomposition: 

\begin{theorem}\label{p39}
\begin{equation}\label{47}
G=\oplus_{i=1}^kw({D_i})\oplus H, \text{ where }f(H)=\{0\}.
\end{equation}
\end{theorem}

\medskip 
From the decomposition (\ref{47}), we get 
$\varphi V=wV\in H$\, and\, $w^2V= \nolinebreak\epsilon V\in\nolinebreak H$ for $V\in H$; therefore, $w^2(H)=H$, and, consequently,

\begin{corollary}\label{p71}
$$
\varphi (H)=w(H)=H. 
$$
\end{corollary}

\begin{remark}\label{p63}
For $\theta_j=\frac{\pi}{2}$ and $X_j\in D_j$, we have
$fwX_j=\epsilon X_j$ and, thus, $wfV_j= \nolinebreak\epsilon V_j$ for any $V_j\in w(D_j)$,
so $f|_{w(D_j)}: w(D_j) \rightarrow D_j$ and \linebreak $w|_{D_j}: D_j \rightarrow w(D_j)$ are inverse to each other for $\epsilon=1$ but anti-inverse for $\epsilon=-1$. 
\end{remark}

\begin{proposition}\label{p43}
For any $X,Y\in D$, we have: 
\begin{align}
g(\varphi X,\varphi Y)&= \sum_{i=0}^k g(pr_iX, pr_iY), \nonumber\\ 
g(fX,fY)&=\sum_{i=0}^k \cos^2 \theta_i \cdot g(pr_iX, pr_iY), \nonumber\\ 
g(wX,wY)&=\sum_{i=1}^k \sin^2 \theta_i \cdot g(pr_iX, pr_iY). \nonumber 
\end{align}
\end{proposition}

\medskip 
Using similar proofs as in the almost contact metric case, we get the following results. 

\begin{proposition}\label{p40}
\begin{align*}
f(w(D_i)) &=D_i \,\text{ for } i=\overline{1,k}\,; \\ 
f(G) &=\oplus_{i=1}^k{D_i}\,. 
\end{align*}

Moreover, $w|_{D_i}$ and $f|_{w(D_i)}$ are injective, for $i= \overline{1,k}$; therefore, $w(D_i)$ and $D_i$ localized in any point of $M$ are isomorphic; hence, both are regular and have the same dimension.
\end{proposition}

\begin{proposition}\label{p24}
For any $X\in D$ and\, $U\in G$, we have:
\begin{align}
f^2X+fwX &= \epsilon X, \nonumber\\ 
wfX+w^2X &=0, \nonumber \\ 
f^2U+fwU &=0, \nonumber \\ 
wfU+w^2U &=\epsilon U. \nonumber
\end{align}
\end{proposition}

\begin{corollary}\label{p38}
\begin{equation}
fwX=\epsilon \sum_{i=1}^k \sin^2 \theta_i \cdot pr_iX\, \text{ for any } X\in D.
\nonumber 
\end{equation}
\end{corollary}

\medskip 
In view of Theorem \ref{p39} and Proposition \ref{p24}, we deduce: 

\begin{corollary}\label{p61}
For any $U_0, V_0\in H$, we have: 
\begin{align}
w^2U_0&=\epsilon U_0\,,\nonumber\\ 
g(wU_0,wV_0)&=g(U_0,V_0)\nonumber,\\ 
|wU_0|&=|U_0|.\nonumber
\end{align}

\end{corollary}

\begin{proposition}\label{p41}
\[
w^2(D_i)=
\begin{cases}
w(D_i) & \text{for }\, \theta_i\neq \frac{\pi}{2}\,,\\ 
\ \{0\} & \textit{for }\, \theta_i=\frac{\pi}{2}\,.
\end{cases}
\]
\end{proposition}

\begin{proposition}\label{p42}
For any $U\in w(D)$, $U=\sum_{i=1}^kU_i$ with $U_i\in w(D_i)$, we have:
\begin{equation}\label{75}
wfU=
\epsilon\sum_{i=1}^k\sin^2\theta_i\cdot U_i\,, 
\nonumber 
\end{equation}
\begin{equation}\label{74}
w^2U=
\epsilon\sum_{i=1}^k\cos^2\theta_i\cdot U_i\,. 
\nonumber 
\end{equation}
\end{proposition}

\begin{proposition}\label{p68}
For any $U,V\in w(D)$, $U=\sum_{i=1}^kU_i$, $V=\sum_{i=1}^kV_i$ with $U_i,V_i\in w(D_i)$, $i=\overline{1,k}$, we have: 
\begin{align}
g(fU,fV)&=\sum_{i=1}^k \sin^2 \theta_i \cdot g(U_i, V_i), \nonumber\\ 
g(wU,wV)&=\sum_{i=1}^k \cos^2 \theta_i \cdot g(U_i, V_i), \nonumber\\ 
g(\varphi U,\varphi V)&= \sum_{i=1}^k g(U_i, V_i). \nonumber
\end{align}
\end{proposition}

\begin{proposition}\label{p44}
For any $X=\sum_{i=0}^kX_i$ and\, $U=\sum_{i=1}^kU_i$ with $X_0\in D_0$, $X_i\in D_i$, and\, $U_i\in w(D_i)$, $i=\overline{1,k}$, we have:
$$|fX|^2={\sum_{i=0}^k\cos^2\theta_i\cdot |X_i|^2}\quad\, \text{and}\,\quad |wU|^2={\sum_{i=1}^k\cos^2\theta_i\cdot |U_i|^2}.$$

In particular, we obtain:\, $|fX_0|= |X_0|$\, for $X_0\in D_0$; 
$$|fX_i|=\cos\theta_i\cdot |X_i|, \quad |wU_i|=\cos\theta_i\cdot  |U_i|\, \text{ for } X_i\in D_i,\, U_i\in w(D_i),\, i=\overline{1,k}.
$$
\end{proposition}

\begin{corollary}\label{p52}
For any $i\in \{1,\ldots,k\}$, $X_i\in D_i$, and\, $U_i\in w(D_i)$, we have: 
$$|wX_i|=\sin\theta_i \cdot |X_i|\quad\, \text{and}\,\quad |fU_i|=\sin\theta_i \cdot |U_i|.$$

\end{corollary}

\medskip 
For more general vector fields, we get 

\begin{proposition}\label{p53}
For $X=\sum_{i=1}^kX_i$, $U=\sum_{i=1}^kU_i$ with $X_i\in D_i$, 
$U_i\in\nolinebreak w(D_i)$, $i=\overline{1,k}$, we have: 
$$|wX|^2={\sum_{i=1}^k\sin^2\theta_i\cdot |X_i|^2}\ \ \text{ and }\ \ |fU|^2={\sum_{i=1}^k\sin^2\theta_i\cdot |U_i|^2}.$$
\end{proposition}

\begin{proposition}\label{p45}
For any $i\in\{1,\ldots ,k\}$ with $\theta_i\neq \frac{\pi}{2}$ and for $X_i, Y_i\in\nolinebreak D_i\verb=\=\{0\}$, $U_i, V_i\in w(D_i)\verb=\=\{0\}$, $X_0, Y_0\in D_0\verb=\=\{0\}$, $U_0, V_0\in H\verb=\=\{0\}$, $\overline{X}, \overline{Y}\in\nolinebreak (D\oplus\nolinebreak G)\verb=\=\{0\}$, if\, ${M}_{\widehat{X_0,Y_0}}\,,\, {M}_{\widehat{X_i,Y_i}}\,,\, {M}_{\widehat{U_0,V_0}}\,,\, {M}_{\widehat{U_i,V_i}}\,,\, {M}_{\widehat{\overline{X},\overline{Y}}}$ are nonempty, we have: \\ 
\hspace*{7pt} (i) \ \;$\cos(\widehat{fX_0,fY_0})=\cos(\widehat{\varphi X_0,\varphi Y_0})=\cos(\widehat{X_0,Y_0})$;\\ 
\hspace*{7pt} (ii) \ $\cos(\widehat{fX_i,fY_i})=\cos(\widehat{\varphi X_i,\varphi Y_i})=\cos(\widehat{X_i,Y_i})$;\\ 
\hspace*{7pt} (iii) \,$g(wU_i,wV_i)=\cos^2\theta_i \cdot g(U_i,V_i)$;\\ 
\hspace*{7pt} (iv) \;$\cos(\widehat{wU_0,wV_0})=\cos(\widehat{U_0,V_0})=\cos(\widehat{\varphi U_0,\varphi V_0})$;\\ 
\hspace*{7pt} (v) \ \;$\cos(\widehat{wU_i,wV_i})=\cos(\widehat{U_i,V_i})=\cos(\widehat{\varphi U_i,\varphi V_i})$;\\ 
\hspace*{7pt} (vi) \ $\cos(\widehat{\varphi \overline{X},\varphi \overline{Y}})=\cos(\widehat{\overline{X},\overline{Y}})$.
\end{proposition}

\begin{theorem}\label{p47}
The distribution $G=\oplus_{i=1}^kw(D_i)\oplus H$ is a $k$-slant distribution with $H$ the invariant component and $\oplus_{i=1}^kw(D_i)$ the proper \mbox{$k$-slant} com\-po\-nent, every slant distribution $w(D_i)$ having the same slant angle as $D_i$, $i= \nolinebreak\overline{1,k}$. 
\end{theorem}

\begin{definition}\label{p48}
We will call $\oplus_{i=1}^kw(D_i)$ \textit{the dual $k$-slant distribution} of\, $\oplus_{i=1}^kD_i$. 
\end{definition}

\begin{remark}\label{p49}
In the same way we constructed the dual of the proper \mbox{$k$-slant} distribution $\oplus_{i=1}^kD_i$ by means of $w$, we can construct the dual of the proper \mbox{$k$-slant} component $\oplus_{i=1}^kw(D_i)$ of the distribution $G$ by means of $f$. This will be $f(\oplus_{i=1}^kw(D_i))=\oplus_{i=1}^kfw(D_i)$. 
\end{remark}

\begin{corollary}\label{p50}
The dual of the proper $k$-slant distribution $\oplus_{i=1}^kw(D_i)$, which is $\oplus_{i=1}^kf(w(D_i))$, is precisely the $k$-slant distribution $\oplus_{i=1}^kD_i$.
\end{corollary}

\medskip 
In view of Proposition \ref{p40} and Corollary \ref{p37}, denoting $w(D_i)$ by $G_i$, we obtain: 

\begin{proposition}\label{p65}
\begin{align*}
w(f(G_i)) & =G_i\text{ for }i=\overline{1,k}\,;\\ 
f^2(G_i) & =\left\{\begin{array}{ll}
D_i & \text{if }\theta_i\neq\frac{\pi}{2}\,, \\ 
\{0\} & \text{if }\theta_i= \frac{\pi}{2}\,. 
\end{array}
\right.
\end{align*}
\end{proposition}

In view of Corollary \ref{p38}, we immediately get 

\begin{lemma}
Let $X, Y\in \oplus_{i=1}^kD_i$, $U, V\in \oplus_{i=1}^kw(D_i)$, and $x\in {M}$. Then: \\ 
\hspace*{7pt} (i) \ \,$X_x\neq 0$ if and only if $(wX)_x\neq 0$;\\ 
\hspace*{7pt} (ii) \ $U_x\neq 0$ if and only if $(fU)_x\neq 0$;\\ 
\hspace*{7pt} (iii) ${M}_{\widehat{X,Y}}={M}_{\widehat{wX,wY}}$\,, and 
${M}_{\widehat{U,V}}={M}_{\widehat{fU,fV}}$\,. 
\end{lemma}

\medskip 
The angle between two angular compatible nonzero vector fields of a slant distribution remains in\-vari\-ant when passing to the corresponding pair of vector fields in the dual distribution, as specified in the following proposition. 

\begin{proposition}\label{p51}
For\, $i\in \{1,\ldots ,k\}$ and\, $X_i, Y_i\in D_i\verb=\=\{0\}$, 
\newline $U_i, V_i\in\nolinebreak w(D_i)\verb=\=\{0\}$ with ${M}_{\widehat{X_i,Y_i}}$ and ${M}_{\widehat{U_i,V_i}}$ nonempty, we have: \\ 
\hspace*{7pt} (i) \ \;\,$g(wX_i,wY_i)=\sin^2\theta_i \cdot g(X_i,Y_i)$;\\ 
\hspace*{7pt} (ii) \;\,$g(fU_i,fV_i)=\sin^2\theta_i \cdot g(U_i,V_i)$;\\ 
\hspace*{7pt} (iii)\, $\cos(\widehat{wX_i,wY_i})=\cos(\widehat{X_i,Y_i})$;\\ 
\hspace*{7pt} (iv)\; $\cos(\widehat{fU_i,fV_i})=\cos(\widehat{U_i,V_i})$.
\end{proposition}

\begin{theorem}\label{p54}
 $f$ and $w$ restricted to $D_i$ or to $w(D_i)$, $i=\overline{1,k}$ \textup(excepting $f|_{D_j}$ and $w|_{w(D_j)}$ with $\theta_j=\frac{\pi}{2}$, in which case $f|_{D_j}$ and $w|_{w(D_j)}$ are vanishing\textup), $f|_{D_0}$, $w|_H$, and $\varphi|_{D\oplus G}$ are conformal maps. 
\end{theorem}

\medskip 
Taking into account the orthogonal decompositions of $D$ and $G$ and the properties stated above, for any two angular compatible vector fields of a \mbox{$k$-slant} distribution, there is in the dual $k$-slant distribution a pair of angular compatible vector fields which form the same angle.

\begin{theorem}\label{p55} 
Let $X_i,Y_i\in D_i$, $U_i,V_i\in w(D_i)$, $i=\overline{1,k}$, and $X=\nolinebreak\sum_{i=1}^kX_i$, $Y=\sum_{i=1}^kY_i$, $U=\sum_{i=1}^kU_i$, $V=\sum_{i=1}^kV_i$. 
If\, $X,\,Y,\,U,\,V$ are nonzero with ${M}_{\widehat{X,Y}}$ and ${M}_{\widehat{U,V}}$ nonempty, then we have: \\ 
\hspace*{7pt} (i) \ \;\,$g(wX,wY)=\sum_{i=1}^k\sin^2\theta_i \cdot g(X_i,Y_i)$;\\ 
\hspace*{7pt} (ii) \;\,$g(fU,fV)=\sum_{i=1}^k\sin^2\theta_i \cdot g(U_i,V_i)$;\\ 
\hspace*{7pt} (iii) \,$\cos(\widehat{wX,wY})=\cos\sphericalangle (\sum_{i=1}^k\sin\theta_i\cdot X_i,\ \sum_{i=1}^k\sin\theta_i\cdot Y_i)$;\\ 
\hspace*{7pt} (iv) \;$\cos(\widehat{fU,fV})=\cos\sphericalangle (\sum_{i=1}^k\sin\theta_i\cdot U_i,\ \sum_{i=1}^k\sin\theta_i\cdot V_i)$.
\end{theorem}

\begin{corollary}\label{p56}
Let $X_i,Y_i\in D_i$, $U_i,V_i\in w(D_i)$, $i=\overline{1,k}$, and $X=\nolinebreak\sum_{i=1}^kX_i$, $Y=\sum_{i=1}^kY_i$, $U=\sum_{i=1}^kU_i$, $V=\sum_{i=1}^kV_i$. If\, $X,\,Y,\,U,\,V$ are nonzero, and ${M}_{\widehat{X,Y}}$, ${M}_{\widehat{U,V}}$ are nonempty, we have: \\ 
\hspace*{7pt} (i) \ \;\,$g(X,Y)=\sum_{i=1}^k\frac{1}{\sin^2\theta_i}g(wX_i,wY_i)$;\\ 
\hspace*{7pt} (ii) \;\,$g(U,V)=\sum_{i=1}^k\frac{1}{\sin^2\theta_i}g(fU_i,fV_i)$;\\ 
\hspace*{7pt} (iii) \,$\cos(\widehat{X,Y})=\cos\sphericalangle (\sum_{i=1}^k\frac{1}{\sin\theta_i}\cdot wX_i,\ \sum_{i=1}^k\frac{1}{\sin\theta_i}\cdot wY_i)$;\\ 
\hspace*{7pt} (iv) \;$\cos(\widehat{U,V})=\cos\sphericalangle (\sum_{i=1}^k\frac{1}{\sin\theta_i}\cdot fU_i,\ \sum_{i=1}^k\frac{1}{\sin\theta_i}\cdot fV_i)$.
\end{corollary}

\begin{corollary}\label{p57}
For\, $i\in \{1,\ldots ,k\}$, $X_i,Y_i\in D_i\verb=\=\{0\}$, and $U_i,V_i\in w(D_i)\verb=\=\{0\}$, if ${M}_{\widehat{X_i,Y_i}}$ and ${M}_{\widehat{U_i,V_i}}$ are nonempty, we have: \\ 
\hspace*{7pt} (i) \ \;\,$g(fwX_i,fwY_i)=\sin^4\theta_i \cdot g(X_i,Y_i)$;\\ 
\hspace*{7pt} (ii) \;\,$g(wfU_i,wfV_i)=\sin^4\theta_i \cdot g(U_i,V_i)$;\\ 
\hspace*{7pt} (iii) \,$\cos(\widehat{fwX_i,fwY_i})=\cos(\widehat{X_i,Y_i})$;\\ 
\hspace*{7pt} (iv) \;$\cos(\widehat{wfU_i,wfV_i})=\cos(\widehat{U_i,V_i})$.
\end{corollary}

\begin{corollary}\label{p58} 
Let $X_i,Y_i\in D_i$, $U_i,V_i\in w(D_i)$, $i=\overline{1,k}$, and $X=\nolinebreak\sum_{i=1}^kX_i$, $Y=\sum_{i=1}^kY_i$, $U=\sum_{i=1}^kU_i$, $V=\sum_{i=1}^kV_i$ with $X,\,Y,\,U,\,V$ nonzero and ${M}_{\widehat{X,Y}}$, ${M}_{\widehat{U,V}}$ nonempty. Then, we have: \\ 
\hspace*{7pt} (i) \ \;\,$g(fwX,fwY)=\sum_{i=1}^k\sin^4\theta_i \cdot g(X_i,Y_i)$;\\ 
\hspace*{7pt} (ii) \;\,$g(wfU,wfV)=\sum_{i=1}^k\sin^4\theta_i \cdot g(U_i,V_i)$;\\ 
\hspace*{7pt} (iii) \,$\cos(\widehat{fwX,fwY})=\cos\sphericalangle (\sum_{i=1}^k\sin^2\theta_i\cdot X_i,\ \sum_{i=1}^k\sin^2\theta_i\cdot Y_i)$;\\ 
\hspace*{7pt} (iv) \;$\cos(\widehat{wfU,wfV})=\cos\sphericalangle (\sum_{i=1}^k\sin^2\theta_i\cdot U_i,\ \sum_{i=1}^k\sin^2\theta_i\cdot V_i)$.
\end{corollary}

\begin{remark}
All the results from this section are also valid for a \mbox{$k$-slant} submanifold $M$ of $(\overline{M},\varphi,g)$ by taking $D=TM$. 
\end{remark}

\begin{remark}\label{p203}
We describe below the notion of \textit{skew CR submanifold of an almost Hermitian} or \textit{almost product Riemannian manifold}, relating it to the concept of $k$-slant submanifold.  

Let $M$ be a skew CR submanifold of ($\overline{M},\varphi,g)$ with $\varphi$ and $g$ satisfying (\ref{h2}), where $\epsilon \in \{-1,1\}$. Let $fZ$ be the tangential component of $\varphi Z$ for any $Z\in T{M}$. 

In view of Lemmas \ref{lema1}, \ref{lema2} and Remark \ref{p194}, $f$ is skew-symmetric for $\epsilon =-1$ and symmetric for $\epsilon =1$, so $f^2$ is symmetric. We will denote by $\lambda_i(x)$, $i=\overline{1,m(x)}$, the distinct eigenvalues of $f_x^2$ acting on the tangent space $T_xM$ for $x\in M$. Since $\varphi$ verifies (\ref{h2}) and $f$ is the composition of a projection and an isometry, these eigenvalues are all  contained in $[-1,0]$ for $\epsilon =-1$ and in $[0,1]$ for $\epsilon =1$. Denoting, for any $x\in M$, by $\mathfrak D_x^i$ the eigenspace of $f_x^2$ cor\-re\-spond\-ing to $\lambda_i(x)$, $i=\overline{1,m(x)}$, since $M$ is a skew CR submanifold of $\overline{M}$, $m(x)$ is independent of $x$, and we will denote it with $m$, and the same is true for $\lambda_i(x)$, $i=\overline{1,m}$ (we will denote these values with $\lambda_1,\lambda_2,\ldots,\lambda_m$), and for the dimension of $\mathfrak D_x^i$, $i=\overline{1,m}$. Denoting by $\mathfrak D_i$ the distribution corresponding to the family $\{\mathfrak D_x^i: x\in M\}$ for $i=\overline{1,m}$, we get for $TM$ the orthogonal decomposition  
$$TM={\mathfrak D_1}\oplus\ldots\oplus{\mathfrak D_m}\,.$$

We notice that each distribution $\mathfrak D_i$ is invariant under $f$. Moreover, for $\epsilon =-1$, if $\lambda_i\neq0$, the cor\-re\-spond\-ing distribution $\mathfrak D_i$ is of even dimension. 
For every $i=\overline{1,m}$, denoting by $\alpha_i\in [0,1]$ the nonnegative value for which $\lambda_i=\epsilon\alpha_i^2$, we get, for $X\in \mathfrak{D}_i\verb=\=\{0\}$, 
$$|fX|^2=\epsilon\lambda_ig(X,X)= \alpha_i^2 |X|^2= \alpha_i^2 |\varphi X|^2$$
\noindent and $|fX|=\alpha_i |X|=\alpha_i |\varphi X|$; hence, $\alpha_i=\cos\zeta_i$, where $\zeta_i$ is the angle between $\varphi X$ and $TM$, the same for any nonzero $X\in \mathfrak D_i$\,. 

It follows that the $\mathfrak D_i$'s, $i=\overline{1,m}$, are the slant components of the distribution, with different corresponding slant angles $\zeta_i$, excepting the one of them that corresponds to $\alpha_i=1$ and is invariant with respect to $\varphi$ if it exists. Thus, $M$ is a $k$-slant submanifold of $\overline{M}$, where $k$ is $m$ or $m-1$. 

\end{remark}

We conclude: 

\begin{proposition}\label{p125}
Any skew CR submanifold of an almost Hermitian or almost product Riemannian man\-i\-fold is a $k$-slant sub\-man\-i\-fold. 
\end{proposition}

\begin{proposition}\label{p212}
Any $k$-slant submanifold of an almost Hermitian or almost product Riemannian man\-i\-fold which is not an anti-invariant or a CR submanifold is a skew CR submanifold. 
\end{proposition}

We shall now provide an example of $k$-slant submanifold and \mbox{$k$-slant} distribution in an almost Hermitian and in an almost product Rie\-mann\-i\-an manifold. 

\begin{example}\label{ex3}
Let $\overline{M}=\mathbb{R}^{4k+2}$ be the Euclidean space for some $k\geq 2$, with the canonical coordinates 
$(x_{1},\ldots, x_{4k+2})$, and let $\{e_{1}=\frac{\partial }{\partial x_{1}},\ldots,e_{4k+2}=\nolinebreak\frac{\partial }{\partial x_{4k+2}}\}$ be the natural basis in the tangent bundle. For $\epsilon \in \{-1,1\}$, let us define a $(1,1)$-tensor field $\varphi$ by: 
$$\varphi e_{1} = e_{2}, \quad \varphi e_{2}=\epsilon e_{1},$$
$$\varphi e_{4j-1} =\frac{j-1}{\sqrt {2(j^2+1)}}{\ }e_{4j}+\frac{j+1}{\sqrt {2(j^2+1)}}{\ }e_{4j+2},$$
$$\varphi e_{4j}=\epsilon \frac{j-1}{\sqrt {2(j^2+1)}}{\ }e_{4j-1}-\frac{j+1}{\sqrt {2(j^2+1)}}{\ }e_{4j+1}, $$
$$\varphi e_{4j+1} =-\epsilon \frac{j+1}{\sqrt {2(j^2+1)}}{\ }e_{4j}+\epsilon \frac{j-1}{\sqrt {2(j^2+1)}}{\ }e_{4j+2},$$
$$\varphi e_{4j+2}=\epsilon \frac{j+1}{\sqrt {2(j^2+1)}}{\ }e_{4j-1}+\frac{j-1}{\sqrt {2(j^2+1)}}{\ }e_{4j+1}$$
for $j\in \{ 1,\ldots,k\}$.
Then, with the Riemannian metric $g$ given by $g(e_{i},e_{j})=\nolinebreak\delta _{ij}$ for $i, j= \overline{1, 4k+2}$, 
$(\overline{M},\varphi,g)$ is an almost Hermitian manifold for $\epsilon =-1$ and an almost product Riemannian manifold for $\epsilon =1$.

We define the submanifold $M$ of $\overline{M}$ by
$$
M:=\{(x_{1},\ldots, x_{4k+2})\in \mathbb{R}^{4k+2} \ | \ x_{4j+1}=x_{4j+2}=0,\ j= \overline{1,k}\},$$
and consider the distributions: 
$$
D_0=\langle e_{1},e_{2}\rangle,\quad D_{j}=\langle e_{4j-1},e_{4j}\rangle, \ j= \overline{1,k}\,.
$$

Then, $M$ is a $k$-slant submanifold of $\overline{M}$ for which the cor\-re\-spond\-ing \mbox{$k$-slant} distribution is $TM=\oplus_{i=0}^kD_i$, with
$D_0$ the invariant component and $D_j$, $j=\overline{1,k}$, slant
distributions having the slant angles 
$$\theta _{j}=\arccos \left(\frac{j-1}{\sqrt {2(j^2+1)}}\right),\ j=\overline{1,k}\,.$$ 

$\oplus_{i=1}^kD_i$ is the proper $k$-slant distribution associated to $M$.

We consider the distributions $L_i:=\langle e_{4i+1}, e_{4i+2}\rangle$ in $(TM)^{\perp}$, $i=\overline{1,k}$\,, and notice that $\oplus_{i=1}^kL_i$ is the dual $k$-slant distribution of $\oplus_{i=1}^kD_i$. 
We have $D_i= f(L_i)$, $i=\overline {1,k}$; hence, the dual $k$-slant distribution of $\oplus_{i=1}^kL_i$ is $\oplus_{i=1}^kD_i$. 
\end{example}

%%%%%%%%%%%%%%%%%%%%%%%%%%%%%%%
\section{$k$-pointwise slant distributions. \newline General considerations}\label{pointwise_gen_consid} 

Let $\epsilon \in \{-1,1\}$, and let $\varphi$ be a $(1,1)$-tensor field on the Riemannian manifold $(\overline{M},g)$ such that 
\begin{equation}\nonumber
g(\varphi X,Y)=\epsilon g(X,\varphi Y) 
\text{ for any }X,Y\in T\overline{M}.
\end{equation}

Throughout this section, we will consider $M$ to be $\overline{M}$ or an immersed submanifold of $\overline{M}$ if not specified otherwise. 

\begin{definition}\hspace{-2pt}(originating from \cite{etayo} and \cite{chen-garay})\label{p73} 
A non-null distribution $D$ on ${M}$ is called a \textit{pointwise slant distribution on ${M}$} (or \textit{in $T{M}$}) if there exists a continuous function $\theta :{M}\rightarrow (0,\frac{\pi}{2}]$ such that, for any $x \in {M}$ and $v \in D_x\verb=\=\{0\}$, we have $\varphi v\neq 0$, and the angle between $\varphi v$ and the vector space $D_x$ is equal to $\theta (x)$ (and, consequently, does not depend on $v$ but only on $x$). 
The function $\theta$ is called the \textit{slant function} of $D$, and we will also call the distribution a $\theta$-\textit{pointwise slant distribution}. 
\end{definition} 

\begin{remark}\label{p150}
The continuity of the slant function is implicit in view of the smoothness of the distribution (see also \cite{latcu}). 
\end{remark}

\begin{remark}\label{p76}
If $D$ is a $\theta$-pointwise slant distribution on $M$, then, for any vector field $X\in D\verb=\=\{0\}$ and $x\in {M}$ for which $X_x\neq 0$, the angle between $\varphi X_x$ and the vector space $D_x$ is $\theta(x)$. 
\end{remark}

The results of this section are valid in any of the settings considered throughout the paper: almost contact metric, almost paracontact metric, almost Hermitian or almost product Riemannian setting.

With a similar argument as in the proof of Proposition \ref{p31}, we obtain 

\begin{proposition}\label{p75}
Let $D_1,\,D_2$ be two orthogonal pointwise slant distributions on ${M}$ having the same slant function $\theta$. Denoting, for any $Z\in T{M}$, by $fZ$ the component of $\varphi Z$ in $D_1\oplus D_2$, assume that: \\ 
\hspace*{7pt} i)\, the orthogonality of vector fields from $D_1\oplus D_2$ is invariant under $\varphi$;\\ 
\hspace*{7pt} ii) $f(D_i)\subseteq D_i$, $i=1,2$.\\ 
Then, the two pointwise slant distributions, $D_1,\,D_2$, can be joined into a single pointwise slant distribution with slant function $\theta $.
\end{proposition}

\begin{corollary}
Let $L_1$, $L_2$,\ldots , $L_m$ be mutually orthogonal distributions on $M$ which are invariant or pointwise slant distributions (at least one) and are invariant with respect to $\tilde f$ (the component of $\varphi$ in $\oplus_{i=1}^m L_i$) such that the orthogonality of vector fields from $\oplus_{i=1}^m L_i$ is invariant under $\varphi$. Then, the direct sum $\oplus_{i=1}^m L_i$ can be represented as an orthogonal sum of pointwise slant distributions with distinct slant func\-tions and at most one invariant distribution. 
\end{corollary}

For any distribution $D$ on ${M}$ and any $Z\in T\overline{M}$, let $fZ$ and $wZ$ be the com\-po\-nents of $\varphi Z_M$ in $D$ and in $D^{\perp}$, respectively. 

\medskip 
Related to the slant function $\theta$ of a general pointwise slant distribution (i.e., the slant function can take any value from $0$ to $\frac{\pi}{2}$) on a Riemannian manifold, for $\varphi$  a symmetric or skew-symmetric structural endomorphism on the manifold which acts isometrically on that distribution, we established in \cite{latcu} the following results, which will prove useful later. 

\begin{theorem}\hspace{-2pt}\cite{latcu}\label{p185}
Let $D$ be a distribution on $M$ which is a general pointwise $\theta$-slant distribution relative to $T\overline{M}$ such that $\varphi|_D$ is an isometry. 

Then, for any $x\in {M}$, $(f^2|_D)_x$ has only one eigenvalue $\lambda(x)$, and $D_x$ is entirely composed of eigenvectors of $\lambda(x)$. The eigenvalue function $\lambda$ is in $C^\infty({M})$, and $\lambda(x) = \epsilon(\cos \theta(x))^2$, respectively $\cos\theta(x)=\sqrt{\epsilon\lambda(x)}$, for any $x\in {M}$. 
In particular: 
\begin{enumerate}
\item[(a)] $f^2X=\lambda X$ for any vector field $X\in D$; 
\item[(b)] the slant function $\theta$ is continuous, and $\cos^2\theta \in C^\infty({M})$; 
\item[(c)] $\cos\theta \in C^\infty({M})$ if\, $\theta(x)\neq \frac{\pi}{2}$ for any $x\in{M}$; 
\item[(d)] $\theta \in C^\infty(M)$ if\, $\theta(x)\in (0,\frac{\pi}{2})$ for any $x\in M$. 
\end{enumerate}
\end{theorem}

\begin{corollary}\hspace{-2pt}\cite{latcu}\label{p122}
Let $D$ be a distribution on $M$ which is a slant distribution relative to $T\overline{M}$ with slant angle $\theta$ such that $\varphi|_D$ is an isometry. Then, for any $x\in M$, $(f^2|_D)_x$ has only one eigenvalue $\lambda=\epsilon\cos^2 \theta$, which is independent of $x$. In particular, $\cos\theta=\sqrt{\epsilon\lambda}$, and $f^2 X=\lambda X$ for any $X\in D$. 
\end{corollary}

We will now introduce the ${k}$-pointwise slant distribution for $k\in \mathbb{N}^*$. 
 
\begin{definition}\label{p77} 
Let $k\in \mathbb{N}^*$ and $D$ be a non-null distribution on ${M}$. We will call $D$ a \textit{$k$-pointwise slant distribution} if there exists an orthogonal decomposition of $D$ into regular distributions, 
$$D=\oplus_{i=0}^k{D_i}$$
with the $D_i$'s non-null distributions for $i=\overline{1,k}$ and $D_0$ possible null, and there exist $k$ distinct continuous functions $\theta_i:{M}\rightarrow (0,\frac{\pi}{2}]$, $i=\overline{1,k}$, such that: \\ 
\hspace*{7pt} (i) \ $D_i$ is a $\theta_i$-pointwise slant distribution for $i=\overline{1,k}$;\\ 
\hspace*{7pt} (ii) \,$\varphi X\in D_0$ for any $X\in D_0$ 
(i.e., $\widehat{(\varphi X, D)}=0=:\theta_0$ for $X\in D_0$ with $\varphi X\neq 0$, and $f(D_0)\subseteq D_0$);\\ 
\hspace*{7pt} (iii) $f(D_i)\subseteq D_i$ for $i=\overline{1,k}$.

\medskip
We will say that $D$ is a \textit{multi-pointwise slant distribution} if $k\geq 2$.

We will call $D$ a $(\theta_1,\theta_2,\ldots,\theta_k)$-\textit{pointwise slant distribution} if we want to specify the slant functions. 

$D_0$ represents the \textit{invariant component} and $\oplus_{i=1}^kD_i$ the \textit{proper \mbox{$k$-point}\-wise slant component} of $D$. 

We will call the distribution $D=\oplus_{i=0}^kD_i$ a \textit{proper $k$-pointwise slant distribution} if $D_0=\{0\}$. 
\end{definition}

\begin{remark}\label{p78}
In view of (iii), we notice that (i) is equivalent to \\ 
\hspace*{7pt} (i') $\varphi v \neq 0$, and $\widehat{(\varphi v, D_x)}=\theta_i(x)$ for any $x\in {M}$ and $v\in (D_i)_x\verb=\=\{0\}$, $i=\overline{1,k}$.
\end{remark}

\begin{remark}\label{p85}
The continuity of the slant functions is implicit in view of the smoothness of the distributions. 
\end{remark}

With a similar argument as for Proposition \ref{p167}, we get 

\begin{proposition}\label{p176}
Let $k\in \mathbb{N}^*$ and $D$ be a non-null distribution on $M$ decomposable into an orthogonal sum of regular distributions, $D=\oplus_{i=0}^k{D_i}$ with $D_i\neq \{0\}$ for $i=\overline{1,k}$ and $D_0$ invariant (possible null). Let $pr_i$ denote the projection operator onto $D_i$ for $i=\overline{1,k}$. If $\varphi$ restricted to $\oplus_{i=1}^k{D_i}$ is an isometry, and $f(D_i)\subseteq D_i$ for $i=\overline{1,k}$, and there exist $k$ distinct continuous functions $\theta_i:{M}\rightarrow (0,\frac{\pi}{2}]$, $i=\overline{1,k}$, such that  
\begin{equation}\nonumber
f^2X=\epsilon \sum_{i=1}^k\cos^2\theta_i\cdot pr_iX\, \text{ for any } X\in \oplus_{i=1}^k{D_i},
\end{equation}
then $D$ is a $k$-pointwise slant distribution with slant functions $\theta_i$ corresponding to $D_i$, $i=\overline{1,k}$. 
\end{proposition}

\medskip 
Let $k\in \mathbb{N}^*$ and $D=\oplus_{i=0}^kD_i$ be a $k$-pointwise slant distribution on ${M}$ with $D_0$ the invariant component. 
With similar arguments as for Remarks \ref{p25}, \ref{p198}, we immediately obtain 

\begin{remark}\label{p79}
Condition (iii) from Definition \ref{p77} of a $k$-pointwise slant distribution can be replaced by \\ 
\hspace*{7pt} (iii') $\varphi(D_i)\perp D_j$ for any $i\neq j$ from $\{1,\ldots,k\}$.
\end{remark}

\begin{remark}\label{p199}
If the orthogonality of vector fields from the proper \mbox{$k$-point}\-wise slant distribution $\oplus_{i=1}^kD_i$ is invariant under $\varphi$,
then: 
\begin{equation}\label{86}
\varphi(D_1),\ldots,\varphi(D_k) \text{ are orthogonal};\nonumber 
\end{equation}
\begin{equation}\label{87}
w(D_i)\perp w(D_j) \text{ for } i\neq j;\nonumber 
\end{equation}
\begin{equation}\label{89}
w(\oplus_{i=1}^kD_i)=\oplus_{i=1}^kw(D_i). \nonumber 
\end{equation}

\end{remark}

Let $M$ be an immersed submanifold of $\overline{M}$. 

\begin{definition}\label{p81}
We will call $M$ a \textit{$k$-pointwise slant submanifold} of $\overline{M}$ if $TM$ is a $k$-pointwise slant distribution.  

We will call $M$ a \textit{multi-pointwise slant submanifold} if $k\geq 2$, or a $(\theta_1,\theta_2,\ldots,\theta_k)$-\textit{pointwise slant sub\-man\-i\-fold} if we want to specify the slant functions $\theta_i$. 

If $TM=\oplus_{i=0}^kD_i$, where $D_0$ denotes the invariant component, we will call $\oplus_{i=1}^kD_i$ the \textit{proper $k$-pointwise slant distribution associated} to $M$. 

We will call $M$ a \textit{proper $k$-pointwise slant submanifold} if\, $TM$ is a proper \mbox{$k$-point}\-wise slant distribution. 
\end{definition}

An equivalent formulation of the above definition is 

\begin{definition}\label{p83}
We say that $M$ is a \textit{$k$-pointwise slant submanifold} of $\overline{M}$ if there exists an orthogonal decomposition of $TM$ into regular distributions, 
$$TM=\oplus_{i=0}^k{D_i}=:D$$
with $D_i\neq \{0\}$ for $i= \overline{1,k}$ and $D_0$ possible null, and there exist $k$ distinct continuous functions $\theta_i: M\rightarrow (0,\frac{\pi}{2}]$, $i=\overline{1,k}$, such that: \\ 
\hspace*{7pt} (i) \;For any $i\in \{1,\ldots,k\}$, $x\in {M}$, and $v\in (D_i)_x\verb=\=\{0\}$, we have $\varphi v \neq 0$ and $\widehat{(\varphi v, (D_i)_x)}=\theta_i(x)$;\\ 
\hspace*{7pt} (ii) \,$\varphi v\in (D_0)_x$ for any $x\in M$ and $v\in (D_0)_x$;\\ 
\hspace*{7pt} (iii) $fv\in (D_i)_x$ for any $x\in M$ and $v\in (D_i)_x$, $i=\overline{1,k}$. 
\end{definition}

\begin{remark}
In view of (iii), we notice that (i) can be replaced by \\ 
\hspace*{7pt} (i') For any $i\in \{1,\ldots,k\}$, $x\in {M}$, and $v\in (D_i)_x\verb=\=\{0\}$, we have $\varphi v \neq 0$ and 
$ \widehat{(\varphi v, T_xM)}=\theta_i(x)$.
\end{remark}

\begin{remark}
In view of Remark \ref{p79}, condition (iii) from above can be replaced by \\ 
\hspace*{7pt} (iii') $\varphi(D_i)\perp D_j$ for any $i\neq j$ from $\{1,\ldots,k\}$.
\end{remark} 

\begin{remark}\label{p186}
In particular, if $M$ is a $k$-pointwise slant submanifold of $\overline{M}$ and $D$ is a distribution on $M$ such that $TM=D\oplus D'_0$, where $D'_0$ is an invariant (possible null) regular distribution, then $D$ is a $k$-pointwise slant distribution. 
\end{remark}

Rewriting Proposition \ref{p176} for submanifolds, we obtain 

\begin{proposition}\label{p177}
Let $k\in \mathbb{N}^*$ and $M$ be an immersed submanifold of $\overline{M}$ such that $TM$ is decomposable into an orthogonal sum of regular distributions, $TM=\oplus_{i=0}^k{D_i}$ with $D_0$ invariant (possible null) and $D_i\neq \{0\}$ for $i=\overline{1,k}$. Let $pr_i$ denote the projection operator from $TM$ onto $D_i$ for $i=\overline{1,k}$. If $\varphi$ restricted to $\oplus_{i=1}^k{D_i}$ is an isometry, and $f(D_i)\subseteq D_i$ for $i=\overline{1,k}$, and there exist $k$ distinct continuous functions $\theta_i: M\rightarrow (0,\frac{\pi}{2}]$, $i=\overline{1,k}$, such that  
\begin{equation}\nonumber
f^2X=\epsilon \sum_{i=1}^k\cos^2\theta_i\cdot pr_iX \, \text{ for any } X\in \oplus_{i=1}^k{D_i}\,,
\end{equation}
then $M$ is a $k$-pointwise slant submanifold of $\overline{M}$ with slant functions $\theta_i$ corresponding to $D_i$, $i=\overline{1,k}$. 

\end{proposition}

Let us consider again $M$ to be $\overline{M}$ or an immersed submanifold of $\overline{M}$. 

Revisiting Theorem \ref{p185}, for $k$-pointwise slant distributions, we get 

\begin{theorem}\label{p143}
Let $D=\oplus_{i=0}^k{D_i}$ be a distribution on $M$ such that $D$ is a \mbox{$k$-point}\-wise slant distribution relative to $T\overline{M}$ with $D_0$ the invariant (possible null) component and slant functions $\theta_i$ corresponding to $D_i$, $i=\overline{1,k}$, and with the property that $\varphi$ restricted to ${\oplus_{i=1}^k{D_i}}$ is an isometry. 

Then, for $i=\overline{1,k}$ and $x\in {M}$, $(f^2|_{D_i})_x$ has only one eigenvalue, $\lambda_i(x)$, and $(D_i)_x$ is entirely composed of eigenvectors of $\lambda_i(x)$; the eigenvalue function $\lambda_i$ is in $C^\infty({M})$, and $\lambda_i(x) = \epsilon(\cos \theta_i(x))^2$, so $\cos\theta_i(x)=\nolinebreak\sqrt{\epsilon\lambda_i(x)}$. 
In particular, for any $i\in \{1,\ldots ,k\}:$ 
\begin{enumerate}
\item[(a)] $f^2X_i=\lambda_i X_i$ for any vector field $X_i\in D_i$; 
\item[(b)] the slant function $\theta_i$ is continuous, and\, $\cos^2\theta_i \in C^\infty({M})$; 
\item[(c)] $\theta_i \in C^\infty(M)$ if \,$\theta_i(x)\in (0,\frac{\pi}{2})$ for any $x\in M$. 
\end{enumerate}
\end{theorem}

\begin{corollary}\label{p187}
Let $D=\oplus_{i=0}^k{D_i}$ be a distribution on $M$ which is a \mbox{$k$-slant} distribution relative to $T\overline{M}$ with $D_0$ the invariant (possible null) component and with slant angles $\theta_i$ corresponding to $D_i$, $i=\overline{1,k}$, such that $\varphi|_{\oplus_{i=1}^k{D_i}}$ is an isometry. Then, for $i=\overline{1,k}$ and $x\in M$, $(f^2|_{D_i})_x$ has only one eigenvalue, $\lambda_i=\nolinebreak\epsilon\cos^2 \theta_i$, this being independent of $x$. In particular, $\cos\theta_i=\nolinebreak\sqrt{\epsilon\lambda_i}$\,, and $f^2 X_i=\lambda_i X_i$ for any $X_i\in D_i$,  $i=\overline{1,k}$. 
\end{corollary}

%%%%%%%%%%%%%%%%%%%%%%%%%%%%%%%%
\section{$k$-pointwise slant distributions in almost contact metric and almost paracontact metric settings}\label{pointwise_alm_cont}

For a fixed $\epsilon\in \{-1,1\}$, let $(\overline{M},\varphi,\xi,\eta,g)$ be an almost ($\epsilon$)-contact metric manifold. 
In view of (\ref{3}), we notice that $\varphi$ restricted to $\langle\xi\rangle^{\perp}$ is an isometry and hence preserves on $\langle\xi\rangle^{\perp}$ the orthogonality of vector fields. 

Throughout this section, we consider that any sub\-man\-i\-fold $M$ of $\overline{M}$ we deal with satisfies $\xi\in TM$. 

In the sequel, until the end of the section, we will consider $M$ to be $\overline{M}$ or an immersed submanifold of $\overline{M}$ if not specified otherwise. 

Let $k\in \mathbb{N}^*$ and $D=\oplus_{i=0}^k{D_i}$ be a $k$-pointwise slant distribution on $M$ with $D_0$ the invariant component such that $\xi \perp D$, and let $G=(D\oplus \langle \xi\rangle)^\perp$ be the orthogonal complement of $D\oplus \langle \xi\rangle$ in $T\overline{M}$. Let $\theta_0=0$ and let $\theta_1,\theta_2,\dots ,\theta_k$ denote the slant functions of $D$. For any $Z\in T\overline{M}$, the components $fZ$ and $wZ$ of $\varphi Z_M$ in $D$ and in $D^\perp$ coincide with the components of $\varphi Z_M$ in $D\oplus \langle \xi\rangle$ and in $G$, respectively.
We will denote by $pr_i$ the projection operator onto $D_i$ for $i=\overline{0,k}$. 

\begin{remark}\label{p115}
With the same arguments as in section \ref{alm_cont}, we obtain: 
\begin{align*} 
(i)\quad\,\ \varphi(D_0) &=D_0,\; w(D_0)=\{0\}, \; f(D_0)=D_0;\\ 
(ii)\,\ \ \ker \eta_M &=D \oplus G, \ \ \varphi(D \oplus G)=D \oplus G;\\ 
(iii)\,\ \varphi^2(D_i) &=D_i \,\text{ for } i=\overline{1,k}, \; 
\varphi^2(G)=G;\\ 
(iv)\ \ f(\varphi X) &=\epsilon X, \; w(\varphi X)=0 \,\text{ for any } X\in D;\\ 
(v)\quad f(\varphi U) &=0, \; w(\varphi U)=\epsilon U \,\text{ for any } U\in G.
\end{align*}
\end{remark}

For any $i\in \{1,\ldots,k\}$ and $X_i\in D_i\verb=\=\{0\}$, from Definition \ref{p77} (i), we have $\varphi X_i\neq 0$ and 
\begin{equation}\label{90}
\|(fX_i)_x\|=\cos \theta_i(x)\cdot \|(\varphi X_i)_x\|\, \text{ for any } x\in {M}; \nonumber 
\end{equation}
hence, we get: 

\begin{proposition}\label{p86}
\begin{align}
&(i)\quad |fX_i|=\cos \theta_i\cdot |\varphi X_i| \,\text{ for any } X_i\in D_i\setminus\{0\}, i=\overline{1,k};\nonumber\\ 
&(ii)\quad
f^2X=\epsilon\sum_{i=0}^k\cos^2\theta_i\cdot pr_iX \, \text{ for any } X\in D.\label{92}
\nonumber 
\end{align}
\end{proposition}

\begin{corollary}\label{p87}
\begin{equation}\label{93}
f((D_i)_x)=(D_i)_x \text{ if } i \text{ and }x \text{ satisfy }\, \theta_i(x)\neq \frac{\pi}{2}\,. \nonumber
\end{equation}
\end{corollary}

Taking into account Remark \ref{p33} and Propositions \ref{p176}, \ref{p86}, we obtain 

\begin{theorem}\label{p178}
Let $\mathfrak D$ be a non-null distribution on $M$ such that $\mathfrak D \perp \xi$ and $\mathfrak D$ is decomposable into an orthogonal sum of regular distributions, $\mathfrak D=\oplus_{i=0}^k{\mathfrak D_i}$ with $\mathfrak D_i\neq \{0\}$ for $i=\overline{1,k}$ and $\mathfrak D_0$ invariant (possible null). Let $pr_i$ denote the projection operator onto $\mathfrak D_i$ for $i=\overline{0,k}$, $f$ the component of $\varphi$ into $\mathfrak D$ (i.e., $f=pr_{\mathfrak D}\circ \varphi$), and $\theta_0=0$. If $f(\mathfrak D_i)\subseteq \mathfrak D_i$ for $i=\overline{1,k}$, then the following assertions are equivalent: \\ 
\hspace*{7pt} (a) There exist $k$ distinct continuous functions  $\theta_i:{M}\rightarrow (0,\frac{\pi}{2}]$, $i=\nolinebreak\overline{1,k}$, such that  
\begin{equation}\nonumber
f^2X= \epsilon \sum_{i=0}^k\cos^2\theta_i\cdot pr_iX \ \text{for any} \  X\in \mathfrak D;
\end{equation}
\hspace*{7pt} (b) $\mathfrak D$ is a $k$-pointwise slant distribution with slant functions $\theta_i$ corresponding to $\mathfrak D_i$, $i=\overline{1,k}$. 
\end{theorem}

\begin{remark}\label{p180}
Theorem \ref{p178} provides a necessary and sufficient condition for a submanifold $M$ of $\overline M$ to be a $k$-pointwise slant submanifold, considering $\mathfrak D=\oplus_{i=0}^k{\mathfrak D_i}$ if $TM=\oplus_{i=0}^k{\mathfrak D_i}\oplus \langle \xi\rangle$.
\end{remark}

\begin{remark}\label{141}
A great part of the results obtained for $k$-slant distributions in the almost ($\epsilon$)-contact metric case are also valid for $k$-pointwise slant distributions, with similar justifications, for another part of them being necessary some minor mod\-i\-fi\-ca\-tions. More precisely, Lemma \ref{p92}, Propositions \ref{p2}, \ref{p3}, \ref{p67}, \ref{p7}, \ref{p14}, \ref{p16}, Theorems \ref{p1}, \ref{p17}, and Corollaries \ref{p62}, \ref{p15}, \ref{p18}--\ref{p20} remain further valid for $k$-pointwise slant distributions as they were stated. 
\end{remark}

The other statements become, after adequate modifications, as follows. 

\begin{corollary}\label{p100}
\begin{equation}\label{94}
fwX=\epsilon \sum_{i=1}^k \sin^2 \theta_i \cdot pr_iX\, \text{ for any } X\in D. 
\nonumber
\end{equation}
\end{corollary}

\begin{remark}\label{p101}
If $j\in\{1,\ldots,k\}$ and $x\in {M}$ such that $\theta_j(x)=\frac{\pi}{2}$, then, for any $X_j\in D_j$, we have
$fw(X_j)_x=\epsilon (X_j)_x$
and $wfw(X_j)_x=\epsilon w(X_j)_x$, which implies
$wf(U_j)_x=\epsilon (U_j)_x$ for any $U_j\in w(D_j)$. 
We conclude that $f_x|_{w((D_j)_x)}:\nolinebreak w((D_j)_x) \rightarrow\nolinebreak (D_j)_x$ and $w_x|_{(D_j)_x}: (D_j)_x \rightarrow w((D_j)_x)$ are anti-inverse to each other for $\epsilon=-1$ but inverse for $\epsilon=1$. 
\end{remark}

\begin{proposition}\label{p102}
For any $X\in D\oplus \langle \xi\rangle$ and\, $U\in G$, we get:
\begin{align}
f^2X+fwX &=\epsilon (X- \eta(X) \xi), \nonumber\\ 
wfX+w^2X &=0, \nonumber \\ 
f^2U+fwU &=0, \nonumber
\\ 
wfU+w^2U &=\epsilon U. \nonumber
\end{align}
\end{proposition}

\begin{corollary}\label{p103}
For any $U_0,V_0\in H$, we have: 
\begin{align}
w^2U_0&=\epsilon U_0\,,\nonumber\\  
g(wU_0,wV_0)&=g(U_0,V_0),\nonumber\\ 
|wU_0|&=|U_0|.\nonumber
\end{align}
\end{corollary}

\begin{proposition}\label{p104}
\[
w^2((D_i)_x)=
\begin{cases}
w((D_i)_x) & \text{for } \ \theta_i(x)\neq \frac{\pi}{2}\,,\\ 
\ \{0\} & \text{for } \ \theta_i(x)=\frac{\pi}{2}\,.
\end{cases}
\]
\end{proposition}

\begin{proposition}\label{p105}
For any $U\in w(D)$, $U=\sum_{i=1}^kU_i$ with $U_i\in w(D_i)$, we have:
\begin{equation}\label{95}
wfU= \epsilon\sum_{i=1}^k\sin^2\theta_i\cdot U_i\,,\nonumber 
\end{equation}
\begin{equation}\label{96}
w^2U= \epsilon\sum_{i=1}^k\cos^2\theta_i\cdot U_i\,.\nonumber 
\end{equation}
\end{proposition}

\begin{proposition}\label{p106}
For any $i\in\{1,\ldots ,k\}$ and $x\in {M}$ with $\theta_i(x)\neq \frac{\pi}{2}$, and any $X_i, Y_i\in D_i\verb=\=\{0\}$, $U_i, V_i\in w(D_i)\verb=\=\{0\}$, $X_0, Y_0\in D_0\verb=\=\{0\}$, $U_0, V_0\in\nolinebreak H\verb=\=\{0\}$,
$\overline{X}, \overline{Y}\in (D\oplus G)\verb=\=\{0\}$ such that ${M}_{\widehat{X_i,Y_i}}$\,, ${M}_{\widehat{X_0,Y_0}}$\,, ${M}_{\widehat{U_i,V_i}}$\,, ${M}_{\widehat{U_0,V_0}}$\,, ${M}_{\widehat{\overline{X},\overline{Y}}}$ are nonempty, we have: \\ 
\hspace*{7pt} (i) \ \;\,$\cos(\widehat{fX_0,fY_0})=\cos(\widehat{\varphi X_0,\varphi Y_0})=\cos(\widehat{X_0,Y_0})$;\\ 
\hspace*{7pt} (ii) \;\,$\cos(\widehat{f(X_i)_x,f(Y_i)_x})=\cos(\widehat{\varphi (X_i)_x,\varphi (Y_i)_x})=\cos(\widehat{(X_i)_x,(Y_i)_x})$;\\ 
\hspace*{7pt} (iii) \,$g(wU_i,wV_i)=\cos^2\theta_i \cdot g(U_i,V_i)$;\\ 
\hspace*{7pt} (iv) \;$\cos(\widehat{wU_0,wV_0})=\cos(\widehat{U_0,V_0})=\cos(\widehat{\varphi U_0,\varphi V_0})$;\\ 
\hspace*{7pt} (v) \ \;$\cos(\widehat{w(U_i)_x,w(V_i)_x})=\cos(\widehat{(U_i)_x,(V_i)_x})=\cos(\widehat{\varphi (U_i)_x,\varphi (V_i)_x})$;\\ 
\hspace*{7pt} (vi) \ $\cos(\widehat{\varphi \overline{X},\varphi \overline{Y}})=\cos(\widehat{\overline{X},\overline{Y}})$.
\end{proposition}

\begin{theorem}\label{p107}
The distribution $G=\oplus_{i=1}^kw(D_i)\oplus H$ is a $k$-pointwise slant distribution with $H$ the invariant component and $\oplus_{i=1}^kw(D_i)$ the proper \linebreak $k$-pointwise slant com\-po\-nent, the pointwise slant distribution $w(D_i)$ having the same slant function $\theta_i$ as $D_i$ for $i= \overline{1,k}$.
\end{theorem}

\begin{definition}\label{p108}
We will call $\oplus_{i=1}^kw(D_i)$ \textit{the dual $k$-pointwise slant distribution} of $\oplus_{i=1}^kD_i$.
\end{definition}

\begin{remark}\label{p109}
In the same way we defined the dual of the proper $k$-pointwise slant component $\oplus_{i=1}^kD_i$ of the distribution $D$ by means of $w$, we can construct the dual of the proper $k$-pointwise slant component $\oplus_{i=1}^kw(D_i)$ of the distribution $G$ by means of $f$. This will be $f(\oplus_{i=1}^kw(D_i))=\oplus_{i=1}^kfw(D_i)$. 
\end{remark}

\begin{corollary}\label{p110}
The dual of the proper $k$-pointwise slant distribution \linebreak $\oplus_{i=1}^kw(D_i)$, which is $\oplus_{i=1}^kf(w(D_i))$, is precisely the $k$-pointwise slant distribution $\oplus_{i=1}^kD_i$.
\end{corollary}

Denoting $w(D_i)$ by $G_i$ for $i=\overline{1,k}$, we obtain: 

\begin{proposition}\label{p111}
\begin{align*}
w(f(G_i)) &=G_i\,\text{ for } i=\overline{1,k}\,;\\ 
f^2((G_i)_x) &=\left\{\begin{array}{ll}
(D_i)_x & \text{if }\,\theta_i(x)\neq\frac{\pi}{2}\,, \\ 
\ \{0\} & \text{if }\,\theta_i(x)= \frac{\pi}{2}\,. 
\end{array}
\right.
\end{align*}

\end{proposition}

\begin{theorem}\label{p112}
 $f_x$ and $w_x$ restricted to $(D_i)_x$ or $w((D_i)_x)$ for $i=\overline{1,k}$ and $x\in {M}$ \textup(excepting $f_x|_{(D_j)_x}$ and $w_x|_{w((D_j)_x)}$ with $\theta_j(x)=\frac{\pi}{2}$, in which case $f_x|_{(D_j)_x}$ and $w_x|_{w((D_j)_x)}$ are vanishing\textup), 
 $f|_{D_0}$, $w|_H$, and $\varphi|_{D\oplus G}$ are conformal maps. 
\end{theorem}

Let $M$ be an immersed submanifold of $\overline{M}$ such that $\xi\in TM$ and let $k\in \mathbb{N}^*$. Since $\langle\xi\rangle$ can be considered as a part of an invariant component of $TM$, the straightforward definition of a $k$-pointwise slant submanifold in this setting will be as follows.  

\begin{definition}\label{p84}
We say that $M$ is a \textit{$k$-pointwise slant submanifold} of the almost ($\epsilon$)-contact metric manifold $(\overline{M},\varphi,\xi,\eta,g)$ if there exists an orthogonal decomposition of $TM$ into regular distributions, 
$$TM=\oplus_{i=0}^k{D_i}\oplus \langle \xi\rangle=:D\oplus \langle \xi\rangle$$
with $D_i\neq \{0\}$ for $i= \overline{1,k}$ and $D_0$ possible null, and there exist $k$ distinct continuous functions $\theta_i :M \rightarrow (0,\frac{\pi}{2}]$, $i=\overline{1,k}$, such that: \\ 
\hspace*{7pt} (i)\; $\widehat{(\varphi X_x, (D_i)_x)}=\theta_i(x)$ for any $X\in D_i\verb=\=\{0\}$, $i=\overline{1,k}$, and $x\in {M}$ with $X_x\neq 0$;\\ 
\hspace*{7pt} (ii)\, $\varphi X\in D_0$ for any $X\in D_0$ (i.e., $\widehat{(\varphi X, TM)}=0$ for $X\in D_0\verb=\=\{0\}$, and $f(D_0)\subseteq D_0$);\\ 
\hspace*{7pt} (iii) $f(D_i)\subseteq D_i$ for $i=\overline{1,k}$. 

\smallskip 
In this case, 
$\oplus_{i=1}^kD_i$ will be the \textit{proper $k$-pointwise slant distribution} associated to $M$. 

\end{definition}

\begin{remark}\label{p195}
In view of (iii) and Remark \ref{p79}, we have: \\ 
(a) Condition (i) can be replaced by \\ 
\hspace*{7pt} (i') $\widehat{(\varphi X_x, T_xM)}=\theta_i(x)$ for any $X\in D_i\verb=\=\{0\}$, $i=\overline{1,k}$, and $x\in {M}$ with $X_x\neq 0$; \\ 
(b) Condition (iii) can be replaced by \\ 
\hspace*{7pt} (iii') $\varphi(D_i)\perp D_j$ for any $i\neq j$ from $\{1,\ldots,k\}$. 
\end{remark}

The above definition can be reformulated as follows. 

\begin{definition}\label{p207}
We say that $M$ is a \textit{$k$-pointwise slant submanifold} of $\overline{M}$ if, in the orthogonal decomposition $TM=D\oplus \langle \xi\rangle$, $D$ is a $k$-pointwise slant distribution. 
\end{definition}

\begin{remark}\label{p113}
All the results of this section can be transferred to any \mbox{$k$-point}\-wise slant submanifold $M$ of $(\overline{M},\varphi,\xi,\eta,g)$, considering, for $TM=\nolinebreak\oplus_{i=0}^k{D_i}\oplus \nolinebreak\langle \xi\rangle$, the $k$-pointwise slant distribution $D=\oplus_{i=0}^k{D_i}$. 
Thus, all the results remain valid in the $k$-pointwise slant submanifold framework. 
\end{remark}

\begin{remark}\label{p119}
We describe below, in sense of Ronsse \cite{ronsse}, the notion of \textit{generic submanifold of an almost} ($\epsilon$)-\textit{contact metric manifold}, relating it to the concept of $k$-pointwise slant submanifold.  

Let $M$ be an immersed submanifold of an almost ($\epsilon$)-contact metric manifold $(\overline{M},\varphi,\xi,\eta,g)$, and, for $Z\in T{M}$, let $fZ$ be the tangential component of $\varphi Z$. 

In view of Lemmas \ref{lema1}, \ref{lema2} and Remark \ref{p194}, $f$ is skew-symmetric for $\epsilon=-1$ and symmetric for $\epsilon=1$, so $f^2$ is symmetric. Denoting by $\lambda_i(x)$, $i=\overline{1,m(x)}$, the distinct eigenvalues of $f_x^2$ acting on the tangent space $T_xM$ for $x\in M$, these eigen\-val\-ues are all real. 
In view of (\ref{7}) and $\varphi \xi=0$, we have $|fX|\leq|\varphi X|\leq |X|$ for any $X\in T{M}$, so $|\lambda_i(x)|\leq 1$ for all $x\in M, i=\overline{1,m(x)}$. Denoting, for any $x\in M$, by $\mathfrak D_x^i$ the eigen\-space cor\-re\-spond\-ing to $\lambda_i(x)$, $i=\overline{1,m(x)}$, each tangent space $T_xM$ of $M$ at $x\in M$ admits the following orthogonal decomposition into the eigenspaces of $f_x^2$: 
$$T_xM={\mathfrak D_x^1}\oplus\ldots\oplus{\mathfrak D_x^{m(x)}}.$$
Moreover, each eigenspace 
$\mathfrak D_x^i$ is invariant under $f_x$. 

Now, let us consider that $M$ is a generic submanifold of $\overline{M}$ in sense of Ronsse, that is: 
\begin{enumerate}
\item $m(x)$ does not depend on $x\in M$ (we will denote $m(x)=m$); 
\item the dimension of ${\mathfrak D_x^i}$, $i=\overline{1,m}$, is independent of $x\in M$; 
\item if one of the $\lambda_i(x)$ 
is $0$ or $1$, say $\lambda_{i_0}(x_0)$, then $\lambda_{i_0}(x)$ has the same value for all $x\in M$.
\end{enumerate}

In this situation, for any tangent space $T_xM$ ($x\in M$), there is the same number $m$ of distinct eigenvalues of $f_x^2$, denoted by $\lambda_1(x)$, \dots , $\lambda_m(x)$. In view of the smoothness of the vector fields, the functions $\lambda_1$, \dots , $\lambda_m$ are continuous on $M$. Denoting by $\mathfrak D_i$ the distribution corresponding to the family $\{\mathfrak D_x^i: x\in M\}$ for $i=\overline{1,m}$, we notice that $TM$ admits the orthogonal decomposition  
$$TM={\mathfrak D_1}\oplus\ldots\oplus{\mathfrak D_m}\,.$$
We notice that each distribution $\mathfrak D_i$ is invariant under $f$.

Since $f\xi=\varphi\xi=0$, one of the $\mathfrak D_i$'s contains $\langle \xi \rangle$ and corresponds to the zero eigenvalue of $f^2$; let $\mathfrak D_m$ be that distribution. Let us decompose $\mathfrak D_m$ into $\langle\xi\rangle$ and its orthogonal in $\mathfrak D_m$, denoted by $\mathfrak D'_m$, $\mathfrak D_m=\langle\xi\rangle\oplus\mathfrak D'_m$. We notice that $\mathfrak D'_m$, if non-null, is a slant distribution with slant angle $\frac{\pi}{2}$.

For any $x\in M$ and $i\in \{1,\ldots,m-1\}$, let $\alpha_i(x)\in (0,1]$ denote the positive value for which $\lambda_i(x)=\epsilon\alpha_i^2(x)$. Then, for any $X\in  \mathfrak{D}_i\verb=\=\{0\}$ and $x\in M$ such that $X_x\neq 0$, we get 
$$\|f_xX_x\|^2=\epsilon\lambda_i(x)g_x(X_x,X_x)= \alpha_i^2(x) \|X_x\|^2= \alpha_i^2(x) \|\varphi_x X_x\|^2$$
and $\|f_xX_x\|=\alpha_i(x) \|X_x\|=\alpha_i(x) \|\varphi_x X_x\|$; hence, $\alpha_i(x)=\cos\zeta_i(x)$, where $\zeta_i(x)$ is the angle between $\varphi X_x$ and $T_xM$. 

The distributions $\mathfrak{D}_i$, $i=\overline{1,m-1}$, are pointwise slant distributions with distinct slant functions $\zeta_i$ except at most one of them, which is invariant with respect to $\varphi$ and corresponds to $\alpha_i=1$ if such one exists. 

It follows that, since $\oplus_{i=1}^{m-1}\mathfrak{D}_i\oplus\mathfrak{D}'_m$ does not reduce to an invariant distribution with respect to $\varphi$, $M$ is a $k$-pointwise slant submanifold of $\overline{M}$, where $k$ is one of the values: $m,\, m-1,\, m-2$. 
\end{remark}

We deduce: 

\begin{proposition}\label{p126}
Any generic submanifold of an almost contact metric or almost paracontact metric man\-i\-fold is a $k$-pointwise slant sub\-man\-i\-fold. 
\end{proposition}

We will show through an example that a $k$-pointwise slant submanifold is not necessarily a generic one. 

\begin{example}\label{ex4}
Let $\overline M=\mathbb{R}^{4k+3}$ be the Euclidean space for some $k\geq 2$, with the canonical coordinates 
$(x_{1},\ldots, x_{4k+3})$, and let $\{e_{1}=\frac{\partial }{\partial x_{1}},\ldots,e_{4k+3}=\frac{\partial }{\partial x_{4k+3}}\}$ be the natural basis in the tangent bundle. Let $\epsilon\in \{-1,1\}$, $\gamma\geq 0$, $\delta>0$, and 
$E_{\gamma,\delta}(j,x)=\sqrt{\|x\|^4+2\gamma \|x\|^2+j^2\delta^2+\gamma^2}$ for any $j\in \mathbb{N^*}$ and $x\in \overline M$. 

Let us define a vector field $\xi $, a $1$-form $\eta $, and a $(1,1)$-tensor field $\varphi$ by: 
$$\xi =e_{4k+3}, \quad \eta =dx_{4k+3},$$
$$\varphi e_{1} = e_{2}, \quad \varphi e_{2}=\epsilon e_{1},$$
\begin{align*}
(\varphi e_{4j-1})_x &=\frac{\|x\|^2+\gamma}{E_{\gamma,\delta}(j,x)}{\ }(e_{4j})_x+\epsilon \frac{j\delta}{E_{\gamma,\delta}(j,x)}{\ }(e_{4j+2})_x,\\ 
(\varphi e_{4j})_x &=\epsilon \frac{\|x\|^2+\gamma}{E_{\gamma,\delta}(j,x)}{\ }(e_{4j-1})_x+\epsilon \frac{j\delta}{E_{\gamma,\delta}(j,x)}{\ }(e_{4j+1})_x,\\ 
(\varphi e_{4j+1})_x &=\frac{j\delta}{E_{\gamma,\delta}(j,x)}{\ }(e_{4j})_x-\epsilon \frac{\|x\|^2+\gamma}{E_{\gamma,\delta}(j,x)}{\ }(e_{4j+2})_x,\\ 
(\varphi e_{4j+2})_x &=\frac{j\delta}{E_{\gamma,\delta}(j,x)}{\ }(e_{4j-1})_x-\frac{\|x\|^2+\gamma}{E_{\gamma,\delta}(j,x)}{\ }(e_{4j+1})_x,
\end{align*}
$$\varphi e_{4k+3}=0$$
for $j= \overline{1,k}$ and $x\in \overline{M}$. 
Let the Riemannian metric $g$ be given by \linebreak $g(e_{i},e_{j})=\nolinebreak\delta _{ij}$, $i, j= \overline{1,4k+3}$. 
Then, $(\overline M, \varphi,\xi,\eta, g )$ is an almost ($\epsilon$)-contact metric manifold. It's to be noticed that, for $\epsilon=-1$, it is an almost contact metric manifold, and, for $\epsilon=1$, it is an almost paracontact metric manifold. 

We define the following submanifold of $\overline M$: 
$$M:=\{(x_{1},\ldots, x_{4k+3})\in \mathbb{R}^{4k+3} \ | \ x_{4j+1}=x_{4j+2}=0,\ j=\overline{1,k}\}.$$

Considering 
$D_0=\langle e_{1},e_{2}\rangle$ and $D_{j}=\langle e_{4j-1},e_{4j}\rangle, \ j=\overline{1,k}$, we notice that, for $\gamma>0$, $M$ is a nontrivial generic submanifold and a $k$-pointwise slant submanifold of $\overline M$, with $TM=\oplus_{i=0}^k D_i\oplus \langle\xi\rangle$, while, for $\gamma=0$, it is a $k$-pointwise slant submanifold but not a generic one. The cor\-re\-spond\-ing \mbox{$k$-point}\-wise slant distribution is $\oplus_{i=0}^k D_i$, where $D_0$ is the invariant component and the $D_j$'s, $j=\overline{1,k}$, are pointwise slant distributions with corresponding slant functions 
$$\theta _{j}(x)=\arccos \left(\frac{\|x\|^2+\gamma}{E_{\gamma,\delta}(j,x)}\right),\; x\in M \text{ for }j=\overline{1,k}\,.$$
Thus, $\oplus_{i=1}^kD_i$ is the proper $k$-pointwise slant distribution associated to $M$.

Let us consider the distributions $G_j:=\langle e_{4j+1}, e_{4j+2}\rangle$, $j=\overline{1,k}$, in $(TM)^{\perp}$. Then, $\oplus_{j=1}^kG_j$ is the dual $k$-pointwise slant distribution of $\oplus_{j=1}^kD_j$. We have $f(G_j)=D_j$ for $j=\overline{1,k}$, and $\oplus_{j=1}^kD_j$ is the dual $k$-pointwise slant distribution of $\oplus_{j=1}^kG_j$.

\end{example}

%%%%%%%%%%%%%%%%%%%%%%%%%%%%%%%%
\section{$k$-pointwise slant distributions in almost Hermitian and almost product Riemannian settings}\label{pointwise_alm_hermitian}

Throughout this section, we  will provide, as for $k$-slant distributions in the almost Hermitian and almost product Riemannian settings, a unitary approach for $k$-pointwise slant distributions in these settings. 

Let $\overline{M}$ be a smooth manifold equipped with a  $(1,1)$-tensor field $\varphi$ and a Riemannian metric $g$ satisfying, for a fixed $\epsilon\in \{-1,1\}$, 
\begin{equation}
\varphi^2=\epsilon I\, \text{ and }\,
g(\varphi X,Y)=\epsilon g(X,\varphi Y)\, 
\text{ for any } X,Y\in T\overline{M}. \nonumber
\end{equation}

\medskip 
In the sequel, let $M$ be $\overline{M}$ or an immersed submanifold of $(\overline{M},\varphi,g)$. Let $k\in \mathbb{N}^*$, $D=\oplus_{i=0}^k{D_i}$ be a $k$-pointwise slant distribution on $M$ with $D_0$ the invariant component, and $G$ be the orthogonal complement of $D$ in $T\overline{M}$. Let $\theta_0=0$ and $\theta_i$ denote the slant function of $D_i$ for $i=\overline{1,k}$. 

\begin{remark}\label{p142}
We notice that, with only a few exceptions (Remarks \ref{p115} (ii), \ref{p180}, Theorem \ref{p178}, and Proposition \ref{p102}), all the other results obtained for \mbox{$k$-point}\-wise slant distributions in the almost contact metric and almost paracontact metric settings (Definition, Propositions, Theorems, Corollaries, Remarks from \ref{p115} to \ref{p112}, together with Lemma \ref{p92}, Propositions \ref{p2}, \ref{p3}, \ref{p67}, \ref{p7}, \ref{p14}, \ref{p16}, Theorems \ref{p1}, \ref{p17}, and Corollaries \ref{p62}, \ref{p15}, \ref{p18}--\ref{p20}) remain also valid in the almost Hermitian and almost product Riemannian settings, with similar proofs, taking into account that in the present settings $D\oplus G=T\overline{M}$. 
\end{remark}

Instead of Remarks \ref{p115} (ii), \ref{p180}, Theorem \ref{p178}, and Proposition \ref{p102}, we have: 

\begin{remark}\label{p116}
\begin{equation}
\varphi(D \oplus G)=D \oplus G. \nonumber 
\end{equation}
\end{remark}

\begin{proposition}\label{p117}
For any $X\in D$ and\, $U\in G$, we have:
\begin{align}
f^2X+fwX &= \epsilon X, \nonumber\\ 
wfX+w^2X &=0, \nonumber \\ 
f^2U+fwU &=0, \nonumber \\ 
wfU+w^2U &=\epsilon U. \nonumber
\end{align}
\end{proposition}

Taking into account Remark \ref{p32} and Propositions \ref{p176}, \ref{p86}, we obtain 

\begin{theorem}\label{p179}
Let $\mathfrak D$ be a non-null distribution on $M$ decomposable into an orthogonal sum of regular distributions, $\mathfrak D=\oplus_{i=0}^k{\mathfrak D_i}$ with $\mathfrak D_i\neq \{0\}$ for $i=\nolinebreak\overline{1,k}$ and $\mathfrak D_0$ invariant (possible null). Let $pr_i$ denote the projection operator onto $\mathfrak D_i$ for $i=\overline{0,k}$, $f$ the component of $\varphi$ into $\mathfrak D$, and $\theta_0=0$. If $f(\mathfrak D_i)\subseteq \mathfrak D_i$ for $i=\overline{1,k}$, then the following assertions are equivalent: \\ 
\hspace*{7pt} (a) There exist $k$ distinct continuous functions  $\theta_i:{M}\rightarrow (0,\frac{\pi}{2}]$, $i=\nolinebreak\overline{1,k}$, such that  
\begin{equation}\nonumber
f^2X= \epsilon \sum_{i=0}^k\cos^2\theta_i\cdot pr_iX \ \text{for any} \  X\in \mathfrak D;
\end{equation}
\hspace*{7pt} (b) $\mathfrak D$ is a $k$-pointwise slant distribution with slant functions $\theta_i$ corresponding to $\mathfrak D_i$, $i=\overline{1,k}$. 
\end{theorem}

\begin{remark}\label{p181}
Theorem \ref{p179} provides a necessary and sufficient condition for a submanifold $M$ of $\overline M$ to be a $k$-pointwise slant submanifold, considering $\mathfrak D=TM$\, if \,$TM=\oplus_{i=0}^k{\mathfrak D_i}$.
\end{remark}

\medskip 
Let $M$ be an immersed submanifold of $\overline{M}$. 
We have the following equivalent formulation of the definition of a $k$-pointwise slant submanifold. 

\begin{definition}\label{p114}
We say that $M$ is a \textit{$k$-pointwise slant submanifold} of $\overline{M}$ if there exists an orthogonal decomposition of $TM$ into regular distributions, 
$$TM=\oplus_{i=0}^k{D_i}=:D$$
with $D_i\neq \{0\}$ for $i= \overline{1,k}$ and $D_0$ possible null, and there exist $k$ distinct continuous functions $\theta_i :M \rightarrow (0,\frac{\pi}{2}]$, $i=\overline{1,k}$, such that: \\ 
\hspace*{7pt} (i)\; $\widehat{(\varphi X_x, (D_i)_x)}=\theta_i(x)$ for any $X\in D_i\verb=\=\{0\}$, $i=\overline{1,k}$, and $x\in {M}$ with $X_x\neq 0$;\\ 
\hspace*{7pt} (ii)\, $\varphi X\in D_0$ for any $X\in D_0$ (i.e., $\widehat{(\varphi X, TM)}=0$ for $X\in D_0\verb=\=\{0\}$, and $f(D_0)\subseteq D_0$);\\ 
\hspace*{7pt} (iii) $f(D_i)\subseteq D_i$ for $i=\overline{1,k}$. 
\end{definition}

\begin{remark}
In view of (iii) and Remark \ref{p79}, we have: \\ 
(a) Condition (i) can be replaced by \\ 
\hspace*{7pt} (i') $\widehat{(\varphi X_x, T_xM)}=\theta_i(x)$ for any $X\in D_i\verb=\=\{0\}$, $i=\overline{1,k}$, and $x\in {M}$ with $X_x\neq 0$; \\ 
(b) Condition (iii) can be replaced by \\ 
\hspace*{7pt} (iii') $\varphi(D_i)\perp D_j$ for any $i\neq j$ from $\{1,\ldots,k\}$. 
\end{remark}

\begin{remark}\label{p118}
All the results from this section related to distributions can be transferred to an arbitrary $k$-pointwise slant submanifold $M$ of $(\overline{M},\varphi,g)$ by taking the $k$-pointwise slant distribution $D=TM$.
Thus, these results remain also valid in the $k$-pointwise slant submanifold framework. 
\end{remark}

\begin{remark}\label{p120}
We describe below, by a unitary approach, the notion of \textit{generic submanifold}, in sense of Ronsse \cite{ronsse}, \textit{of an almost Hermitian} or \textit{almost product Riemannian manifold}, relating it to the concept of $k$-pointwise slant submanifold.  

Let $M$ be an immersed submanifold of ($\overline{M},\varphi,g)$ with $\varphi$ and $g$ satisfying (\ref{h2}), where $\epsilon \in \{-1,1\}$, and let $fZ$ be the tangential component of $\varphi Z$ for any $Z\in T{M}$. 

Since $f^2$ is symmetric (see Lemmas \ref{lema1}, \ref{lema2} and Remark \ref{p194}) and $f$ is the composition of a projection and an isometry, denoting by $\lambda_i(x)$, $i=\overline{1,m(x)}$, the distinct eigenvalues of $f_x^2$ acting on the tangent space $T_xM$ for $x\in M$, these eigenvalues are all  contained in $[-1,0]$ for $\epsilon =-1$ and in $[0,1]$ for $\epsilon =1$. Let $\mathfrak D^i_x$ denote the eigenspace of $f_x^2$ cor\-re\-spond\-ing to $\lambda_i(x)$, $i=\overline{1,m(x)}$, $x\in M$. 

Now, we will consider that $M$ is a generic submanifold of $\overline{M}$ in sense of Ronsse, that is, $m(x)$ is independent of $x$ (and it will be denoted by $m$), and the dimension of $\mathfrak D^i_x$ is also independent of $x$ for $i=\overline{1,m}$. Moreover, if one of the $\lambda_i(x)$ is $0$ or $1$, say $\lambda_{i_0}(x_0)$, then $\lambda_{i_0}(x)$ has the same value for all $x\in M$. Taking into account the smoothness of the vector fields, the functions $\lambda_1, \dots , \lambda_m$ are continuous on $M$. Denoting by $\mathfrak D_i$ the distribution corresponding to the family $\{\mathfrak D^i_x: x\in M\}$ for $i=\overline{1,m}$, we obtain for $TM$ the orthogonal decomposition  
$$TM={\mathfrak D_1}\oplus\ldots\oplus{\mathfrak D_m}\,.$$

We notice that each distribution $\mathfrak D_i$ is invariant under $f$. Moreover, for $\epsilon =-1$, if $\lambda_i(x)\neq0$ for an $i\in\{1,\ldots,m\}$ and a point $x\in M$, then $\lambda_i(x)\neq0$ for that $i$ and all $x\in M$, and the cor\-re\-spond\-ing distribution $\mathfrak D_i$ is of even dimension. 
For any $x\in M$ and $i\in\{1,\ldots,m\}$, let $\alpha_i(x)\in [0,1]$ denote the nonnegative value for which $\lambda_i(x)=\epsilon\alpha_i^2(x)$. Then, for $X\in  \mathfrak{D}_i\verb=\=\{0\}$ and $x\in M$ such that $X_x\neq 0$, we get 
$$\|f_xX_x\|^2=\epsilon\lambda_i(x)g_x(X_x,X_x)= \alpha_i^2(x) \|X_x\|^2= \alpha_i^2(x) \|\varphi_x X_x\|^2$$
and, hence, $\|f_x X_x\|=\alpha_i(x) \|X_x\|=\alpha_i(x) \|\varphi_x X_x\|$, so $\alpha_i(x)=\cos\zeta_i(x)$, where $\zeta_i(x)$ is the angle between $\varphi_x X_x$ and $T_xM$, which is the same for any nonzero $X\in \mathfrak D_i$ with $X_x\neq 0$. 

We conclude that the $\mathfrak D_i$'s, $i=\overline{1,m}$, are pointwise slant distributions with different corresponding slant functions $\zeta_i$ excepting the one of them that would correspond to $\alpha_i=1$ and would be invariant with respect to $\varphi$ (if it exists). Thus, $M$ is a $k$-pointwise slant submanifold of $\overline{M}$, where $k$ is $m$ or $m-1$. 
\end{remark}

We deduce: 

\begin{proposition}\label{p127}
Any generic submanifold of an almost Hermitian or almost product Riemannian man\-i\-fold is a $k$-pointwise slant submanifold. 
\end{proposition}

We will show through an example that a $k$-pointwise slant submanifold is not necessarily a generic one. 

\begin{example}\label{ex5}
Let $\overline{M}=\mathbb{R}^{4k+2}$ be the Euclidean space for some $k\geq 2$, with the canonical coordinates 
$(x_{1},\ldots, x_{4k+2})$, and let $\{e_{1}=\frac{\partial }{\partial x_{1}},\ldots,e_{4k+2}=\nolinebreak\frac{\partial }{\partial x_{4k+2}}\}$ be the natural basis in the tangent bundle. Let $\epsilon \in \{-1,1\}$, $\gamma\geq1$, and 
$E_\gamma (j,x)=\sqrt{2\|x\|^4+2(\gamma+j-1)\|x\|^2+(\gamma^2-2\gamma +j^2 +1)}$ for any $j\in \mathbb{N^*}$ and $x\in \overline{M}$. Let us define a $(1,1)$-tensor field $\varphi$ by: 
$$\varphi e_{1} = e_{2}, \quad \varphi e_{2}=\epsilon e_{1},$$
$$(\varphi e_{4j-1})_x =\frac{\|x\|^2+\gamma-1}{E_\gamma (j,x)}{\ }(e_{4j})_x+\frac{\|x\|^2+j}{E_\gamma (j,x)}{\ }(e_{4j+2})_x,$$
$$(\varphi e_{4j})_x=\epsilon \frac{\|x\|^2+\gamma-1}{E_\gamma (j,x)}{\ }(e_{4j-1})_x-\frac{\|x\|^2+j}{E_\gamma (j,x)}{\ }(e_{4j+1})_x, $$
$$(\varphi e_{4j+1})_x =-\epsilon \frac{\|x\|^2+j}{E_\gamma (j,x)}{\ }(e_{4j})_x+\epsilon \frac{\|x\|^2+\gamma-1}{E_\gamma (j,x)}{\ }(e_{4j+2})_x,$$
$$(\varphi e_{4j+2})_x=\epsilon \frac{\|x\|^2+j}{E_\gamma (j,x)}{\ }(e_{4j-1})_x+\frac{\|x\|^2+\gamma-1}{E_\gamma (j,x)}{\ }(e_{4j+1})_x$$
for $j\in \{ 1,\ldots,k\}$ and $x\in \overline{M}$.
Then, with the Riemannian metric $g$ given by $g(e_{i},e_{j})=\delta _{ij}$, $i, j\in \{1,\ldots, 4k+2\}$, 
$(\overline{M},\varphi,g)$ is an almost Hermitian manifold for $\epsilon =-1$ and an almost product Riemannian manifold for $\epsilon =1$.

We define the submanifold $M$ of $\overline{M}$ by
$$
M:=\{(x_{1},\ldots, x_{4k+2})\in \mathbb{R}^{4k+2} \ | \ x_{4j+1}=x_{4j+2}=0,\ j= \overline{1,k}\}.$$
We will consider the distributions: 
$$
D_0=\langle e_{1},e_{2}\rangle,\quad D_{j}=\langle e_{4j-1},e_{4j}\rangle, \ j= \overline{1,k}\,.
$$

Then, $M$ is a generic submanifold and a $k$-pointwise slant submanifold of $\overline{M}$ for $\gamma>1$, and it is a $k$-pointwise slant submanifold but not a generic submanifold of $\overline{M}$ for $\gamma=1$. The corresponding $k$-pointwise slant distribution is $TM=\oplus_{i=0}^kD_i$, with
$D_0$ the invariant component and $D_j$, $j=\overline{1,k}$, pointwise slant distributions having the slant functions 
$$\theta _{j}(x)=\arccos \left(\frac{\|x\|^2+\gamma-1}{E_\gamma (j,x)}\right),\; x\in M,\ j=\overline{1,k}\,.$$ 
$\oplus_{i=1}^kD_i$ is the proper $k$-pointwise slant distribution associated to $M$.

We consider the distributions $L_i:=\langle e_{4i+1}, e_{4i+2}\rangle$ in $(TM)^{\perp}$, $i=\overline{1,k}$, and notice that $\oplus_{i=1}^kL_i$ is the dual $k$-pointwise slant distribution of $\oplus_{i=1}^kD_i$. We have $D_i= f(L_i)$, $i=\overline {1,k}$; hence, the dual $k$-pointwise slant distribution of $\oplus_{i=1}^kL_i$ is $\oplus_{i=1}^kD_i$. 
\end{example}

%%%%%%%%%%%%%%%%%%%%%%%%%%%%
\section{$k$-slant distributions via $k$-pointwise slant distributions}\label{slant_via_pointwise_slant}

\hfill{Motto: \textit{Repetitio est mater studiorum.}} 

\medskip 
In the sequel, our aim is to find sufficient conditions for $k$-pointwise slant distributions to be $k$-slant distributions, in various settings. In view of Theorem \ref{p143}, we notice that, for any pointwise slant component of a \mbox{$k$-point}\-wise slant distribution $D$, there is the same direct relation between the associated slant function and the corresponding eigenvalue function of $f^2|_D$. 
Related to this self-adjoint operator, Chen presents in \cite{chen2} (Lemma 3.1), in the almost Hermitian setting, a result linking the property of the eigenvalues of $f^2|_D$ to be constant to the condition for the tensor field to be parallel. Investigating that type of correspondence and using the above considerations, we will achieve our goal and also obtain some related results. In particular, we will obtain sufficient conditions for a $k$-pointwise slant submanifold to be a $k$-slant submanifold. Moreover, our study will lead to the introduction of a special subclass of distributions, the pointwise $k$-slant distributions, for which we will present corresponding results. Our statements will be valid in the almost Hermitian, the almost product Riemannian, the almost contact metric, and the almost paracontact metric setting.

\medskip 
Let $\overline{M}$ be a smooth manifold equipped with a Riemannian metric $g$ and a $(1,1)$-tensor field $\varphi$ satisfying (\ref{1}) for a fixed $\epsilon\in \{-1,1\}$, i.e., 
\begin{equation}
g(\varphi X,Y)=\epsilon g(X,\varphi Y) 
\text{ for any }X,Y\in T\overline{M}, \nonumber
\end{equation}
and 
\begin{equation}\label{97}
\varphi^2=\epsilon I \text{ \ or \ }\varphi^2=\epsilon (I-\eta\otimes\xi), 
\end{equation}
where $\xi$ is a fixed unitary vector field on the Riemannian manifold $(\overline{M},g)$ and $\eta$ denotes the dual $1$-form of $\xi$. In this way, our approach will be unitary for all of the above mentioned settings. 

We will consider that any sub\-man\-i\-fold $M$ of $\overline{M}$ we deal with satisfies the condition $\xi\in TM$ if the second formula in (\ref{97}) is valid. 
 
Throughout this section, $M$ will be $\overline{M}$ or an immersed submanifold of $\overline{M}$ if not specified otherwise. Let $k\in \mathbb{N}^*$ and $D=\oplus_{i=0}^k{D_i}$ be a \mbox{$k$-point}\-wise slant distribution on $M$, where $D_0$ is the invariant component, with $D\perp \xi$ if we consider the setting given by the second formula in (\ref{97}). Let $\theta_i$ denote the slant function of $D_i$ for $i=\overline{1,k}$, and let $\theta_0=0$. 

For any $Z\in T\overline{M}$, we have denoted by $fZ$ the component of $\varphi Z_M$ in $D$. We notice that $f^2|_D$ is symmetric with respect to $g$, and the $D_i$'s are regular distributions for $i=\overline{1,k}$. 
In view of Theorem \ref{p143}, for each $x\in M$ and $i\in \nolinebreak\{1,\dots ,k\}$, $(D_i)_x$ is a linear space composed entirely of eigenvectors of the only eigenvalue $\lambda_i(x)$ of $(f^2|_{D_i})_x$, and $\lambda_i(x)$ is different from $1$ or $(-1)$. 
Let $\overline\nabla$ denote the Levi-Civita connection on $\overline{M}$. 

\begin{proposition}\label{p129}
Let $i_0\in \{1,\ldots,k\}$ and\, $\overline\nabla_X  Y\in D$ for any $X,Y\in D_{i_0}$. Then, the following two assertions are equivalent: \\ 
1) $(\overline\nabla_X f^2) Y=0$ for any $X,Y\in D_{i_0}$. \\ 
2) i) $f^2(\overline\nabla_X Y)=\lambda_{i_0}\cdot \overline\nabla_X Y$ for any $X,Y\in D_{i_0}$; \\ 
\hspace*{7 pt} ii) $X(\lambda_{i_0})=0$ for any $X\in D_{i_0}$.
\end{proposition}

\begin{proof}
1)$\Rightarrow$2):  In view of Theorem \ref{p143}, we have $f^2Y=\lambda_{i_0}Y$ for any $Y\in \nolinebreak D_{i_0}$. Let $X,Y\in D_{i_0}$. From $(\overline\nabla_X f^2)Y=0$, we get 
$\overline\nabla_X (f^2Y)=\nolinebreak f^2(\overline\nabla_X Y)$; 
hence, 
$$X(\lambda_{i_0})\cdot Y+\lambda_{i_0}\cdot \overline\nabla_X Y =f^2(\overline\nabla_X Y).$$
For $Y_{i_0}$ a unitary vector field in $D_{i_0}$, $\overline\nabla_X Y_{i_0}$ and $f^2(\overline\nabla_X Y_{i_0})$ are orthogonal to $Y_{i_0}$; hence, taking $Y=Y_{i_0}$ above, $X(\lambda_{i_0})=0$. \\ 
It follows that $f^2(\overline\nabla_X Y)=\lambda_{i_0}\cdot \overline\nabla_X Y$ for any $X,Y\in D_{i_0}$. \\ 
2)$\Rightarrow$1):
Let $X,Y\in D_{i_0}$. Because $f^2Y=\lambda_{i_0}\cdot Y$ and 
$f^2(\overline\nabla_X Y)=\lambda_{i_0}\cdot \overline\nabla_X Y$, we have 
$$(\overline\nabla_X f^2) Y= \overline\nabla_X (\lambda_{i_0}\cdot Y)-\lambda_{i_0}\cdot (\overline\nabla_X Y)= X(\lambda_{i_0})\cdot Y= 0. \eqno\qedhere$$
\end{proof}

\begin{remark}\label{p144}
The condition $f^2(\overline\nabla_X Y)=\lambda_{i_0}\cdot \overline\nabla_X Y$ for $X,Y\in D_{i_0}$ doesn't mean that $\overline\nabla_X Y\in D_{i_0}$, even if localized in a point $x\in {M}$, because the linear space $(D_{i_0})_x$ is not, in general, the entire eigenspace in $D_x$ of the eigenvalue $\lambda_{i_0}(x)$. For this, it would be necessary that $\lambda_{i_0}(x)\neq \lambda_{i}(x)$ for all $i\neq i_0$. We will solve this problem, of the eigenspace, a little later. 
\end{remark}

\begin{corollary}\label{p130}
Let\, $\overline\nabla_X  Y\in D$ for any $X,Y\in D_{i}$, $i=\overline {1,k}$. 
Then, the following two assertions are equivalent: \\ 
1) $(\overline\nabla_X f^2) Y=0$ for any $X,Y\in D_{i}$, $i=\overline {1,k}$. \\ 
2) i)  $f^2(\overline\nabla_X Y)=\lambda_{i}\cdot \overline\nabla_X Y$ for any $X,Y\in D_{i}$, $i=\overline {1,k}$; \\ 
\hspace*{7 pt} ii) $X(\lambda_{i})=0$ for any $X\in D_{i}$, $i=\overline {1,k}$.
\end{corollary}

\begin{remark}\label{p131}
The equivalence in the above Corollary is, in particular, valid for $D$ completely integrable with respect to $\overline\nabla$ (i.e., $\overline\nabla_X  Y\in D$\, for any $X,Y\in D$) or for all of the $D_i$'s ($i=\overline {1,k}$) completely integrable with respect to $\overline\nabla$. 
\end{remark}

With the same type of justifications, we get the following results. 

\begin{proposition}\label{p132}
Let $i_0\in \{1,\ldots,k\}$ and\, $\overline\nabla_X  Y\in D$ for any $X\in T{M}$ and\, $Y\in D_{i_0}$. Then, the following two assertions are equivalent: \\ 
1) $(\overline\nabla_X f^2) Y=0$ for any $X\in T{M}$ and\, $Y\in D_{i_0}$. \\ 
2) i) $f^2(\overline\nabla_X Y)=\lambda_{i_0}\cdot \overline\nabla_X Y$ for any $X\in T{M}$ and\, $Y\in D_{i_0}$; \\ 
\hspace*{7 pt} ii) the restriction of $D_{i_0}$ to any connected component of ${M}$ is a slant distribution \textup($\lambda_{i_0}$ is constant on any connected component of ${M}$\textup). In particular, if ${M}$ is a connected manifold, $D_{i_0}$ is a slant distribution. 
\end{proposition}

\begin{theorem}\label{p133}
Let\, $\overline\nabla_X  Y\in D$ for any $X\in T{M}$ and\, $Y\in \oplus_{i=1}^k D_i$. Then, the following two assertions are equivalent: \\ 
1) $(\overline\nabla_X f^2) Y=0$ for any $X\in T{M}$ and\, $Y\in \oplus_{i=1}^k D_i$. \\ 
2) i) $f^2(\overline\nabla_X Y)=\lambda_{i}\cdot \overline\nabla_X Y$ for any $X\in T{M}$ and\, $Y\in D_{i}$, $i=\overline{1,k}$; \\ 
\hspace*{7 pt} ii) the restriction of $D$ to any connected component of ${M}$ is a \mbox{$k'$-slant} distribution, where $k'\in \{1,\dots, k\}$ depends on the values of the $\lambda_i$'s on the considered connected component \textup($\lambda_{i}$ is constant on any connected component of ${M}$, $i=\overline{1,k}$\textup).
\end{theorem}

\begin{theorem}\label{p169}
Let\, $\overline\nabla_X  Y\in D$ for any $X\in T{M}$ and\, $Y\in \oplus_{i=1}^k D_i$. If ${M}$ is a connected manifold, then the following two assertions are equivalent: \\ 
1) $(\overline\nabla_X f^2) Y=0$ for any $X\in T{M}$ and\, $Y\in \oplus_{i=1}^k D_i$. \\ 
2) i) $f^2(\overline\nabla_X Y)=\lambda_{i}\cdot \overline\nabla_X Y$ for any $X\in T{M}$ and\, $Y\in D_{i}$, $i=\overline{1,k}$; \\ 
\hspace*{7 pt} ii) $D$ is a $k$-slant distribution \textup($\lambda_{1},\lambda_{2},\ldots, \lambda_{k}$ are constant and different on ${M}$\textup).
\end{theorem}

\begin{proposition}\label{p134}
Let $i_0\in \{1,\ldots,k\}$ and\, $\overline\nabla_X  Y\in D$ for any $X\in D$ and\, $Y\in D_{i_0}$. Then, the following two assertions are equivalent: \\ 
1) $(\overline\nabla_X f^2) Y=0$ for any $X\in D$ and\, $Y\in D_{i_0}$. \\ 
2) i) $f^2(\overline\nabla_X Y)=\lambda_{i_0}\cdot \overline\nabla_X Y$ for any $X\in D$ and\, $Y\in D_{i_0}$; \\ 
\hspace*{7 pt} ii) $X(\lambda_{i_0})=0$ for any $X\in D$.
\end{proposition}

\begin{corollary}\label{p135}
Let\, $\overline\nabla_X  Y\in D$ for any $X\in D$ and\, $Y\in \oplus_{i=1}^k D_i$. Then, the following two assertions are equivalent: \\ 
1) $(\overline\nabla_X f^2) Y=0$ for any $X\in D$ and\, $Y\in \oplus_{i=1}^k D_i$. \\ 
2) i) $f^2(\overline\nabla_X Y)=\lambda_{i}\cdot \overline\nabla_X Y$ for any $X\in D$ and\, $Y\in D_{i}$, $i=\overline{1,k}$; \\ 
\hspace*{7 pt} ii) $X(\lambda_{i})=0$ for any $X\in D$, $i=\overline{1,k}$. 
\end{corollary}

\begin{theorem}\label{p188}
Let $D$ be completely integrable with respect to $\overline\nabla$. If $M'$ is a connected submanifold of ${M}$ such that\, $TM'=D$, then the following two assertions are equivalent: \\ 
1) $(\overline\nabla_X f^2) Y=0$ for any $X\in D$ and\, $Y\in \oplus_{i=1}^k D_i$. \\ 
2) i) $f^2(\overline\nabla_X Y)=\lambda_{i}\cdot \overline\nabla_X Y$ for any $X\in D$ and\, $Y\in D_{i}$, $i=\overline{1,k}$; \\ 
\hspace*{7 pt} ii) There is $k'\in \{1,\dots, k\}$ such that $M'$ is a $k'$-slant submanifold of ${M}$. 
\end{theorem}

\begin{proof}
Applying Corollary \ref{p135} to the present setting, condition 2) ii), \linebreak $X(\lambda_{i})=0$ for any $X\in TM'$, $i=\overline{1,k}$, is equivalent to the fact that the functions $\lambda_i$ are constant (but not necessarily different) on $M'$ for $i=\overline{1,k}$, and these constants represent the eigenvalues different from $1$ or $(-1)$ of $(f^2|_D)_x$ for $x \in M'$. Finally, we apply Theorem \ref{p143} and Proposition \ref{p31}. 
\end{proof}

%%%%%%%%%%%%%%%%%%%%%%

Let $M$ be a $k$-pointwise slant submanifold of $\overline{M}$, $\nabla$ be the Levi-Civita connection induced by $\overline{\nabla}$ on $M$, and $D$ be given by $D:=\oplus_{i=0}^k D_i$ for\, $TM=\nolinebreak\oplus_{i=0}^k D_i$ in the setting given by the first formula in (\ref{97}) or for\, $TM=\oplus_{i=0}^k D_i\oplus \langle\xi\rangle$ in the case of the second formula in (\ref{97}), where $D_0$ denotes the invariant component of $D$. 
Observe that the $D_i$'s, ${i=\overline{1,k}}$, are regular distributions on $M$ whose localizations in each point $x\in {M}$ are linear spaces consisting of eigenvectors corresponding to the eigenvalues $\lambda_i(x)$, ${i=\overline{1,k}}$, different from $1$ or $(-1)$ of $(f^2|_{D})_x$, respectively. 

\begin{remark}\label{p189}
Taking into account that $f(D_i)\subseteq D_i$ for all $i=\overline{1,k}$, we notice that $f^2\vert_{TM}$ is symmetric relative to $g$. It implies that, even if the condition ''$\nabla_X Y\in D$ for any $X,Y\in D$'' would not be satisfied (in the almost \mbox{($\epsilon$)-con}\-tact metric setting, when $TM=D\oplus\langle\xi\rangle$), the proof of Proposition \ref{p129} remains fully valid if, in it, $\overline\nabla$ is everywhere replaced with $\nabla$. Thus, relative to $\nabla$, we get the following results. 
\end{remark}

\begin{proposition}\label{p136}
For\, $i_0\in \{1,\ldots,k\}$, the following two assertions are equivalent: \\ 
1) $(\nabla_X f^2) Y=0$ for any $X,Y\in D_{i_0}$. \\ 
2) i) $f^2(\nabla_X Y)=\lambda_{i_0}\cdot \nabla_X Y$ for any $X,Y\in D_{i_0}$;\\ 
\hspace*{7 pt} ii) $X(\lambda_{i_0})=0$ for any $X\in D_{i_0}$.
\end{proposition}

\begin{corollary}\label{p137}
The following two assertions are equivalent: \\ 
1) $(\nabla_X f^2) Y=0$ for any $X,Y\in D_{i}$, $i=\overline {1,k}$. \\ 
2) i) $f^2(\nabla_X Y)=\lambda_{i}\cdot \nabla_X Y$ for any $X,Y\in D_{i}$, $i=\overline {1,k}$; \\ 
\hspace*{7 pt} ii) $X(\lambda_{i})=0$ for any $X\in D_{i}$, $i=\overline {1,k}$.
\end{corollary}

\begin{proposition}\label{p138}
For\, $i_0\in \{1,\ldots,k\}$, the following two assertions are equivalent: \\ 
1) $(\nabla_X f^2) Y=0$ for any $X\in TM$ and\, $Y\in D_{i_0}$. \\ 
2) i) $f^2(\nabla_X Y)=\lambda_{i_0}\cdot \nabla_X Y$ for any $X\in TM$ and\, $Y\in D_{i_0}$; \\ 
\hspace*{7 pt} ii) the restriction of $D_{i_0}$ to any connected component of $M$ is a slant distribution \textup($\lambda_{i_0}$ is constant on any connected component of $M$\textup).
\end{proposition}

\begin{theorem}\label{p139}
The following two assertions are equivalent: \\ 
1) $(\nabla_X f^2) Y=0$ for any $X\in TM$ and\, $Y\in \oplus_{i=1}^k D_i$. \\ 
2) i)  $f^2(\nabla_X Y)=\lambda_{i}\cdot \nabla_X Y$ for any $X\in TM$ and\, $Y\in D_{i}$, $i=\overline{1,k}$; \\ 
\hspace*{7 pt} ii) any open connected component of $M$ is a $k'$-slant submanifold of $\overline{M}$, where $k'\in \{1,\dots, k\}$ depends on the values of the $\lambda_i$'s on the considered connected component \textup($\lambda_{i}$ is constant on any connected component of $M$ for $i=\overline{1,k}$\textup). 
\end{theorem}

\begin{theorem}\label{p190}
If $M$ is a connected submanifold of $\overline{M}$, then the following two assertions are equivalent: \\ 
1) $(\nabla_X f^2) Y=0$ for any $X\in T{M}$ and\, $Y\in \oplus_{i=1}^k D_i$. \\ 
2) i) $f^2(\nabla_X Y)=\lambda_{i}\cdot \nabla_X Y$ for any $X\in T{M}$ and\, $Y\in D_{i}$, $i=\overline{1,k}$; \\ 
\hspace*{7 pt} ii) $M$ is a $k$-slant submanifold of $\overline{M}$ \textup($\lambda_{1},\lambda_{2},\ldots, \lambda_{k}$ are constant and different on $M$\textup).
\end{theorem}

Let us consider again $M$ to be $\overline{M}$ or an immersed submanifold of $\overline{M}$. 

To solve the problem of the eigenspace, mentioned in Remark \ref{p144}, we will impose to the considered $k$-pointwise slant distribution $D$ the least restrictive requirement, which is a necessary condition for $(D_i)_x$ to be the entire eigenspace in $D_x$ corresponding to the eigenvalue $\lambda_i(x)$ for $i=\overline{1,k}$ and $x \in M$, namely: $\lambda_i(x)\neq \lambda_j(x)$ for any $i\neq j$ and $x \in M$. 
It means that, additionally to the orthogonal decomposition of $D$ into regular distributions, for any point $x$, we will have an orthogonal decomposition of $D_x$ into $k$ slant subspaces with different slant angles and eventually an invariant subspace. To achieve this, we need a stronger concept than that of $k$-pointwise slant distribution. It naturally leads us to the following definition. 

\begin{definition}\label{p145} 
A non-null distribution $D$ on ${M}$ will be called a \textit{pointwise \mbox{$k$-slant} distribution} if there exists an orthogonal decomposition of $D$ into regular distributions, 
$$D=\oplus_{i=0}^k{D_i}$$
with $D_i\neq \{0\}$ for $i=\overline{1,k}$ and $D_0$ possible null, and there exist $k$ pointwise distinct continuous functions $\theta_i:{M}\rightarrow (0,\frac{\pi}{2}]$ (i.e., continuous  and $\theta_i(x)\neq\nolinebreak \theta_j(x)$ for any $i\neq j$ and any point $x \in M$), $i=\overline{1,k}$, such that: \\ 
\hspace*{7pt} (i)\; $D_i$ is a pointwise $\theta_i$-slant distribution for $i=\overline{1,k}$;\\ 
\hspace*{7pt} (ii)\, $\varphi X\in D_0$ for any $X\in D_0$ 
(i.e., $\widehat{(\varphi X, D)}=0=:\theta_0$ for $X\in D_0$ with $\varphi X\neq 0$, and $f(D_0)\subseteq D_0$);\\ 
\hspace*{7pt} (iii) $f(D_i)\subseteq D_i$ for $i=\overline{1,k}$.

\medskip 
We will also call $D$ a \textit{pointwise $(\theta_1,\theta_2,\ldots,\theta_k)$-slant distribution}. 

$D_0$ represents the \textit{invariant component} and $\oplus_{i=1}^kD_i$ the \textit{proper pointwise \mbox{$k$-slant} component} of $D$. 

We will call the distribution $D=\oplus_{i=0}^kD_i$ a \textit{proper pointwise \mbox{$k$-slant} distribution} if $D_0=\{0\}$. 
\end{definition}

\begin{remark}\label{p146} \ \\ 
(a) Condition (i) is equivalent to \\ 
\hspace*{7pt} (i') $\varphi v \neq 0$, and $\widehat{(\varphi v, D_x)}=\theta_i(x)$ for any $x\in {M}$ and $v\in (D_i)_x\verb=\=\{0\}$, $i=\overline{1,k}$.\\ 
(b) Condition (iii) can be replaced by \\ 
\hspace*{7pt} (iii') $\varphi(D_i)\perp D_j$ for any $i\neq j$ from $\{1,\ldots,k\}$.
\end{remark}

\begin{remark}\label{p80}
For particular values of $k$ and of the slant functions, we get the following types of pointwise $k$-slant distributions. 

For $k=1$ and $D_0=\{0\}$, $D$ is a \textit{pointwise slant distribution}. For $k=1$, $D_0\neq \{0\}$, and $\theta_1$ different from the constant function $\frac{\pi}{2}$, $D$ is a \textit{pointwise semi-slant distribution}. For $k=2$ and $D_0=\{0\}$, $D$ is a \textit{pointwise bi-slant distribution}; it is a \textit{pointwise hemi-slant distribution} if one of the slant functions is constant on ${M}$, equal to $\frac{\pi}{2}$. 
\end{remark}

Corresponding to the latter concept of distribution, we have the concept of pointwise $k$-slant submanifold. 

\begin{definition}\label{p147}
For $M$ an immersed submanifold of $\overline{M}$ and $k\in \mathbb{N}^*$, we will call $M$ a \textit{pointwise $k$-slant submanifold} of $\overline{M}$ if $TM$ is a pointwise \mbox{$k$-slant} distribution, $\oplus_{i=0}^kD_i$ (where $D_0$ denotes the invariant component), relative to $T\overline{M}$. 

$\oplus_{i=1}^kD_i$ will be called the \textit{proper pointwise $k$-slant distribution associated} to $M$. 

We will call $M$ a \textit{pointwise $(\theta_1,\theta_2,\ldots,\theta_k)$-slant sub\-man\-i\-fold} if we want to specify the slant functions $\theta_i$. 

A pointwise $k$-slant submanifold $M$ will be called \textit{proper} if $TM$ is a proper pointwise $k$-slant distribution. 
\end{definition}

As examples of pointwise $k$-slant submanifolds, and, implicit, of pointwise $k$-slant distributions, we have: 

\begin{example}\label{ex6}
Replacing everywhere in Example \ref{ex4} the term ''generic'' with ''pointwise $k$-slant'', we obtain a pointwise $k$-slant submanifold which is also a $k$-pointwise slant submanifold for $\gamma>0$, but we obtain a $k$-pointwise slant submanifold which is not pointwise $k$-slant for $\gamma=0$ in the almost contact metric and almost paracontact metric settings.  
\end{example}

\begin{example}\label{ex7}
Replacing everywhere in Example \ref{ex5} the term ''generic'' with ''pointwise $k$-slant'', we obtain a pointwise $k$-slant submanifold which is also a $k$-pointwise slant submanifold for $\gamma>1$, but we obtain a $k$-pointwise slant submanifold which is not pointwise $k$-slant for $\gamma=1$ in the almost Hermitian and almost product Riemannian settings. 
\end{example}

\begin{proposition}\label{p155}
Any pointwise $k$-slant distribution is a $k$-pointwise slant distribution, but the converse is not true, in any of the considered settings. 
\end{proposition}

\begin{proof}
The first statement follows directly from the definition. 
The last statement is illustrated in Example \ref{ex6} for $\gamma=0$, in the almost contact metric and almost paracontact metric settings, and in Example \ref{ex7} for $\gamma=1$, in the almost Hermitian and almost product Riemannian settings. Thus, the statement that a $k$-pointwise slant distribution is not always a pointwise $k$-slant distribution is proven. 
\end{proof}

\begin{remark}\label{p200}
All the results valid for $k$-pointwise slant distributions are also valid for pointwise $k$-slant distributions. 
Taking into account that a pointwise \mbox{$k$-slant} distribution is a $k$-pointwise slant distribution for which the slant functions of the pointwise slant components are pointwise distinct, we conclude that, in such situations, the statements relating to $k$-pointwise slant distributions can be rewritten and are valid for pointwise $k$-slant distributions. 
In particular, the image of a proper pointwise $k$-slant distribution in its orthogonal complement through $w$ is a proper pointwise $k$-slant distribution. 
\end{remark}

\begin{theorem}\label{p215}
If\, $D=\oplus_{i=0}^k D_i$ is a pointwise $k$-slant distribution on $M$, with $D_0$ the invariant component, $\xi \perp D$ (if\, $\xi$ exists), and $G$ is the orthogonal complement in $T\overline{M}$ of $D$ or of\, $D\oplus \langle \xi\rangle$ (if\, $\xi$ exists), then 
\begin{equation}\nonumber 
G=\oplus_{i=1}^kw({D_i})\oplus H, \text{ where }f(H)=\{0\}. 
\end{equation}
The distribution $G$ is a pointwise $k$-slant distribution with $H$ the invariant component and\, $\oplus_{i=1}^kw(D_i)$ the proper pointwise $k$-slant com\-po\-nent, the pointwise slant distribution $w(D_i)$ having the same slant function $\theta_i$ as $D_i$ for $i= \overline{1,k}$.
\end{theorem}

\begin{definition}\label{p216}
We will call\, $\oplus_{i=1}^kw(D_i)$ \textit{the dual pointwise $k$-slant distribution} of\, $\oplus_{i=1}^kD_i$.
\end{definition}

\begin{remark}\label{p217}
In the same way we defined the dual of the proper pointwise \mbox{$k$-slant} component $\oplus_{i=1}^kD_i$ of the distribution $D$ by means of $w$, we can construct the dual of the proper pointwise $k$-slant component $\oplus_{i=1}^kw(D_i)$ of the distribution $G$ by means of $f$. This will be $f(\oplus_{i=1}^kw(D_i))=\oplus_{i=1}^kfw(D_i)$. 
\end{remark}

\begin{corollary}\label{p218}
The dual of the proper pointwise $k$-slant distribution \linebreak $\oplus_{i=1}^kw(D_i)$, which is $\oplus_{i=1}^kf(w(D_i))$, is precisely the pointwise $k$-slant distribution $\oplus_{i=1}^kD_i$.
\end{corollary}

Revisiting Remarks \ref{p119} and \ref{p120}, we conclude: 

\begin{proposition}\label{p208}
In any of the considered settings (almost Hermitian, almost product Riemannian, almost contact metric, or almost paracontact metric), any generic submanifold of a Riemannian manifold is a pointwise \mbox{$k$-slant} submanifold. 
\end{proposition}

Moreover, 

\begin{proposition}\label{p209}
Any pointwise $k$-slant submanifold which is not an anti-invariant or a CR submanifold and whose non-constant slant functions don't take the value $\frac{\pi}{2}$ is a generic submanifold in any of the considered settings. 
\end{proposition}

We will show through the following examples that a pointwise \mbox{$k$-slant} submanifold is not necessarily a generic one, in any of the mentioned settings. 

\begin{example}\label{ex8}
Let $\overline M=\mathbb{R}^{4k+3}$ be the Euclidean space for some $k\geq 2$, with the canonical coordinates 
$(x_{1},\ldots, x_{4k+3})$, and let $\{e_{1}=\frac{\partial }{\partial x_{1}},\ldots,e_{4k+3}=\frac{\partial }{\partial x_{4k+3}}\}$ be the natural basis in the tangent bundle. Let $\epsilon\in \{-1,1\}$, $\gamma\geq 0$, $\delta>0$, and 
$E_{\gamma,\delta}(j,x)=\sqrt{\|x\|^4+2[(j-1)\delta+\gamma] \|x\|^2+\delta^2+[(j-1)\delta+\gamma]^2}$\, 
for any $j\in \mathbb{N^*}$ and $x\in \overline M$. 

Let us define a vector field $\xi $, a $1$-form $\eta $, and a $(1,1)$-tensor field $\varphi$ by: 
$$\xi =e_{4k+3}, \quad \eta =dx_{4k+3},$$
$$
\varphi e_{1} = e_{2}, \quad \varphi e_{2}=\epsilon e_{1},$$
\begin{align*}
(\varphi e_{4j-1})_x &=\frac{\|x\|^2+(j-1)\delta+\gamma}{E_{\gamma,\delta}(j,x)}{\ }(e_{4j})_x+\epsilon \frac{\delta}{E_{\gamma,\delta}(j,x)}{\ }(e_{4j+2})_x,\\ 
(\varphi e_{4j})_x &=\epsilon \frac{\|x\|^2+(j-1)\delta+\gamma}{E_{\gamma,\delta}(j,x)}{\ }(e_{4j-1})_x+\epsilon \frac{\delta}{E_{\gamma,\delta}(j,x)}{\ }(e_{4j+1})_x,\\ 
(\varphi e_{4j+1})_x &=\frac{\delta}{E_{\gamma,\delta}(j,x)}{\ }(e_{4j})_x-\epsilon \frac{\|x\|^2+(j-1)\delta+\gamma}{E_{\gamma,\delta}(j,x)}{\ }(e_{4j+2})_x,\\ 
(\varphi e_{4j+2})_x &=\frac{\delta}{E_{\gamma,\delta}(j,x)}{\ }(e_{4j-1})_x-\frac{\|x\|^2+(j-1)\delta+\gamma}{E_{\gamma,\delta}(j,x)}{\ }(e_{4j+1})_x,
\end{align*}
$$\varphi e_{4k+3}=0$$
for $j= \overline{1,k}$ and $x\in \overline{M}$. 
Let the Riemannian metric $g$ be given by \linebreak $g(e_{i},e_{j})=\nolinebreak\delta _{ij}$, $i, j= \overline{1,4k+3}$. 
We notice that, for $\epsilon=-1$, $(\overline M, \varphi,\xi,\eta, g )$ is an almost contact metric manifold, and, for $\epsilon=1$, it is an almost paracontact metric manifold. 

We define the following submanifold of $\overline M$: 
$$
M:=\{(x_{1},\ldots, x_{4k+3})\in \mathbb{R}^{4k+3} \ | \ x_{4j+1}=x_{4j+2}=0,\ j=\overline{1,k}\}.$$

We will consider 
$D_0=\langle e_{1},e_{2}\rangle$ and $D_{j}=\langle e_{4j-1},e_{4j}\rangle$, $j=\overline{1,k}$. 
We notice that, for $\gamma>0$, $M$ is a generic and a pointwise $k$-slant submanifold of $\overline M$, with $TM=\oplus_{i=0}^k D_i\oplus \langle\xi\rangle$, while, for $\gamma=0$, it is a pointwise \mbox{$k$-slant} submanifold of $\overline M$ but not a generic one. The cor\-re\-spond\-ing pointwise \mbox{$k$-slant} distribution is $\oplus_{i=0}^k D_i$, where $D_0$ is the invariant component and the $D_j$'s, $j=\overline{1,k}$, are pointwise slant distributions with corresponding slant functions 
$$\theta _{j}(x)=\arccos \left(\frac{\|x\|^2+(j-1)\delta+\gamma}{E_{\gamma,\delta}(j,x)}\right),\; x\in M,\ j=\overline{1,k}\,.$$
$\oplus_{i=1}^kD_i$ is the proper $k$-pointwise slant distribution associated to $M$.

Let us consider the distributions $G_j:=\langle e_{4j+1}, e_{4j+2}\rangle$, $j=\overline{1,k}$, in $(TM)^{\perp}$. Then, $\oplus_{j=1}^kG_j$ is the dual pointwise $k$-slant distribution of $\oplus_{j=1}^k D_j$. We have $f(G_j)=D_j$ for $j=\overline{1,k}$, and $\oplus_{j=1}^kD_j$ is the dual pointwise \mbox{$k$-slant} distribution of $\oplus_{j=1}^kG_j$.
\end{example}

\begin{example}\label{ex9}
Let $\overline{M}=\mathbb{R}^{4k+2}$ be the Euclidean space for some $k\geq 2$, with the canonical coordinates 
$(x_{1},\ldots, x_{4k+2})$, and let $\{e_{1}=\frac{\partial }{\partial x_{1}},\ldots,e_{4k+2}=\nolinebreak\frac{\partial }{\partial x_{4k+2}}\}$ be the natural basis in the tangent bundle. Let $\epsilon \in \{-1,1\}$, $\gamma\geq1$, and denote 
$E_\gamma (j,x)=\sqrt{2\|x\|^4+2(j+\gamma-1)\|x\|^2+(j^2+\gamma^2 +2j\gamma-4(j+\gamma)+5)}$ for any $j\in \mathbb{N^*}$ and $x\in \overline{M}$. Let us define a $(1,1)$-tensor field $\varphi$ by: 
$$\varphi e_{1} = e_{2}, \quad \varphi e_{2}=\epsilon e_{1},$$
$$(\varphi e_{4j-1})_x =\frac{\|x\|^2+j+\gamma-2}{E_\gamma (j,x)}{\ }(e_{4j})_x+\frac{\|x\|^2+1}{E_\gamma (j,x)}{\ }(e_{4j+2})_x,$$
$$(\varphi e_{4j})_x=\epsilon \frac{\|x\|^2+j+\gamma-2}{E_\gamma (j,x)}{\ }(e_{4j-1})_x-\frac{\|x\|^2+1}{E_\gamma (j,x)}{\ }(e_{4j+1})_x, $$
$$(\varphi e_{4j+1})_x =-\epsilon \frac{\|x\|^2+1}{E_\gamma (j,x)}{\ }(e_{4j})_x+\epsilon \frac{\|x\|^2+j+\gamma-2}{E_\gamma (j,x)}{\ }(e_{4j+2})_x,$$
$$(\varphi e_{4j+2})_x=\epsilon \frac{\|x\|^2+1}{E_\gamma (j,x)}{\ }(e_{4j-1})_x+\frac{\|x\|^2+j+\gamma-2}{E_\gamma (j,x)}{\ }(e_{4j+1})_x$$
for $j\in \{ 1,\ldots,k\}$ and $x\in \overline{M}$.
Then, with the Riemannian metric $g$ given by $g(e_{i},e_{j})=\delta _{ij}$, $i, j\in \{1,\ldots, 4k+2\}$, 
$(\overline{M},\varphi,g)$ is an almost Hermitian manifold for $\epsilon =-1$ and an almost product Riemannian manifold for $\epsilon =1$.

We define the submanifold $M$ of $\overline{M}$ by
$$
M:=\{(x_{1},\ldots, x_{4k+2})\in \mathbb{R}^{4k+2} \ | \ x_{4j+1}=x_{4j+2}=0,\ j= \overline{1,k}\}.$$
We will consider the distributions: 
$$
D_0=\langle e_{1},e_{2}\rangle,\quad D_{j}=\langle e_{4j-1},e_{4j}\rangle, \ j= \overline{1,k}\,.
$$

Then, $M$ is a generic and a pointwise $k$-slant submanifold of $\overline{M}$ for $\gamma>1$, and it is a pointwise $k$-slant but not a generic submanifold of $\overline{M}$ for $\gamma=1$. The corresponding pointwise $k$-slant distribution is $TM=\oplus_{i=0}^kD_i$, with $D_0$ the invariant component and the $D_j$'s, $j=\overline{1,k}$, the pointwise slant components, having the slant functions 
$$\theta _{j}(x)=\arccos \left(\frac{\|x\|^2+j+\gamma-2}{E_\gamma (j,x)}\right),\; x\in M,\ j=\overline{1,k}\,.$$ 
$\oplus_{i=1}^kD_i$ is the proper pointwise $k$-slant distribution associated to $M$.

We consider the distributions $L_i:=\langle e_{4i+1}, e_{4i+2}\rangle$ in $(TM)^{\perp}$, $i=\overline{1,k}$, and notice that $\oplus_{i=1}^kL_i$ is the dual pointwise $k$-slant distribution of $\oplus_{i=1}^kD_i$. We have $D_i= f(L_i)$, $i=\overline {1,k}$; hence, the dual pointwise $k$-slant distribution of $\oplus_{i=1}^kL_i$ is $\oplus_{i=1}^kD_i$. 
\end{example}

\begin{remark}\label{p214}
Any generic submanifold is a pointwise $k$-slant submanifold, but the converse is not true, in any of the considered settings: almost Hermitian, almost product Riemannian, almost contact metric, or almost paracontact metric setting. 
\end{remark}

\begin{proposition}\label{p213}
The pointwise $k$-slant concept is more general than the generic concept, in any of the considered settings. 
\end{proposition}

Corresponding to Theorem \ref{p178}, from the almost contact metric and almost paracontact metric settings, and to Theorem \ref{p179}, from the almost Hermitian and almost product Riemannian settings, which relate to \mbox{$k$-point}\-wise slant distributions, for the characterization of the pointwise $k$-slant distributions in any of these settings, we have the following result. 

\begin{theorem}\label{p182}
Let $\mathfrak D$ be a non-null distribution on $M$ decomposable into an orthogonal sum of regular distributions, $\mathfrak D=\oplus_{i=0}^k{\mathfrak D_i}$ with $\mathfrak D_i\neq \{0\}$ for $i=\nolinebreak\overline{1,k}$ and $\mathfrak D_0$ invariant (possible null). Additionally, in the almost \mbox{($\epsilon$)-con}\-tact metric setting, we consider that $\mathfrak D\perp \xi$. Let $pr_i$ denote the projection operator onto $\mathfrak D_i$ for $i=\overline{0,k}$, $f$ the component of $\varphi$ into $\mathfrak D$, and $\theta_0=0$. If $f(\mathfrak D_i)\subseteq \mathfrak D_i$ for $i=\overline{1,k}$, then the following assertions are equivalent: \\ 
\hspace*{7pt} (a) There exist $k$ pointwise distinct continuous functions $\theta_i:\nolinebreak {M}\rightarrow\nolinebreak (0,\frac{\pi}{2}]$, $i=\nolinebreak\overline{1,k}$, such that  
\begin{equation}\nonumber
f^2X= \epsilon \sum_{i=0}^k\cos^2\theta_i\cdot pr_iX\, \text{ for any } X\in \mathfrak D;
\end{equation}
\hspace*{7pt} (b) $\mathfrak D$ is a pointwise $k$-slant distribution with slant functions $\theta_i$ corresponding to $\mathfrak D_i$, $i=\overline{1,k}$. 
\end{theorem}

Let $D=\oplus_{i=0}^k{D_i}$ be a pointwise $k$-slant distribution on ${M}$, where, additionally, $D\perp \xi$ if we consider the setting given by the second formula in (\ref{97}). Let us consider the already established notations for $\theta_i$, $\lambda_i$ ($i=\overline{1,k}$), $f$, and $w$, and let $\overline\nabla$ be the Levi-Civita connection on $\overline{M}$.  

Taking into account that, for any $x\in M$ and $i\in \{1,\dots ,k\}$, the linear space $(D_i)_x$ is the entire eigenspace in $D_x$ of the eigenvalue $\lambda_i(x)$ (different from $1$ or $(-1)$), we obtain new variants for some of the mentioned results. 

\begin{proposition}\label{p156}
Let $i_0\in \{1,\ldots,k\}$ and\, $\overline\nabla_X  Y\in D$ for any $X,Y\in D_{i_0}$. Then, the following two assertions are equivalent: \\ 
1) $(\overline\nabla_X f^2) Y=0$ for any $X,Y\in D_{i_0}$. \\ 
2) i) $D_{i_0}$ is completely integrable with respect to $\overline\nabla$; \\ 
\hspace*{7 pt} ii) $X(\lambda_{i_0})=0$ for any $X\in D_{i_0}$.
\end{proposition}

\begin{corollary}\label{p157}
Let\, $\overline\nabla_X  Y\in D$ for any $X,Y\in D_{i}$, $i=\overline {1,k}$. 
Then, the following two assertions are equivalent: \\ 
1) $(\overline\nabla_X f^2) Y=0$ for any $X,Y\in D_{i}$, $i=\overline {1,k}$. \\ 
2) i) $D_{i}$ is completely integrable with respect to $\overline\nabla$ for any $i\in \{1,\ldots,k\}$; \\ 
\hspace*{7 pt} ii) $X(\lambda_{i})=0$ for any $X\in D_{i}$, $i=\overline {1,k}$.
\end{corollary}

\begin{remark}\label{p158}
The equivalence in the above Corollary is, in particular, valid for $D$ completely integrable with respect to $\overline\nabla$. 
\end{remark}

\begin{proposition}\label{p159}
Let $i_0\in \{1,\ldots,k\}$ and\, $\overline\nabla_X  Y\in D$ for any $X\in T{M}$ and\, $Y\in D_{i_0}$. Then, the following two assertions are equivalent: \\ 
1) $(\overline\nabla_X f^2) Y=0$ for any $X\in T{M}$ and\, $Y\in D_{i_0}$. \\ 
2) i) $\overline\nabla$ restricts to $D_{i_0}$ \textup(i.e., $\overline\nabla_X  Y\in D_{i_0}$ for any $X\in T{M}$ and\, $Y\in D_{i_0}$\textup); \\ 
\hspace*{7 pt} ii) the restriction of $D_{i_0}$ to any connected component of ${M}$ is a slant distribution.
\end{proposition}

\begin{theorem}\label{p160}
Let\, $\overline\nabla_X  Y\in D$ for any $X\in T{M}$ and\, $Y\in \oplus_{i=1}^k D_i$. Then, the following two assertions are equivalent: \\ 
1) $(\overline\nabla_X f^2) Y=0$ for any $X\in T{M}$ and\, $Y\in \oplus_{i=1}^k D_i$. \\ 
2) i) $\overline\nabla$ restricts to $D_{i}$ for $i=\overline{1,k}$; \\ 
\hspace*{7 pt} ii) the restriction of $D$ to any connected component of ${M}$ is a \mbox{$k$-slant} distribution. In particular, if ${M}$ is connected, $D$ is a $k$-slant distribution. 
\end{theorem}

\begin{proposition}\label{p161}
Let $i_0\in \{1,\ldots,k\}$ and\, $\overline\nabla_X  Y\in D$ for any $X\in D$ and\, $Y\in D_{i_0}$. Then, the following two assertions are equivalent: \\ 
1) $(\overline\nabla_X f^2) Y=0$ for any $X\in D$ and\, $Y\in D_{i_0}$. \\ 
2) i) $\overline\nabla_X  Y\in D_{i_0}$ for any $X\in D$ and\, $Y\in D_{i_0}$; \\ 
\hspace*{7 pt} ii) $X(\lambda_{i_0})=0$ for any $X\in D$.
\end{proposition}

\begin{corollary}\label{p162}
Let\, $\overline\nabla_X  Y\in D$ for any $X\in D$ and\, $Y\in \oplus_{i=1}^k D_i$. Then, the following two assertions are equivalent: \\ 
1) $(\overline\nabla_X f^2) Y=0$ for any $X\in D$ and\, $Y\in \oplus_{i=1}^k D_i$. \\ 
2) i) $\overline\nabla_X  Y\in D_{i}$ for any $X\in D$ and\, $Y\in D_{i}$, $i=\overline{1,k}$; \\ 
\hspace*{7 pt} ii) $X(\lambda_{i})=0$ for any $X\in D$, $i=\overline{1,k}$. 
\end{corollary}

\begin{theorem}\label{p191}
Let $D$ be completely integrable with respect to $\overline\nabla$. If $M'$ is a connected submanifold of ${M}$ such that $TM'=D$, then the following two assertions are equivalent: \\ 
1) $(\overline\nabla_X f^2) Y=0$ for any $X\in D$ and\, $Y\in \oplus_{i=1}^k D_i$. \\ 
2) i) $\overline\nabla_X  Y\in D_{i}$ for any $X\in D$ and\, $Y\in D_{i}$, $i=\overline{1,k}$; \\ 
\hspace*{7 pt} ii) $M'$ is a $k$-slant submanifold of ${M}$. 
\end{theorem}

Let $M$ be an immersed submanifold of $\overline{M}$, and let $f$ be the tangential com\-po\-nent of $\varphi$. An equivalent formulation of the definition of a pointwise \mbox{$k$-slant} submanifold is the following one. 

\begin{definition}\label{p148}
We say that $M$ is a \textit{pointwise $k$-slant submanifold} of $\overline{M}$ if there exists an orthogonal decomposition of $TM$ into regular distributions, 
$$TM=\oplus_{i=0}^k{D_i}$$
with $D_i\neq \{0\},\ i= \overline{1,k}$, and $D_0$ possible null, and there exist $k$ pointwise distinct continuous functions $\theta_i: M\rightarrow (0,\frac{\pi}{2}]$ ($\theta_i(x)\neq \theta_j(x)$ for any $i\neq j$ and any point $x \in M$), $i=\overline{1,k}$, such that: \\ 
\hspace*{7pt} (i) \;For any $x\in {M}$, $i\in \{1,\ldots,k\}$, and $v\in (D_i)_x\verb=\=\{0\}$, we have $\varphi v \neq 0$ and $ \widehat{(\varphi v, (D_i)_x)}=\theta_i(x)$;\\ 
\hspace*{7pt} (ii) \,$\varphi v\in (D_0)_x$ for any  $x\in M$ and $v\in (D_0)_x$;\\ 
\hspace*{7pt} (iii) $fv\in (D_i)_x$ for any $x\in M$ and $v\in (D_i)_x$, $i=\overline{1,k}$. 
\end{definition}

\begin{remark}\label{p149}\ \\ 
(a) Condition (i) can be replaced by \\ 
\hspace*{7pt} (i') For any $x\in {M}$, $i\in \{1,\ldots,k\}$, and $v\in (D_i)_x\verb=\=\{0\}$, we have $\varphi v \neq 0$ and 
$ \widehat{(\varphi v, T_xM)}=\theta_i(x)$.\\ 
(b) Condition (iii) can be replaced by \\ 
\hspace*{7pt} (iii') $\varphi(D_i)\perp D_j$ for any $i\neq j$ from $\{1,\ldots,k\}$.
\end{remark} 

\begin{remark}\label{p82}  
As particular cases of pointwise $k$-slant submanifolds, we mention the following situations. 

If $k=1$ and $D_0=\{0\}$, $M$ is a \textit{pointwise slant submanifold}. If $k=1$, $D_0\neq \{0\}$, and $\theta_1$ is different from the constant function $\frac{\pi}{2}$, $M$ is a \textit{pointwise semi-slant submanifold}. If $k=2$ and $D_0=\{0\}$, $M$ is a \textit{pointwise bi-slant submanifold}; it is a \textit{pointwise hemi-slant submanifold} if one of the slant functions is constant on ${M}$, equal to $\frac{\pi}{2}$. 
\end{remark}

\begin{remark}\label{p140}
Any pointwise $k$-slant submanifold is a $k$-pointwise slant submanifold, but the converse is not true, in any of the considered settings (as it was shown for distributions in the proof of Proposition \ref{p155}). 

Accordingly, all the results obtained for $k$-pointwise slant submanifolds are also valid for pointwise $k$-slant submanifolds. 
\end{remark}

\begin{remark}\label{p183}
Theorem \ref{p182} provides a necessary and sufficient condition for a submanifold $M$ of $\overline M$ to be a pointwise $k$-slant submanifold, considering $\mathfrak D=\oplus_{i=0}^k{\mathfrak D_i}$ for $TM=\oplus_{i=0}^k{\mathfrak D_i}$ in the almost Hermitian or almost product Riemannian setting, respectively for $TM=\oplus_{i=0}^k{\mathfrak D_i}\oplus\langle\xi\rangle$ in the almost contact metric or almost paracontact metric setting. 
\end{remark}

Let $M$ be a pointwise $k$-slant submanifold of $\overline{M}$, $\nabla$ be the Levi-Civita connection induced by $\overline{\nabla}$ on $M$, and  $D:=\oplus_{i=0}^k D_i$ for $TM=\oplus_{i=0}^k D_i$ in the setting given by the first formula in (\ref{97}) or for $TM=\oplus_{i=0}^k D_i\oplus\langle\xi\rangle$ in the case of the second formula in (\ref{97}). For each point $x\in{M}$, the linear spaces $(D_i)_x$ ($i=\overline{1,k}$) are the entire eigenspaces in $D_x$ of the eigenvalues $\lambda_i(x)$ ($i=\overline{1,k}$) different from $1$ or $(-1)$ of $(f^2|_D)_x$, respectively. 

Relative to $\nabla$, we have the following results.  

\begin{proposition}\label{p163}
For\, $i_0\in \{1,\ldots,k\}$, the following two assertions are equivalent: \\ 
1) $(\nabla_X f^2) Y=0$ for any $X,Y\in D_{i_0}$. \\ 
2) i) $D_{i_0}$ is completely integrable with respect to $\nabla$; \\ 
\hspace*{7 pt} ii) $X(\lambda_{i_0})=0$ for any $X\in D_{i_0}$.
\end{proposition}

\begin{corollary}\label{p164}
The following two assertions are equivalent: \\ 
1) $(\nabla_X f^2) Y=0$ for any $X,Y\in D_{i}$, $i=\overline {1,k}$. \\ 
2) i) $D_{i}$ is completely integrable with respect to $\nabla$ for $i=\overline {1,k}$; \\ 
\hspace*{7 pt} ii) $X(\lambda_{i})=0$ for any $X\in D_{i}$, $i=\overline {1,k}$.
\end{corollary}

\begin{proposition}\label{p165}
For\, $i_0\in \{1,\ldots,k\}$, the following two assertions are equivalent: \\ 
1) $(\nabla_X f^2) Y=0$ for any $X\in TM$ and\, $Y\in D_{i_0}$. \\ 
2) i) $\nabla$ restricts to $D_{i_0}$ \textup(i.e., $\nabla_X  Y\in D_{i_0}$ for any $X\in TM$ and\, $Y\in D_{i_0}$\textup); \\ 
\hspace*{7 pt} ii) the restriction of $D_{i_0}$ to any connected component of $M$ is a slant distribution. 
\end{proposition}

\begin{theorem}\label{p166}
The following two assertions are equivalent: \\ 	
1) $(\nabla_X f^2) Y=0$ for any $X\in TM$ and\, $Y\in \oplus_{i=1}^k D_i$. \\ 
2) i) $\nabla$ restricts to $D_{i}$ for any $i \in \{1,\ldots,k\}$; \\ 
\hspace*{7 pt} ii)  any open connected component of $M$ is a $k$-slant submanifold of $\overline{M}$ \textup($\lambda_{1}, \lambda_{2},\ldots, \lambda_{k}$ are constant and different on any connected component of $M$\textup). 
\end{theorem}

\begin{remark}\label{152}
We notice that, in view of conditions (\ref{1}) and (\ref{97}), all the results of this section are valid in any of the settings considered in the paper: almost Hermitian, almost product Riemannian, almost contact metric, or almost paracontact metric setting. 
\end{remark}

%%%%%%%%%%%%%%%%%%%%%%%%%%%
\section*{Final remarks}

The present paper introduces some general frameworks, that of $k$-(pointwise) slant distribution and correspondingly that of $k$-(pointwise) slant submanifold, which cover the variants of slant and pointwise slant submanifolds studied in the literature by now. In these new frameworks, we establish not only properties corresponding to some of those known for the particular variants of (pointwise) slant submanifolds but also new types of results. 

It has to be mentioned that the skew CR and generic submanifolds in sense of Ronsse are $k$-slant and $k$-pointwise slant submanifolds, respectively (see Propositions \ref{p123}, \ref{p124}, \ref{p125}, \ref{p126}, \ref{p127}), but the $k$-pointwise slant concept is more general than the generic one, as follows from Examples \ref{ex4} and \ref{ex5}. 
Actually, the generic submanifolds are pointwise $k$-slant submanifolds (see Proposition \ref{p208}), the pointwise $k$-slant concept being more general than the generic one, as follows from Examples \ref{ex8} and \ref{ex9}. Moreover, the pointwise \mbox{$k$-slant} submanifolds are particular cases of $k$-pointwise slant submanifolds, as it was shown in Proposition \ref{p155} and Examples \ref{ex6} and \ref{ex7}. 

The work reveals the possibility of a new and general treatment of the problems regarding the slant and pointwise slant phenomenon in different contexts. Moreover, it seems that our approach, which starts directly from the $k$-(pointwise) slant concepts, can be at least as fruitful as the classical one.

%%%%%%%%%%%%%%%%%%%%%%%%%
%%%%%%%%%%%%%%%%%%%%%%%%%%%%%%%

\end{document}